\documentclass[11pt]{article}
\usepackage{amssymb}
\usepackage{amsmath}
\usepackage{mathrsfs}
\usepackage{graphics}
\usepackage{graphicx}
\usepackage{xcolor}
\usepackage{subfigure}
\usepackage[T1]{fontenc}
\usepackage{latexsym,amssymb,amsmath,amsfonts,amsthm}\usepackage{txfonts}
\newcommand{\be}{\begin{eqnarray}}
\newcommand{\ben}{\begin{eqnarray*}}
\newcommand{\en}{\end{eqnarray}}
\newcommand{\enn}{\end{eqnarray*}}

\newtheorem{theorem}{Theorem}[section]
\newtheorem{lemma}{Lemma}[section]
\newtheorem{prp}[theorem]{Proposition}
\newtheorem{thm}[theorem]{Theorem}

\newtheorem{dfn}{Definition}[section]

\begin{document}
\renewcommand{\theequation}{\arabic{section}.\arabic{equation}}
\begin{titlepage}
\title{\bf 3D tamed Navier-Stokes equations driven by multiplicative L\'{e}vy noise: Existence, uniqueness and large deviations
}
\author{Zhao Dong$^{1}$, \quad Rangrang Zhang$^{2,}$\thanks{Corresponding author.}\\
{\small $^1$ RCSDS, Academy of Mathematics and Systems Science, Chinese Academy of Sciences, Beijing 100190, China.}\\
{\small $^2$ School of  Mathematics and Statistics,
Beijing Institute of Technology, Beijing, 100081, China.}\\
({\sf dzhao@amt.ac.cn}, {\sf rrzhang@amss.ac.cn} )}
\date{}
\end{titlepage}
\maketitle

\noindent\textbf{Abstract}:
In this paper, we show the existence and uniqueness of a strong solution to stochastic 3D tamed Navier-Stokes equations driven by multiplicative L\'{e}vy noise with periodic boundary conditions. Then we establish the large deviation principles of the strong solution on the state space $\mathcal{D}([0,T];\mathbb{H}^1)$, where the weak convergence approach plays a key role.



\section{Introduction}
This paper concerns Navier-Stokes equations, which describes the time evolution of an incompressible fluid.
Mathematically or statistically, such physical laws should incorporate with noise influences, due to the lack of knowledge of certain physical parameters as well as bias or incomplete measurements arising in experiments or modeling. Fix any $T>0$ and let $(\Omega,\mathcal{F},\mathbb{P},\{\mathcal{F}_t\}_{t\in
[0,T]})$ be a stochastic basis. Without loss of generality, here the filtration $\{\mathcal{F}_t\}_{t\in [0,T]}$ is assumed to be complete and $W(t)$ is $\{\mathcal{F}_t\}_{t\in [0,T]}-$Wiener process. We use $\mathbb{E}$ to denote the expectation with respect to $\mathbb{P}$.
The Cauchy problem of the three dimensional stochastic Navier-Stokes equations driven by Gaussian random noise can be written as
\begin{eqnarray*}
\left\{
  \begin{array}{ll}
   du-\nu \Delta u dt+(u\cdot \nabla) udt+\nabla p dt=\sigma(t, u(t))dW(t), &  {\rm{in}} \ \mathbb{T}^3\times[0,\infty)\\
    div\ u(t,x)=0,  & {\rm{in}} \ \mathbb{T}^3\times[0,\infty)\\
    u(0)=u_0, &
  \end{array}
\right.
\end{eqnarray*}
where $u(t,x)=(u^1(t,x),u^2(t,x),u^3(t,x))$ represents the velocity field, $\nu$ is the viscosity constant, $p(t,x)$ denotes the pressure. $\mathbb{T}^3\subset\mathbb{R}^3$ denotes the three dimensional torus (suppose the periodic length is $1$). $\sigma$ is a measurable function, which will be specified in subsection \ref{s-2}.

 As we all know, the stochastic 2D Navier-Stokes equation has been studied extensively in the literature, however, there exist serious obstacles in dealing with the stochastic 3D Navier-Stokes equations. Up to now, the existence of martingale solutions and stationary solutions of the stochastic 3D Navier-Stokes equation was proved by Flandoli and Gatarek \cite{FG95} and Mikulevicius and Rozovskii \cite{M-R} under more general conditions. However, the uniqueness still remains open. Later, a new model called stochastic 3D tamed Navier-Stokes equations was proposed by R\"{o}ckner and Zhang in \cite{R-Z-0-0}, which is given by
\begin{eqnarray}\label{equ-81}
\left\{
  \begin{array}{ll}
   du-\nu \Delta u dt+(u\cdot \nabla) udt+\nabla p dt+g_N(|u(t)|^2)u(t)dt=\sigma(t, u(t))dW(t), &  \\
    div\ u(t,x)=0, \quad t>0, & \\
    u(0)=u_0, &
  \end{array}
\right.
\end{eqnarray}
where $g_N$ is a smooth function from $\mathbb{R}^{+}$ to $\mathbb{R}^{+}$, whose precise definition is given by (\ref{equ-83}). As stated in \cite{R-Z-0-0}, the motivation to study (\ref{equ-81}) comes from the deterministic case. The important observation is that the strong solution of 3D Navier-Stokes equation coincides with the strong solution  of (\ref{equ-81}) for large enough $N$.
There are several results on the stochastic 3D tamed Navier-Stokes equations driven by Gaussian random noise. We mention two of them. In \cite{R-Z-0-0}, the authors established the existence of a unique strong solution (strong in the probabilistic sense and weak in the PDE sense) to equation (\ref{equ-81}) indirectly by employing the Yamada-Watanabe Theorem, i.e.,  proving the existence of martingale solutions and pathwise uniqueness. Then, they also studied the Feller property and invariant measures for the corresponding semigroup generated by the strong solution.  Using a direct approach, R\"{o}ckner and Zhang \cite{R-Z} established the existence and uniqueness of strong solutions to (\ref{equ-81}). Moreover, they proved small time large deviation principles for the stochastic 3D tamed Navier-Stokes equations. For more information on this model, we refer the reader to \cite{R-Z-Z, R-Z-0} and the references therein.

\
In recent years, introducing a jump-type noises such as L\'{e}vy-type or
Poisson-type perturbations has become extremely popular for modeling natural phenomena,
because these noises are good choice to reproduce the performance of some natural
phenomena in real world models, such as some large moves and unpredictable events.
There
is a large amount of literature on the existence and uniqueness of solutions to stochastic
partial differential equations (SPDEs) driven by jump-type noises. For example, Brze\'{z}niak et al. \cite{BHZ} studied the existence and uniqueness of the solution to an abstract
nonlinear equation driven by multiplicative L\'{e}vy noise. Their results can cover some types of SPDEs, such as the stochastic 2D Navier-Stokes Equations,
the 2D stochastic Magneto-Hydrodynamic Equations, the 2D stochastic Boussinesq Model
for the B\'{e}nard Convection, the 2D stochastic Magnetic B\'{e}rnard Problem and several stochastic Shell Models of
turbulence (the readers can refer to \cite{BLZ}). However, there are still a plenty of important models that do not satisfy the conditions required by \cite{BHZ}.
Recently, S. Shang et al. \cite{S-Z-Z} considered a stochastic model of incompressible non-Newtonian fluids of second grade on a bounded domain of $\mathbb{R}^2$ driven by L\'evy noise. Applying the variational approach, the authors established the global existence and uniqueness of strong probabilistic solution.
As far as we know, there are no results on the stochastic 3D tamed Navier-Stokes equations driven by multiplicative L\'{e}vy noise, which can be written as
\begin{eqnarray}\label{equ-80}
\left\{
  \begin{array}{ll}
   du-\nu \Delta u dt+(u\cdot \nabla) udt+\nabla pdt+g_N(|u|^2(t))u(t)dt=\int_{\mathbb{Z}} \sigma(t-, u(t-),z)\tilde{\eta}(dt, dz), &  \\
    div\ u(t,x)=0, \quad t>0, & \\
    u(0)=u_0\in \mathbb{H}^1, &
  \end{array}
\right.
\end{eqnarray}
where $\tilde{\eta}$ is the compensated time homogeneous Poisson random measure on a certain locally compact Polish space $(\mathbb{Z}, \mathcal{B}(\mathbb{Z}))$. On the other hand, the large deviation principle for stochastic partial differential equations (SPDE) driven by L\'{e}vy noise attracts a lot of interests from mathematical community. R\"{o}ckner and Zhang \cite{R-TS} established
large deviations for SPDEs driven by an additive jump noise. The case of multiplicative L\'{e}vy noise was studied by \'{S}wiech and Zabczyk \cite{S-Z}
and Budhiraja, Chen and Dupuis \cite{B-C-D} where the large deviation was obtained on a larger space (hence, with a weaker topology) than the actual state space of the solution. Yang, Zhai
and Zhang \cite{Y-Z-Z} obtained the large deviation principles on the actual state space of stochastic
evolution equations with regular coefficients driven by multiplicative L\'{e}vy noise. Later, Zhai and Zhang \cite{Z-Z} proved the large deviations for 2D stochastic
Navier-Stokes equations driven by
multiplicative L\'{e}vy noises on the space $\mathcal{D}([0,T];\mathbb{H})$, the space of $\mathbb{H}$-valued right continuous functions with left limits on $[0,T ]$.

The purpose of this paper is two-fold. The first part is to show the existence and uniqueness of strong solution to (\ref{equ-80}) based on Galerkin's approximation and a kind of local monotonicity of the coefficients. Concretely, we prove the result via three steps: we firstly make some non-trivial
a priori estimates of the Galerkin's approximation, then we show that the limit of those
approximate solutions solves the original equation by applying the monotonicity arguments,
finally we prove the uniqueness of solutions.
The second part is to prove the small perturbation large deviation principle (LDP) for the stochastic 3D tamed Navier-Stokes equations driven by multiplicative L\'{e}vy noise on the space $\mathcal{D}([0,T];\mathbb{H}^1)$, which provides the exponential decay  of small probabilities associated  with
the corresponding stochastic dynamical systems with small noise. The proof of the large deviations will be based on the weak convergence approach initiated by Budhiraja, Chen and Dupuis \cite{B-C-D} and Budhiraja, Dupuis and Maroulas \cite{B-D-M}. As an important part of the proof, we need to obtain global well-posedness of the so-called skeleton equation by using similar method as the first part. To complete the proof of the large deviation principle, we also need to show the weak convergence of the perturbations of the
system (\ref{equ-80}) to the skeleton equation. During the proof process,
we firstly need to establish the
tightness of the solutions of the perturbations of the
system (\ref{equ-80}) in a larger space $\mathcal{D}([0,T];D(A^{-\alpha}))$ with $\alpha\geq 0$,
then with the aid of the Skorohod representation theorem we obtain the weak convergence
actually takes place in the space $\mathcal{D}([0,T];\mathbb{H}^1)$.
\

Our paper is organized as follows. The mathematical formulation of stochastic 3D  tamed Navier-Stokes equations and some useful nonlinear term estimates are presented in Section 2. In Section 3, we prove the existence and uniqueness of strong solution to the  stochastic 3D  tamed Navier-Stokes equations. The weak convergence method and the statement of the main result are introduced in Section 4. Then the skeleton equation is studied in Section 5. At last, the large deviation principle is proved in Section 6.

Throughout this paper, $C$ is a positive constant whose value may be different
from line to line.

\section{Formulations}
Let $u(t,x)=(u^1(t,x),u^2(t,x),u^3(t,x))$ be a vector function
on $\mathbb{T}^3$. The following notations will be used:
 $|u|^2=\sum^3_{i=1}|u^i|^2$, $div \ u:=\sum^3_{i=1}\partial_i u^i$ and $(u\cdot \nabla)u=\sum^3_{i=1}u^i\partial_i u$. Throughout the paper, $g_N(\cdot)$ will denote a fixed smooth function from $\mathbb{R}^+$ to $\mathbb{R}^+$ such that for some $N>0$,
\begin{eqnarray}\label{equ-83}
\left\{
  \begin{array}{ll}
    g_N(r)=0, & {\rm{if}}\ r\leq N, \\
    g_N(r)=\frac{r-N}{\nu}, & {\rm{if}}\ r\geq N+1, \\
    0\leq g'_N(r)\leq C, & {\rm{if}}\ r\geq 0.
  \end{array}
\right.
\end{eqnarray}
Without loss of generality, we assume the viscosity coefficient $\nu=1$.
\subsection{Functional spaces}\label{ss-2}
Let $\mathcal{L}(K_1;K_2)$ (resp. $\mathcal{L}_2(K_1;K_2)$) be the space of bounded (resp. Hilbert-Schmidt) linear operators from the Hilbert space $K_1$ to $K_2$, whose norm is denoted by $\|\cdot\|_{\mathcal{L}(K_1;K_2)}(\|\cdot\|_{\mathcal{L}_2(K_1;K_2)})$.
For a topological space $\mathcal{E}$, denote the corresponding Borel $\sigma-$field by $\mathcal{B}(\mathcal{E})$. For a metric space $\mathbb{X}$, $C([0,T];\mathbb{X})$ stands for the space of continuous functions from $[0,T]$ into $\mathbb{X}$ and $\mathcal{D}([0,T];\mathbb{X})$ represents the space of right continuous functions with left limits from $[0,T]$ into $\mathbb{X}$. For a metric space $\mathbb{Y}$, denote by $M_b(\mathbb{Y}), C_b(\mathbb{Y})$ the space of real valued bounded $\mathcal{\mathbb{Y}}/ \mathcal{\mathbb{R}}-$measurable maps and real valued bounded continuous functions, respectively.

Let $C^{\infty}(\mathbb{T}^3)=C^{\infty}(\mathbb{T}^3; \mathbb{R}^3)$ denote the set of all smooth periodic functions from $\mathbb{R}^3$ to $\mathbb{R}^3$. For $p\geq 1$,  $L^p(\mathbb{T}^3)=L^p(\mathbb{T}^3;\mathbb{R}^3)$ stands for the vector-valued $L^p$ space in which the norm is denoted by $\|\cdot\|_{L^p}$. For a non-negative integer $m\geq 0$, let $H^m$ be the usual Sobolev space on $\mathbb{T}^3$ with values in $\mathbb{R}^3$, i.e., the closure of $C^{\infty}(\mathbb{T}^3)$ with respect to the norm:
\[
\|u\|^2_{H^m}=\int_{\mathbb{T}^3} |(I-\Delta)^{\frac{m}{2}}u|^2dx,
\]
where
\[
(I-\Delta)^{\frac{m}{2}} u:=\Big(( I-\Delta)^{\frac{m}{2}} u^1, (I-\Delta)^{\frac{m}{2}} u^2,( I-\Delta)^{\frac{m}{2}} u^3\Big),
\]
is defined by Fourier transformation.

For $m\in \mathbb{N}_0$, set
\[
\mathbb{H}^m:=\{u\in H^{m}: div\ u=0\}.
\]
Then the norm of $H^m$ restricted to $\mathbb{H}^m$ will be denoted by $\|\cdot\|_{\mathbb{H}^m}$. In particular, $\mathbb{H}^0$ is a closed linear subspace of the Hilbert space $L^2(\mathbb{T}^3)=H^0$. Let $\mathcal{P}$ be the orthogonal projection from $L^2(\mathbb{T}^3)$ to $\mathbb{H}^0$.
It is well-known that $\mathcal{P}$ commutes with the derivative operators and that $\mathcal{P}$ can be restricted to a bounded linear operator from $H^m$ to $\mathbb{H}^m$ (see \cite{Lions}).
For any $u\in \mathbb{H}^0$ and $v\in L^2(\mathbb{T}^3)$, we have
\[
\langle u, v\rangle_{\mathbb{H}^0}:=\langle u, \mathcal{P}v\rangle_{\mathbb{H}^0}=\langle u, v\rangle_{L^2}.
\]
Moreover, for $u\in \mathbb{H}^0$ and $v\in \mathbb{H}^2$, the inner product $\langle u,v\rangle_{\mathbb{H}^1}$ is taken
in the generalized sense, i.e.,
\[
\langle u, v\rangle_{\mathbb{H}^1}:=\langle u, (I-\Delta)v\rangle_{\mathbb{H}^0}.
\]
For any $u\in H^2\cap\mathbb{H}^1$, define
\begin{eqnarray}\label{equ-3}
A u:=-\mathcal{P}\Delta u.
\end{eqnarray}
It is well-known that the Stokes operator $A$ is a positive self-adjoint operator in $\mathbb{H}^0$ with a compact resolvent. Let $\{e_i\}^{\infty}_{i=1}\subset \mathbb{H}^2$ be an orthonormal basis of $\mathbb{H}^0$ composed of eigenfunctions of $A$ with corresponding eigenvalues $0< \lambda_1\leq \lambda_2\leq\cdot\cdot\cdot\rightarrow \infty$ satisfies $Ae_i=\lambda_i e_i$.
 We will use fractional powers of the operator $A$, denoted by $A^{\alpha}$, as well as their domains $D(A^{\alpha})$ for $\alpha\in \mathbb{R}$. Note that
\[
D(A^{\alpha})=\Big\{u=\sum^{\infty}_{i=1}u_i\cdot e_i:\sum^{\infty}_{i=1}\lambda_i^{2\alpha}u^2_i<\infty\Big\}.
\]
We may endow $D(A^{\alpha})$ with the inner product
\[
(u,v)_{D(A^{\alpha})}=\langle A^{\alpha} u,A^{\alpha} v\rangle_{\mathbb{H}^0}\quad {\rm{for}}\ u,v\in D(A^{\alpha}).
\]
Hence, $(D(A^{\alpha}), (\cdot,\cdot)_{D(A^{\alpha})})$ is a Hilbert space and $\{\lambda^{-\alpha}_i e_i\}_{i\in\mathbb{N}}$ is a complete orthonormal system of $D(A^{\alpha})$. By Riesz representative theorem, $D(A^{-\alpha})$ is the dual space of $D(A^{\alpha})$.

For any $u,v\in \mathbb{H}^1$, set
\begin{eqnarray*}
B(u,v):=\mathcal{P}((u\cdot\nabla)v).
\end{eqnarray*}
If $u=v$, we write $B(u)=B(u,u)$. By the incompressible condition, it gives that $\langle B(u,v),v\rangle_{\mathbb{H}^0}=0$.

Letting the operator $\mathcal{P}$ act on both sides of (\ref{equ-80}), we get
\begin{eqnarray}\label{equ-4}
\left\{
  \begin{array}{ll}
   du+Au dt+B(u)dt+\mathcal{P}g_N(|u|^2(t))u(t)dt=\int_{\mathbb{Z}} \sigma(t-, u(t-),z)\tilde{\eta}(dt,dz), &  \\
    div\ u(t,x)=0, \quad t>0, & \\
    u(0)=u_0\in \mathbb{H}^1. &
  \end{array}
\right.
\end{eqnarray}


\subsection{Poisson random measure and Hypotheses}\label{s-2}
Let $(\mathbb{Z},\mathcal{B}(\mathbb{Z}))$ be a locally compact Polish space and let $\vartheta$ be a $\sigma-$finite positive measure on it. Suppose $(\Omega, \mathcal{F}, \mathcal{F}_t, \mathbb{P})$ is a filtered probability space with expectation $\mathbb{E}$.
Set $C_c(\mathbb{Z})$ be the space of continuous functions with compact supports. Denote
\[
\mathcal{M}_{FC}(\mathbb{Z}):=\Big\{{\rm{measure}}\ \vartheta\  {\rm{on}}\ (\mathbb{Z}, \mathcal{B}(\mathbb{Z}))\ {\rm{such\ that}}\ \vartheta(K)<\infty\  {\rm{for\ every\ compact}}\ K\ {\rm{in}}\ \mathbb{Z}\Big\}.
\]
Endow $\mathcal{M}_{FC}(\mathbb{Z})$ with the weakest topology such that for every $f\in C_c(\mathbb{Z})$, the function $\vartheta\rightarrow \langle f, \vartheta\rangle=\int_{\mathbb{Z}}f(u)d\vartheta(u), \vartheta\in \mathcal{M}_{FC}(\mathbb{Z})$ is continuous. This topology can be metrized such that $\mathcal{M}_{FC}(\mathbb{Z})$ is a Polish space (see \cite{B-D-M}).
Let $T>0$, set $\mathbb{Z}_T=[0,T]\times \mathbb{Z}$. Fix a measure $\vartheta\in \mathcal{M}_{FC}(\mathbb{Z})$ and let $\vartheta_T=\lambda_T\otimes \vartheta$, where $\lambda_T$ is Lebesgue measure on $[0,T]$. We recall the definition of Poisson random measure from \cite{I-W} that
\begin{dfn}
We call measure $\eta$ a Poisson random measure on $\mathbb{Z}_T$ with intensity measure $\vartheta_T$ is a $\mathcal{M}_{FC}(\mathbb{Z})-$valued random variable such that
 \begin{description}
   \item[(1)] for each $B\in \mathcal{B}(\mathbb{Z}_T)$ with $\vartheta_T(B)<\infty$, $\eta(B)$ is a Poisson distribution with mean $\vartheta_T(B)$,
   \item[(2)] for disjoint $B_1,\cdot\cdot\cdot, B_k\in \mathcal{B}(\mathbb{Z}_T)$, $\eta(B_1),\cdot\cdot\cdot, \eta(B_k) $ are mutually independent random variables.
 \end{description}

\end{dfn}
We will denote by
$\tilde{\eta}=\eta-\vartheta_T$ the compensated time homogeneous Poisson random measure associated to $\eta$.
Assume $(H, |\cdot|_H)$ is a Hilbert space. Let $L^2(\Omega\times [0,T]; L^2(\mathbb{Z}, \vartheta; H))$ be the space of predictable process
$X: \mathbb{R}^+\times \mathbb{Z}\times \Omega\rightarrow H$ satisfying
\begin{eqnarray*}
\mathbb{E}\int^T_0\int_{\mathbb{Z}}|X(r,z)|^2_H\vartheta (dz)dr< \infty, \quad T>0.
\end{eqnarray*}
Then, it follows from \cite{BHZ} that for every $X\in L^2(\Omega\times [0,T]; L^2(\mathbb{Z}, \vartheta; H))$,
\begin{eqnarray}
\mathbb{E}\Big|\int^t_0\int_{\mathbb{Z}}X(r,z)\tilde{\eta}(dr,dz)\Big|^2_H=\mathbb{E}\int^t_0\int_{\mathbb{Z}}|X(r,z)|^2_H\vartheta (dz)dr,\quad t\geq0.
\end{eqnarray}

To obtain the global well-posedness of (\ref{equ-4}), we need the following hypotheses.
\begin{description}
   \item[\textbf{Hypothesis H0}]
Let $\sigma$ be a predictable mapping from $[0,T]\times \mathbb{H}^1\times \mathbb{Z}\rightarrow \mathbb{H}^1$ (resp. $[0,T]\times \mathbb{H}^0\times \mathbb{Z}\rightarrow \mathbb{H}^0$).
\begin{description}
\item[(A)] There exists a positive constant ${K}_1$ such that
\begin{eqnarray}\label{equ-5-1}
                                       \int_{\mathbb{Z}}\|\sigma(t,u,z)\|^{2}_{\mathbb{H}^0}\vartheta(dz)\leq {K}_1(1+\|u\|^2_{\mathbb{H}^0}),\quad u\in \mathbb{H}^0, \ t\in [0,T].
                                       \end{eqnarray}
And there exists a positive constant $K_2$ such that
                                       \begin{eqnarray}\label{equ-6-1}
                                       \int_{\mathbb{Z}}\|\sigma(t,u_1,z)-\sigma(t,u_2,z)\|^2_{\mathbb{H}^0}\vartheta(dz)\leq K_2\|u_1-u_2\|^2_{\mathbb{H}^0},\quad u_1,u_2\in \mathbb{H}^0, \ t\in [0,T].
         \end{eqnarray}
  \item[(B)]
There exists a positive constant $L_1$ such that
                                        \begin{eqnarray}\label{equ-5}
                                       \int_{\mathbb{Z}}\|\sigma(t,u,z)\|^{2}_{\mathbb{H}^1}\vartheta(dz)\leq L_1(1+\|u\|^2_{\mathbb{H}^1}),\quad u\in \mathbb{H}^1, \ t\in [0,T].
                                       \end{eqnarray}
There exists a positive constant $L_2$ such that
                                        \begin{eqnarray}\label{equ-5-11}
                                       \int_{\mathbb{Z}}\|\sigma(t,u,z)\|^{6}_{\mathbb{H}^1}\vartheta(dz)\leq L_2(1+\|u\|^6_{\mathbb{H}^1}),\quad u\in \mathbb{H}^1, \ t\in [0,T].
                                       \end{eqnarray}

 And there exists a positive constant $L_3$ such that
                                       \begin{eqnarray}\label{equ-6}
                                       \int_{\mathbb{Z}}\|\sigma(t,u_1,z)-\sigma(t,u_2,z)\|^2_{\mathbb{H}^1}\vartheta(dz)\leq L_3\|u_1-u_2\|^2_{\mathbb{H}^1},\quad u_1,u_2\in \mathbb{H}^1, \ t\in [0,T].
                                       \end{eqnarray}

                                 \end{description}
\end{description}
Now, we introduce the definition of a strong solution to (\ref{equ-4}).
\begin{dfn}\label{def-1}
The system (\ref{equ-4}) has a strong solution if for every stochastic basis $(\Omega,\mathcal{F},\mathbb{P}, \{\mathcal{F}_t\}_{t\geq 0})$ and a time homogeneous Poisson random measure $\tilde{\eta}$ on $(\mathbb{Z},\mathcal{B}(\mathbb{Z}))$ over the stochastic basis with intensity measure $\vartheta$, there exists a progressively measurable process $u:[0,T]\times \Omega\rightarrow \mathbb{H}^1$ with $\mathbb{P}-$a.s.
\begin{eqnarray}\label{equ-8}
u(\cdot, \omega)\in \mathcal{D}([0,T];\mathbb{H}^1)\cap L^2([0,T]; \mathbb{H}^2)
\end{eqnarray}
such that
\begin{eqnarray}\label{equ-9}
u(t)=u_0-\int^t_0A u(s)ds-\int^t_0 B( u(s)) ds-\int^t_0\mathcal{P}g_N(|u|^2)uds+ \int^t_0\int_{\mathbb{Z}}\sigma(s-,u(s-),z)\tilde{\eta}(ds,dz),
\end{eqnarray}
holds in $\mathbb{H}^0$ for all $t\geq 0$, $\mathbb{P}-$a.s..
\end{dfn}

\subsection{Some inequalities and It\^{o} formula}
Let
\[
F(u):=-Au -B(u)-\mathcal{P}(g_N(|u|^2)u).
\]
In order to prove the global well-posedness of (\ref{equ-4}), we need the following a priori estimates of nonlinear terms. Referring to Lemma 2.3 in \cite{R-Z-0-0} and the proof of Theorem 3.1 and Lemma 5.3 in \cite{R-Z}, it gives that
\begin{lemma}
\begin{description}
  \item[(1)] For $u\in \mathbb{H}^0$,
\begin{eqnarray}\label{equ-11}
\langle F(u), u\rangle_{\mathbb{H}^0}\leq C_N \|u\|^2_{\mathbb{H}^0}.
\end{eqnarray}
  \item[(2)] For $u\in \mathbb{H}^2$,
\begin{eqnarray}\label{equ-7}
\langle F(u), u\rangle_{\mathbb{H}^1}\leq -\frac{1}{2}\|u\|^2_{\mathbb{H}^2} -\frac{1}{2}\||u|\cdot |\nabla u|\|^2_{L^2}+C_N \|\nabla u\|^2_{\mathbb{H}^0}+\|u\|^2_{\mathbb{H}^0}.
\end{eqnarray}
  \item[(3)] For $u_1, u_2\in \mathbb{H}^2$, it follows that
\begin{eqnarray*}
-\langle A(u_1-u_2), u_1-u_2\rangle_{\mathbb{H}^0}&=&-\|u_1-u_2\|^2_{\mathbb{H}^1}+\|u_1-u_2\|^2_{\mathbb{H}^0},\\
-\langle B(u_1,u_1)-B(u_2,u_2), u_1-u_2\rangle_{\mathbb{H}^0}&\leq& \frac{1}{2}\|u_1-u_2\|^2_{\mathbb{H}^1}+C\|u_2\|_{\mathbb{H}^1}\|u_2\|_{\mathbb{H}^2}\|u_1-u_2\|^2_{\mathbb{H}^0},\\
-\langle g_N(|u_1|^2)u_1-g_N(|u_2|^2)u_2, u_1-u_2\rangle_{\mathbb{H}^0}&\leq& C\|u_2\|_{\mathbb{H}^1}\|u_2\|_{\mathbb{H}^2}\|u_1-u_2\|^2_{\mathbb{H}^0}.
\end{eqnarray*}
Hence, we have
\begin{eqnarray}\label{equ-18}
\langle F(u_1)-F(u_2), u_1-u_2\rangle_{\mathbb{H}^0}\leq-\frac{1}{2}\|u_1-u_2\|^2_{\mathbb{H}^1}
+C_0(\|u_2\|_{\mathbb{H}^1}\|u_2\|_{\mathbb{H}^2}+1)\|u_1-u_2\|^2_{\mathbb{H}^0}.
\end{eqnarray}
We emphasize the constant $C_0>1$, as it plays a key role in the proof of Theorem \ref{thm-1}.
  \item[(4)] For $u_1, u_2\in \mathbb{H}^2$, it gives that
\begin{eqnarray}\notag
&&\langle F(u_1)-F(u_2), u_1-u_2\rangle_{\mathbb{H}^1}\\
\label{equ-55}
&\leq&-\frac{1}{4}\|u_1-u_2\|^2_{\mathbb{H}^2}+C(1+\|u_1\|^4_{\mathbb{H}^1}+\|u_2\|^4_{\mathbb{H}^1}+\|u_2\|^2_{\mathbb{H}^2})\|u_1-u_2\|^2_{\mathbb{H}^1}.
\end{eqnarray}

\end{description}
\end{lemma}
The main tool in the present paper is the It\^{o} formula, whose proof can be found in \cite{BHZ}.
\begin{lemma}\label{lem-1}
Assume that $E$ is a Hilbert space with norm $\|\cdot\|_E$. Let $X$ be a process given by
\[
X_t=X_0+\int^t_0 a(s)ds+\int^t_0\int_{\mathbb{Z}}f(s,z)\tilde{\eta}(ds,dz), \ t\geq 0,
\]
where $a$ is an $E-$valued progressively measurable process on the space $(\mathbb{R}^{+}\times \Omega, \mathcal{B}(\mathbb{R}^{+})\times \mathcal{F})$ such that for all $t\geq 0$, $\int^t_0 \|a(s,\omega)\|_Eds<\infty, $ $\mathbb{P}-$a.s. and $f$ is a predictable process on $\mathbb{E}$ with $\mathbb{E}\int^t_0\int_{\mathbb{Z}}\|f(s,z)\|^2_E\vartheta(dz)ds<\infty$, for each $t>0$. Denote by $G$ a separable Hilbert space. Let $\phi: E\rightarrow G$ be a function of class $C^1$ such that the first derivative $\phi': E\rightarrow \mathcal{L}(E;G)$ is $(p-1)-$H\"{o}lder continuous. Then for every $t>0$, we have $\mathbb{P}-$a.s.
\begin{eqnarray*}
\phi(X_t)&=&\phi(X_0)+\int^t_0 \phi'(X_s)(a(s))ds+\int^t_0\int_{\mathbb{Z}}[\phi'(X_{s-})f(s,z)]\tilde{\eta}(ds,dz)\\
&& \ +\int^t_0\int_{\mathbb{Z}}[\phi(X_{s-}+f(s,z))-\phi(X_{s-})-\phi'(X_{s-})f(s,z)]\eta(ds, dz).
\end{eqnarray*}
\end{lemma}

\section{Existence and uniqueness }
In this part, we aim to prove the following result.
\begin{thm}\label{thm-1}
Assume Hypothesis H0 holds and the initial value $u_0\in L^2(\Omega, \mathcal{F}_0; \mathbb{H}^1)$. Then there exists a unique strong solution of (\ref{equ-4}) in the sense of Definition \ref{def-1}. Moreover, there exists a positive constant $C$ such that
\begin{eqnarray}\label{equ-35}
\mathbb{E}\Big(\sup_{0\leq t\leq T}\|u(t)\|^2_{\mathbb{H}^1}+\int^T_0 \|u(t)\|^2_{\mathbb{H}^2}dt\Big)\leq C(1+\mathbb{E}\|u_0\|^2_{\mathbb{H}^1}).
\end{eqnarray}
\end{thm}
Motivated by \cite{R-Z}, the proof process of Theorem \ref{thm-1} is based on Galerkin's approximation and a kind of local monotonicity of the 3D tamed Navier-Stokes equation.
\begin{proof}
\textbf{Step 1: } Assume $u_0\in L^6(\Omega, \mathcal{F}_0; \mathbb{H}^1)$.\\
Recall $\{e_i\}^{\infty}_{i=1}\subset \mathbb{H}^2$ be an orthonormal basis in $\mathbb{H}^0$ composed of eigenvectors of  $A$ such that $span\{e_i, i\geq 1\}$ is dense in $\mathbb{H}^1$. Moreover, it is easy to see that $\{e_i\}^{\infty}_{i=1}$ is also orthogonal in $\mathbb{H}^1$. Denote by $\Pi_n$ the orthogonal projection from $\mathbb{H}^0$ onto the finite dimensional space $\mathbb{H}_n:=span\{e_1,e_2, \cdot\cdot\cdot, e_n\}$:
\[
\Pi_n v:=\sum^n_{i=1}\langle v, e_i\rangle_{\mathbb{H}^0}e_i.
\]
Then $\Pi_n$ is also the orthogonal projection from $\mathbb{H}^1$ onto $\mathbb{H}_n$. Now, consider the following finite dimensional stochastic differential equation in $\mathbb{H}_n$
\begin{eqnarray}\label{equ-10}
\left\{
  \begin{array}{ll}
    du_n(t)=[\Pi_n F(u_n(t))]dt+\int_{\mathbb{Z}} {\sigma}_n(t-,u_n(t-),z)\tilde{\eta}(dt,dz), &  \\
    u_n(0)=\Pi_nu_0, &
  \end{array}
\right.
\end{eqnarray}
where $\sigma_n:=\Pi_n \sigma$. Taking into account (\ref{equ-5-1}) and (\ref{equ-6-1}), we know that $\sigma_n$ is globally Lipschitz. Moreover, for $u\in \mathbb{H}_n$, we deduce from (\ref{equ-11}) that
\begin{eqnarray*}
\langle \Pi_n F(u), u\rangle_{\mathbb{H}^0}\leq \langle F(u), u\rangle_{\mathbb{H}^0}\leq C_N \|u\|^2_{\mathbb{H}^0},
\end{eqnarray*}
and
\begin{eqnarray*}
\| \Pi_n F(u)-\Pi_n F(v)\|_{\mathbb{H}^0}\leq  C_N \|u-v\|^2_{\mathbb{H}^0},
\end{eqnarray*}
which implies that $F_n(u)=\Pi_n F(u)$ is of linear growth and locally Lipschitz in $\mathbb{H}^0$. Based on the above and (\ref{equ-6-1}), it follows from \cite{A-B-W} that (\ref{equ-10}) admits a unique c\`{a}dl\`{a}g local strong solution $u_n$ in $\mathbb{H}_n$. Then, by the skew symmetric of the nonlinear term $B$, the local solution can be extended to any time interval $[0,T]$, $T>0$.

In the following, we aim to prove
\begin{eqnarray}\label{equ-12}
\sup_n \mathbb{E}\Big(\sup_{t\in [0,T]}\|u_n(t)\|^{2}_{\mathbb{H}^1}+\int^T_0\|u_n(t)\|^2_{\mathbb{H}^2}dt\Big)&\leq& C(1+\mathbb{E}\|u_0\|^2_{\mathbb{H}^1}),\\
\label{equ-13}
\sup_n \mathbb{E}\Big(\sup_{t\in [0,T]}\|u_n(t)\|^{6}_{\mathbb{H}^1}+\int^T_0\|u_n(t)\|^{4}_{\mathbb{H}^1}\|u_n(t)\|^2_{\mathbb{H}^2}dt\Big)&\leq&  C(1+\mathbb{E}\|u_0\|^6_{\mathbb{H}^1}).
\end{eqnarray}
Applying It\^{o} formula (Lemma \ref{lem-1}) to the function $\varphi(x)=|x|^2$ and by $\varphi(x+y)-\varphi(x)-\langle y, \nabla \varphi(x)\rangle=\varphi(y)=|y|^2$, we deduce that for $0\leq t\leq T$,
\begin{eqnarray*}
\|u_n(t)\|^{2}_{\mathbb{H}^1}&=&\|\Pi_n u_0\|^{2}_{\mathbb{H}^1}+2\int^t_0\langle u_n(s), F_n(u_n(s))\rangle_{\mathbb{H}^1}ds+2\int^t_0\int_{\mathbb{Z}}\langle u_n(s-), {\sigma}_n(s-,u_n(s-),z)\rangle_{\mathbb{H}^1}\tilde{\eta}(ds, dz)\\
&&\ +\int^t_0\int_{\mathbb{Z}}\|{\sigma}_n(s-,u_n(s-),z)\|^2_{\mathbb{H}^1}\eta(ds, dz).
\end{eqnarray*}
Define a stopping time
\[
\tau^n_R:=\inf\{t\in [0,T]: \|u_n(t)\|^{2}_{\mathbb{H}^1}\geq R\}.
\]
With the aid of  (\ref{equ-7}), we deduce that
\begin{eqnarray*}
&&\sup_{0\leq s\leq t\wedge\tau^n_R }\|u_n(s)\|^{2}_{\mathbb{H}^1}+\int^{t\wedge\tau^n_R }_0\|u_n(s)\|^2_{\mathbb{H}^2}ds+\int^{t\wedge\tau^n_R }_0\||u_n|\cdot |\nabla u_n|\|^2_{L^2}ds\\
&\leq& \|\Pi_n u_0\|^{2}_{\mathbb{H}^1}+2C_N\int^{t\wedge\tau^n_R }_0 \|u_n(s)\|^2_{\mathbb{H}^1}ds+2\sup_{0\leq s\leq t\wedge\tau^n_R }\int^s_0\int_{\mathbb{Z}}\langle u_n(r-), {\sigma}_n(r-,u_n(r-),z)\rangle_{\mathbb{H}^1}\tilde{\eta}(dr,dz)\\
&&\ +\sup_{0\leq s\leq t\wedge\tau^n_R }\int^s_0\int_{\mathbb{Z}}\|{\sigma}_n(r-,u_n(r-),z)\|^2_{\mathbb{H}^1}\eta(dr,dz).
\end{eqnarray*}
Applying the Burkholder-Davis-Gundy inequality (see \cite{I}), (\ref{equ-5}), the H\"{o}lder inequality and the Young inequality, we get
\begin{eqnarray*}
&&\mathbb{E}\Big|2\sup_{0\leq s\leq t\wedge\tau^n_R }\int^s_0\int_{\mathbb{Z}}\langle u_n(r-), {\sigma}_n(r-,u_n(r-),z)\rangle_{\mathbb{H}^1}\tilde{\eta}(dr, dz)\Big|\\
&\leq& C\mathbb{E}\Big[\int^{t\wedge\tau^n_R }_0\int_{\mathbb{Z}}\|u_n(s)\|^{2}_{\mathbb{H}^1}\|{\sigma}_n(s,u_n(s),z)\|^2_{\mathbb{H}^1}\vartheta(dz)ds\Big]^{\frac{1}{2}}\\
&\leq & C\Big[\mathbb{E}\sup_{0\leq s\leq t\wedge\tau^n_R }\|u_n(s)\|^2_{\mathbb{H}^1} \Big]^{\frac{1}{2}}\Big[L_1\mathbb{E}\int^{t\wedge\tau^n_R }_0(1+\|u_n(s)\|^{2}_{\mathbb{H}^1})ds\Big]^{\frac{1}{2}}\\
&\leq & \frac{1}{2}\mathbb{E}\sup_{0\leq s\leq t\wedge\tau^n_R }\|u_n(s)\|^2_{\mathbb{H}^1}+CL_1\mathbb{E}\int^{t\wedge\tau^n_R }_0(1+\|u_n(s)\|^{2}_{\mathbb{H}^1})ds\\
&\leq &\frac{1}{2}\mathbb{E}\sup_{0\leq s\leq t\wedge\tau^n_R }\|u_n(s)\|^2_{\mathbb{H}^1}+CL_1 T+CL_1\mathbb{E}\int^{t }_0\sup_{0\leq r\leq s\wedge\tau^n_R}\|u_n(r)\|^{2}_{\mathbb{H}^1}ds.
\end{eqnarray*}
Taking into account that the process
\[
t\mapsto \int^t_0\int_{\mathbb{Z}}\|{\sigma}_n(s-,u_n(s-),z)\|^2_{\mathbb{H}^1}\eta(ds, dz)
\]
has only positive jumps and by (\ref{equ-5}), we deduce that
\begin{eqnarray*}
&&\mathbb{E}\sup_{0\leq s\leq t\wedge\tau^n_R }\int^s_0\int_{\mathbb{Z}}\|{\sigma}_n(r-,u_n(r-),z)\|^2_{\mathbb{H}^1}\eta(dr, dz)\\
&\leq & \mathbb{E}\int^{t\wedge\tau^n_R }_0\int_{\mathbb{Z}}\|{\sigma}_n(r-,u_n(r-),z)\|^2_{\mathbb{H}^1}\eta(dr, dz)\\
&=& \mathbb{E}\int^{t\wedge\tau^n_R }_0\int_{\mathbb{Z}}\|{\sigma}_n(s,u_n(s),z)\|^2_{\mathbb{H}^1}\vartheta(dz)ds\\
&\leq& L_1\mathbb{E}\int^{t\wedge\tau^n_R }_0(1+\|u_n(s)\|^{2}_{\mathbb{H}^1})ds\\
&\leq& L_1T+ L_1\int^{t }_0\mathbb{E}\sup_{0\leq r\leq s\wedge\tau^n_R}\|u_n(r)\|^2_{\mathbb{H}^1}ds.
\end{eqnarray*}
Collecting the above estimates, we arrive at
\begin{eqnarray*}
&&\mathbb{E}\sup_{0\leq s\leq t\wedge\tau^n_R }\|u_n(s)\|^{2}_{\mathbb{H}^1}+\mathbb{E}\int^{t\wedge\tau^n_R }_0\|u_n(s)\|^2_{\mathbb{H}^2}ds\\
&\leq& \mathbb{E}\|\Pi_n u_0\|^{2}_{\mathbb{H}^1}+CL_1T+L_1 T+(2C_N+CL_1+L_1)\mathbb{E}\int^{t }_0\sup_{0\leq r\leq s\wedge\tau^n_R}\|u_n(r)\|^{2}_{\mathbb{H}^1}ds.
\end{eqnarray*}
By Gronwall inequality, it gives that
\begin{eqnarray*}
&&\mathbb{E}\sup_{0\leq s\leq t\wedge\tau^n_R }\|u_n(s)\|^{2}_{\mathbb{H}^1}+\mathbb{E}\int^{t\wedge\tau^n_R }_0\|u_n(s)\|^2_{\mathbb{H}^2}ds\\
&\leq&  \Big(\mathbb{E}\|\Pi_n u_0\|^{2}_{\mathbb{H}^1}+CL_1T+L_1 T\Big)\exp\Big\{(2C_N+CL_1+L_1)T\Big\}\\
&\leq& C(1+\mathbb{E}\| u_0\|^{2}_{\mathbb{H}^1}),
\end{eqnarray*}
where the positive constant $C=C(C_N,L_1,T)$ independent of $n$.

Since the process $u_n(t)$, $t\in [0,T]$ is adapted and c\`{a}dl\`{a}g, we see that $\lim_{R\rightarrow \infty}\mathbb{P}\{\tau^n_R<T\}=0$. Based on the Fatou's lemma, we conclude that (\ref{equ-12}) holds.

Applying the finite dimensional It\^{o} formula to the function $\|u_n(t)\|^6_{\mathbb{H}^1}$, it yields
\begin{eqnarray}\notag
&&\|u_n(t)\|^6_{\mathbb{H}^1}\\ \notag
&=&\|\Pi_n u_0\|^6_{\mathbb{H}^1}+6\int^t_0\|u_n(s)\|^4_{\mathbb{H}^1}\langle u_n(s), F_n(u_n(s)) \rangle_{\mathbb{H}^1} ds\\ \notag
&&\ +6\int^t_0\int_{\mathbb{Z}}\|u_n(s)\|^4_{\mathbb{H}^1}\langle u_n(s), {\sigma}_n(s-,u_n(s-),z) \rangle_{\mathbb{H}^1} \tilde{\eta}(ds,dz)\\ \notag
&&\ + \int^t_0\int_{\mathbb{Z}}(\|u_n(s-)+{\sigma}_n(s-,u_n(s-),z)\|^{6}_{\mathbb{H}^1}-\|u_n(s-)\|^{6}_{\mathbb{H}^1}-6\|u_n(s-)\|^{4}_{\mathbb{H}^1}\langle u_n(s),{\sigma}_n(s-,u_n(s-),z)\rangle_{\mathbb{H}^1})\eta(ds,dz)\\
\label{equ-15}
&=:&\|\Pi_n u_0\|^6_{\mathbb{H}^1}+I^1_n(t)+I^2_n(t)+I^3_n(t).
\end{eqnarray}
By (\ref{equ-7}), we deduce that
\begin{eqnarray*}
\mathbb{E}\sup_{0\leq s\leq t\wedge \tau^n_R}I^1_n(s)&\leq& -3\int^{t\wedge \tau^n_R}_0\|u_n(s)\|^4_{\mathbb{H}^1}\|u_n(s)\|^2_{\mathbb{H}^2}ds-3\int^{t\wedge \tau^n_R}_0\|u_n(s)\|^4_{\mathbb{H}^1}\||u_n(s)|\cdot|\nabla u_n(s)|\|^2_{L^2}ds\\
&&\ +C_N\int^{t\wedge \tau^n_R}_0\|u_n(s)\|^6_{\mathbb{H}^1}ds.
\end{eqnarray*}
Applying the Burkholder-Davis-Gundy inequality, (\ref{equ-5}), the H\"{o}lder inequality and the Young inequality, it follows that
\begin{eqnarray*}
\mathbb{E}\Big[\sup_{0\leq s\leq t\wedge \tau^n_R}I^2_n(s)\Big]
&\leq& 6C\mathbb{E}\Big[\int^{t\wedge \tau^n_R}_0\int_{\mathbb{Z}}\|u_n(s)\|^8_{\mathbb{H}^1}\|u_n(s)\|^2_{\mathbb{H}^1}\|{\sigma}_n(s,u_n(s),z)\|^2_{\mathbb{H}^1}\vartheta(dz)ds\Big]^{\frac{1}{2}}\\
&\leq& 6C\mathbb{E}\Big[\int^{t\wedge \tau^n_R}_0\|u_n(s)\|^{10}_{\mathbb{H}^1}L_1(1+\|u_n(s)\|^{2}_{\mathbb{H}^1})ds\Big]^{\frac{1}{2}}\\
&\leq& 6C\mathbb{E}\Big[\sup_{0\leq s\leq t\wedge \tau^n_R}\|u_n(s)\|^{3}_{\mathbb{H}^1}\Big(L_1\int^{t\wedge \tau^n_R}_0(1+\|u_n(s)\|^{6}_{\mathbb{H}^1})ds\Big)^{\frac{1}{2}}\Big]\\
&\leq& \frac{1}{2}\mathbb{E}\sup_{0\leq s\leq t\wedge \tau^n_R}\|u_n(s)\|^{6}_{\mathbb{H}^1}+CL_1+CL_1\mathbb{E}\int^{t\wedge \tau^n_R}_0\|u_n(s)\|^{6}_{\mathbb{H}^1}ds.
\end{eqnarray*}
By the Taylor formula, we have
\begin{eqnarray}\label{equ-14}
\Big||x+h|^{2p}-|x|^{2p}-2p|x|^{2(p-1)}(x,h)\Big|\leq C_p(|x|^{2(p-1)}|h|^2+|h|^{2p}),
\end{eqnarray}
where $C_p$ is a finite positive constant.

With the help of (\ref{equ-14}), (\ref{equ-5}) and (\ref{equ-5-11}), we deduce that
\begin{eqnarray*}
&&\mathbb{E}\Big[\sup_{0\leq s\leq t\wedge \tau^n_R}I^3_n(s)\Big]\\
&\leq& \mathbb{E}\int^{t\wedge \tau^n_R}_0\int_{\mathbb{Z}}\Big|\|u_n(s-)+{\sigma}_n(s-,u_n(s-),z)\|^{6}_{\mathbb{H}^1}-\|u_n(s-)\|^{6}_{\mathbb{H}^1}-6\|u_n(s-)\|^{4}_{\mathbb{H}^1}\langle u_n(s),{\sigma}_n(s-,u_n(s-),z)\rangle_{\mathbb{H}^1}\Big|\eta(ds,dz)\\
&\leq& C_3\mathbb{E}\int^{t\wedge \tau^n_R}_0\int_{\mathbb{Z}}\Big(\|u_n(s)\|^{4}_{\mathbb{H}^1}\|{\sigma}_n(s-,u_n(s-),z)\|^{2}_{\mathbb{H}^1}+\|{\sigma}_n(s,u_n(s),z)\|^{6}_{\mathbb{H}^1}\Big)\vartheta(dz)ds\\
&\leq& C_3(L_1+L_2)\mathbb{E}\int^{t\wedge \tau^n_R}_0(1+\|u_n(s)\|^{6}_{\mathbb{H}^1})ds\\
&\leq&  C_3(L_1+L_2)T+C_3(L_1+L_2)\mathbb{E}\int^{t}_0\sup_{0\leq r\leq s\wedge \tau^n_R}\|u_n(r)\|^{6}_{\mathbb{H}^1}ds.
\end{eqnarray*}
By (\ref{equ-15}), it follows that
\begin{eqnarray*}
&&\mathbb{E}\sup_{0\leq s\leq t\wedge \tau^n_R}\|u_n(s)\|^6_{\mathbb{H}^1}+6\mathbb{E}\int^{t\wedge \tau^n_R}_0\|u_n(s)\|^4_{\mathbb{H}^1}\|u_n(s)\|^2_{\mathbb{H}^2}ds+6\mathbb{E}\int^{t\wedge \tau^n_R}_0\|u_n(s)\|^4_{\mathbb{H}^1}\||u_n(s)|\cdot|\nabla u_n(s)|\|^2_{L^2}ds\\
&\leq&\mathbb{E}\|\Pi_n u_0\|^6_{\mathbb{H}^1}+CL_1+C_3(L_1+L_2)T +(C_N+CL_1+C_3L_1+C_3L_2)\int^{t}_0\mathbb{E}\sup_{0\leq r\leq s\wedge \tau^n_R}\|u_n(r)\|^6_{\mathbb{H}^1}ds.\\
\end{eqnarray*}
Using the Gronwall inequality, we get
\begin{eqnarray*}
&&\mathbb{E}\sup_{0\leq s\leq t\wedge \tau^n_R}\|u_n(s)\|^6_{\mathbb{H}^1}+\mathbb{E}\int^{t\wedge \tau^n_R}_0\|u_n(s)\|^4_{\mathbb{H}^1}\|u_n(s)\|^2_{\mathbb{H}^2}ds\\
&\leq& C(1+\mathbb{E}\|\Pi_n u_0\|^6_{\mathbb{H}^1}),
\end{eqnarray*}
where $C=C(C_N, L_1,L_2,T)$ independent of $n$.
Recall $\tau^n_R\uparrow T$ as $R\rightarrow \infty$, $\mathbb{P}-$a.s. and $\mathbb{P}\{\tau^n_R<T\}=0$. By Fatou's lemma, it gives the equation (\ref{equ-13}).

Based on (\ref{equ-12})-(\ref{equ-13}) and referring to (3.14) in \cite{R-Z}, we have
\begin{eqnarray}\notag
&&\sup_n\int^T_0\mathbb{E}[\|F_n(u_n(t))\|^2_{\mathbb{H}^0}]dt\\ \notag
&\leq& C\sup_n \int^T_0\mathbb{E}(\|u_n(t)\|^6_{\mathbb{H}^1}+\|u_n(t)\|^2_{\mathbb{H}^2})dt\\
\label{equ-16}
&\leq& CT\sup_n \mathbb{E}\sup_{t\in [0,T]}\|u_n(t)\|^6_{\mathbb{H}^1}+C\sup_n\mathbb{E}\int^T_0 \|u_n(t)\|^2_{\mathbb{H}^2}dt<\infty.
\end{eqnarray}
Moreover, utilizing (\ref{equ-5}), it follows that
\begin{eqnarray}\notag
&&\sup_n\mathbb{E}\int^T_0\int_{\mathbb{Z}}\|\sigma_n(t,u_n(t),z)\|^2_{\mathbb{H}^1}\vartheta(dz)dt\\ \notag
&\leq& L_1\sup_n\mathbb{E}\int^T_0(1+\|u_n(t)\|^2_{\mathbb{H}^1})dt\\ \label{equ-17}
&\leq& L_1T+L_1T\sup_n\mathbb{E}\sup_{t\in [0,T]}\|u_n(t)\|^2_{\mathbb{H}^1}<\infty.
\end{eqnarray}
With the aid of (\ref{equ-12})-(\ref{equ-13}) and (\ref{equ-16})-(\ref{equ-17}), we deduce that there exist a sequence of processes still denoted by $\{u_n; n\geq 1\}$ and elements
\begin{eqnarray*}
\bar{u}&\in& L^2(\Omega\times [0,T]; \mathbb{H}^2) \cap L^6(\Omega; L^{\infty}([0,T];\mathbb{H}^1)),\\
G&\in& L^2(\Omega\times [0,T]; \mathbb{H}^0), \quad S\in L^2(\Omega\times [0,T];L^2(\mathbb{Z},\vartheta; \mathbb{H}^1)),
\end{eqnarray*}
such that
\begin{description}
  \item[(i)] $u_n\rightarrow \bar{u}$ \ weakly\ in $L^2(\Omega\times [0,T]; \mathbb{H}^2)$,\ hence\ weakly\ in \ $ L^2(\Omega\times [0,T]; \mathbb{H}^1)$ and $ L^2(\Omega\times [0,T]; \mathbb{H}^0)$,
  \item[(ii)] $u_n\rightarrow \bar{u}$ \ in\ $L^2(\Omega; L^{\infty}([0,T];\mathbb{H}^1))$ \ with\ respect\ to\ the\ weak\ star\ topology,
  \item[(iii)] $ F_n(u_n)\rightarrow G$ weakly in $L^2(\Omega\times [0,T]; \mathbb{H}^0)$,
  \item[(iv)] $\sigma_n(t,u_n(t), z)\rightarrow S$ weakly in $L^2(\Omega\times [0,T];L^2(\mathbb{Z},\vartheta; \mathbb{H}^1) )$,
  \item[(v)] $u_n\rightarrow \bar{u}$\ weakly\ in \ $L^6(\Omega\times [0,T]; \mathbb{H}^1)$.
\end{description}

In the following, we devote to proving that there exists a solution to (\ref{equ-4}).
Define a process
\begin{eqnarray*}
u(t):=u_0+\int^t_0G(s)ds+\int^t_0\int_{\mathbb{Z}}S(s,z)\tilde{\eta}(ds,dz),\ \ t\in [0,T] \ \ {\rm{in}} \ \ \mathbb{H}^0.
\end{eqnarray*}
We can show that
$u=\bar{u}$\ \ $ dt\otimes \mathbb{P}-$a.s. Indeed, let us fix a function $\varphi\in L^2(\Omega\times [0,T]; \mathbb{R})$, by (i)-(iv), it yields that
\begin{eqnarray*}
\mathbb{E}\int^T_0 \langle \bar{u}(t), \varphi(t)e_i\rangle_{\mathbb{H}^0} dt&=&\lim_{n\rightarrow \infty}\mathbb{E}\int^T_0 \langle u_n(t), \varphi(t)e_i\rangle_{\mathbb{H}^0} dt\\
&=& \lim_{n\rightarrow \infty}\mathbb{E}\int^T_0 \langle u_n(0), \varphi(t)e_i\rangle_{\mathbb{H}^0} dt+\lim_{n\rightarrow \infty}\mathbb{E}\int^T_0\int^t_0 \langle  F_n(u_n(s)), \varphi(t)e_i\rangle_{\mathbb{H}^0} dsdt\\
&&\ +\lim_{n\rightarrow \infty}\mathbb{E}\int^T_0\langle \int^t_0\int_{\mathbb{Z}}\sigma_n(s-,u_n(s-),z)\tilde{\eta}(ds, dz), \varphi(t)e_i\rangle_{\mathbb{H}^0} dt\\
&=& \lim_{n\rightarrow \infty}\mathbb{E} \langle u_n(0),e_i\rangle_{\mathbb{H}^0} \int^T_0\varphi(t) dt+\lim_{n\rightarrow \infty}\mathbb{E}\int^T_0 \langle  F_n(u_n(s)), \int^T_s\varphi(t)dte_i\rangle_{\mathbb{H}^0} ds\\
&&\ +\lim_{n\rightarrow \infty}\mathbb{E}\int^T_0\langle \int^t_0\int_{\mathbb{Z}}\sigma_n(s-,u_n(s-),z)\tilde{\eta}(ds, dz), \varphi(t)e_i\rangle_{\mathbb{H}^0} dt\\
&=& \mathbb{E} \langle u_0,e_i\rangle_{\mathbb{H}^0} \int^T_0\varphi(t) dt+\mathbb{E}\int^T_0 \langle  G(s), \int^T_s\varphi(t)dte_i\rangle_{\mathbb{H}^0} ds\\
&&\ +\mathbb{E}\int^T_0\langle \int^t_0\int_{\mathbb{Z}}S(s,z)\tilde{\eta}(ds, dz), \varphi(t)e_i\rangle_{\mathbb{H}^0} dt\\
&=& \mathbb{E}\int^T_0\langle u(t), \varphi(t)e_i\rangle_{\mathbb{H}^0} dt.
\end{eqnarray*}
Hence, we have $u=\bar{u}$\ \ $ dt\otimes \mathbb{P}-$a.s. and $u\in L^2(\Omega\times [0,T]; \mathbb{H}^2)$. Moreover, referring to Theorem A.1 in \cite{BHZ}, it gives that $u$ is an $\mathbb{H}^0-$valued c\`{a}dl\`{a}g and $\mathcal{F}_t-$adapted process, and for any $t\in [0,T]$, the following formula holds $\mathbb{P}-$a.s.
\begin{eqnarray}\notag
\|u(t)\|^2_{\mathbb{H}^0}&=&\|u_0\|^2_{\mathbb{H}^0}+2\int^t_0 \langle G(s), u(s)\rangle_{\mathbb{H}^0} ds+2\int^t_0\int_{\mathbb{Z}}\langle u(s-), S(s,z)\rangle_{\mathbb{H}^0}\tilde{\eta}(ds,dz)\\
\label{equ-805}
&&\ +\int^t_0\int_{\mathbb{Z}}\|S(s,z)\|^2_{\mathbb{H}^0}\eta(ds, dz).
\end{eqnarray}
Referring to \cite{BHZ}, in order to prove that $u$ is a solution of (\ref{equ-4}), it suffices to show
\begin{eqnarray}\label{equu-1}
G(s,\omega)&=&F(u(s,\omega)), \quad {\rm{for}}\ d\mathbb{P}\otimes dt-a.a. (s, \omega)\in [0,T]\times \Omega,\\
\label{equu-2}
S(s,\omega,z)&=&\sigma(s,\omega, u(s,\omega),z),\quad {\rm{for}} \ d\mathbb{P}\otimes dt\times \vartheta-a.a. (s,\omega,z)\in [0,T]\times \Omega\times \mathbb{Z}.
\end{eqnarray}
In the following part, we devote to proving (\ref{equu-1})-(\ref{equu-2}) one by one. The proof is mainly based on the Minty-Browder monotonicity argument observed by Menaldi, Sritharan and Sundar (\cite{M-S,S-S}).

Let $v$ be a progressively measurable process belonging to $L^2(\Omega\times [0,T];\mathbb{H}^2)\cap L^6(\Omega; L^{\infty}([0,T];\mathbb{H}^1))$. Applying It\^{o} formula, we have
\begin{eqnarray*}
e^{-\int^t_0\rho(v(s),s)ds}\|u_n(t)\|^2_{\mathbb{H}^0}&=&\|\Pi_nu_0\|^2_{\mathbb{H}^0}-\int^t_0e^{-\int^s_0\rho(v(r),r)dr}\rho(v(s),s)\|u_n(s)\|^2_{\mathbb{H}^0}ds\\
&&\ +2\int^t_0e^{-\int^s_0\rho(v(r),r)dr}\langle u_n(s),  F_n(u_n(s))\rangle_{\mathbb{H}^0} ds\\
&&\ +2\int^t_0\int_{\mathbb{Z}}e^{-\int^s_0\rho(v(r),r)dr}\langle u_n(s), \sigma_n(s-,u_n(s-),z) \rangle_{\mathbb{H}^0}\tilde{\eta}(ds,dz)\\
&&\ +\int^t_0\int_{\mathbb{Z}}e^{-\int^s_0\rho(v(r),r)dr}\|\sigma_n(s-,u_n(s-),z)\|^2_{\mathbb{H}^0}\eta(ds,dz).
\end{eqnarray*}
Taking expectation of both sides of the above equation and using an identity $|x|^2=2(x,y)-|y|^2+|x-y|^2$, we get
\begin{eqnarray*}
&&\mathbb{E}[e^{-\int^t_0\rho(v(s),s)ds}\|u_n(t)\|^2_{\mathbb{H}^0}]-\mathbb{E}\|\Pi_n u_0\|^2_{\mathbb{H}^0}\\
&=&\mathbb{E}\Big[-\int^t_0e^{-\int^s_0\rho(v(r),r)dr}\rho(v(s),s)\Big(2\langle u_n(s), v(s)\rangle_{\mathbb{H}^0}-\|v(s)\|^2_{\mathbb{H}^0}\Big)ds\Big]\\
&&\ +2\mathbb{E}\Big[\int^t_0e^{-\int^s_0\rho(v(r),r)dr}\Big(\langle  F_n(u_n(s))-F_n(v(s)), v(s)\rangle_{\mathbb{H}^0}+\langle F_n(v(s)), u_n(s)\rangle_{\mathbb{H}^0}\Big) ds\Big]\\
&&\ +\mathbb{E}\Big[\int^t_0\int_{\mathbb{Z}}e^{-\int^s_0\rho(v(r),r)dr}\Big(2\langle\sigma_n(s,u_n(s),z), \sigma_n(v(s),z) \rangle_{\mathbb{H}^0}-\|\sigma_n(s,v(s),z)\|^2_{\mathbb{H}^0}\Big)\vartheta(dz)ds\Big]\\
&&\ +\mathbb{E}\Big[-\int^t_0e^{-\int^s_0\rho(v(r),r)dr}\rho(v(s),s)\|u_n(s)-v(s)\|^2_{\mathbb{H}^0}ds\Big]\\
&&\ +2\mathbb{E}\Big[\int^t_0e^{-\int^s_0\rho(v(r),r)dr}\langle F_n(u_n(s))-F_n(v(s)), u_n(s)-v(s)\rangle_{\mathbb{H}^0} ds\Big]\\
&&\ +\mathbb{E}\Big[\int^t_0\int_{\mathbb{Z}}e^{-\int^s_0\rho(v(r),r)dr}\|\sigma_n(s,u_n(s),z)-\sigma_n(s,v(s),z)\|^2_{\mathbb{H}^0}\vartheta(dz)ds\Big].
\end{eqnarray*}
Set $\rho(v(s),s)= 2C_0(\|v(s)\|_{\mathbb{H}^1}\|v(s)\|_{\mathbb{H}^2}+1)+K_2$, where $C_0$ is in (\ref{equ-18}) and $K_2$ is appeared in (\ref{equ-6-1}). Then, we deduce from (\ref{equ-6-1}) and (\ref{equ-18})that
\begin{eqnarray*}
&&\mathbb{E}\Big[-\int^t_0e^{-\int^s_0\rho(v(r),r)dr}\rho(v(s),s)\|u_n(s)-v(s)\|^2_{\mathbb{H}^0}ds\Big]\\
&&\ +2\mathbb{E}\Big[\int^t_0e^{-\int^s_0\rho(v(r),r)dr}\langle F_n(u_n(s))-F_n(v(s)), u_n(s)-v(s)\rangle_{\mathbb{H}^0} ds\Big]\\
&&\ +\mathbb{E}\Big[\int^t_0\int_{\mathbb{Z}}e^{-\int^s_0\rho(v(r),r)dr}\|\sigma_n(s,u_n(s),z)-\sigma_n(s,v(s),z)\|^2_{\mathbb{H}^0}\vartheta(dz)ds\Big]\leq 0.
\end{eqnarray*}
Thus,
\begin{eqnarray*}
&&\mathbb{E}[e^{-\int^t_0\rho(v(s),s)ds}\|u_n(t)\|^2_{\mathbb{H}^0}]-\mathbb{E}\|\Pi_nu_0\|^2_{\mathbb{H}^0}\\
&\leq&\mathbb{E}\Big[-\int^t_0e^{-\int^s_0\rho(v(r),r)dr}\rho(v(s),s)\Big(2\langle u_n(s), v(s)\rangle_{\mathbb{H}^0}-\|v(s)\|^2_{\mathbb{H}^0}\Big)ds\Big]\\
&&\ +2\mathbb{E}\Big[\int^t_0e^{-\int^s_0\rho(v(r),r)dr}\Big(\langle F_n(u_n(s))-F_n(v(s)), v(s)\rangle_{\mathbb{H}^0}+\langle F_n(v(s)), u_n(s)\rangle_{\mathbb{H}^0}\Big) ds\Big]\\
&&\ +\mathbb{E}\Big[\int^t_0\int_{\mathbb{Z}}e^{-\int^s_0\rho(v(r),r)dr}\Big(2\langle\sigma_n(s,u_n(s),z), \sigma_n(s,v(s),z) \rangle_{\mathbb{H}^0}-\|\sigma_n(s,v(s),z)\|^2_{\mathbb{H}^0}\Big)\vartheta(dz)ds\Big]\\
&=:& J_1(t)+J_2(t)+J_3(t).
\end{eqnarray*}
By the weak convergence (i), we have
\begin{eqnarray}\label{equ-19}
J_1(t)\rightarrow \mathbb{E}\Big[-\int^t_0e^{-\int^s_0\rho(v(r),r)dr}\rho(v(s),s)(2\langle \bar{u}(s), v(s)\rangle_{\mathbb{H}^0}-\|v(s)\|^2_{\mathbb{H}^0})ds\Big].
\end{eqnarray}
Using the weak convergence (iii), we deduce that
\begin{eqnarray}\label{equ-800}
2\mathbb{E}\int^t_0e^{-\int^s_0\rho(v(r),r)dr}\langle F_n(u_n(s)), v(s)\rangle_{\mathbb{H}^0} ds\rightarrow 2\mathbb{E}\int^t_0e^{-\int^s_0\rho(v(r),r)dr}\langle G(s), v(s)\rangle_{\mathbb{H}^0} ds.
\end{eqnarray}
With the aid of the Lebesgue dominated convergence theorem, it yields that
\begin{eqnarray}\label{equ-801}
2\mathbb{E}\int^t_0e^{-\int^s_0\rho(v(r),r)dr}\langle -F_n(v(s)), v(s)\rangle_{\mathbb{H}^0}  ds\rightarrow 2\mathbb{E}\int^t_0e^{-\int^s_0\rho(v(r),r)dr}\langle - F(v(s)), v(s)\rangle_{\mathbb{H}^0}  ds.
\end{eqnarray}
Using the Cauchy-Schwarz inequality, it follows that
\begin{eqnarray*}
&&2\mathbb{E}\int^t_0e^{-\int^s_0\rho(v(r),r)dr}(\langle F_n(v(s)), u_n(s)\rangle_{\mathbb{H}^0}-\langle F(v(s)), \bar{u}(s) \rangle_{\mathbb{H}^0}) ds\\
&=& 2\mathbb{E}\int^t_0e^{-\int^s_0\rho(v(r),r)dr}(\langle F_n(v(s))-F(v(s)), u_n(s)\rangle_{\mathbb{H}^0}+\langle F(v(s)), u_n(s)-\bar{u}(s) \rangle_{\mathbb{H}^0}) ds\\
&\leq& 2\Big(\mathbb{E}\int^t_0e^{-\int^s_0\rho(v(r),r)dr}\|F_n(v(s))-F(v(s))\|^2_{\mathbb{H}^0}ds\Big)^{\frac{1}{2}}\Big(\mathbb{E}\int^t_0e^{-\int^s_0\rho(v(r),r)dr}\|u_n(s)\|^2_{\mathbb{H}^0}ds\Big)^{\frac{1}{2}}\\
&&\ +2\mathbb{E}\int^t_0e^{-\int^s_0\rho(v(r),r)dr}\langle F(v(s)), u_n(s)-\bar{u}(s) \rangle_{\mathbb{H}^0}) ds\\
&\leq& 2\Big[\sup_n\Big(\mathbb{E}\int^t_0e^{-\int^s_0\rho(v(r),r)dr}\|u_n(s)\|^2_{\mathbb{H}^0}ds\Big)^{\frac{1}{2}}\Big]\Big(\mathbb{E}\int^t_0e^{-\int^s_0\rho(v(r),r)dr}\| F_n(v(s))-F(v(s))\|^2_{\mathbb{H}^0}ds\Big)^{\frac{1}{2}}\\
&&\ +2\mathbb{E}\int^t_0e^{-\int^s_0\rho(v(r),r)dr}\langle F(v(s)), u_n(s)-\bar{u}(s) \rangle_{\mathbb{H}^0} ds.
\end{eqnarray*}
By using (\ref{equ-12}), the Lebesgue dominated convergence theorem and the weak convergence (i), it gives that
\begin{eqnarray}\label{equ-802}
2\mathbb{E}\int^t_0e^{-\int^s_0\rho(v(r),r)dr}\Big(\langle F_n(v(s)), u_n(s)\rangle_{\mathbb{H}^0}-\langle F(v(s)), \bar{u}(s) \rangle_{\mathbb{H}^0}\Big) ds\rightarrow 0, \quad {\rm{as}}\quad n\rightarrow \infty.
\end{eqnarray}
Combining (\ref{equ-800})-(\ref{equ-802}), we have
\begin{eqnarray}\label{equ-803}
J_2(t)\rightarrow 2\mathbb{E}\int^t_0e^{-\int^s_0\rho(v(r),r)dr}\Big(\langle G(s)-F(v(s)), v(s)\rangle_{\mathbb{H}^0}+\langle F(v(s)), \bar{u}(s)\rangle_{\mathbb{H}^0}\Big) ds, \quad {\rm{as}} \quad n\rightarrow \infty.
\end{eqnarray}
For the estimates of $J_3(t)$, we adopt the method from \cite{BHZ}. By using the H\"{o}lder inequality, we get
\begin{eqnarray*}
&&\mathbb{E}\Big[\int^t_0\int_{\mathbb{Z}}e^{-\int^s_0\rho(v(r),r)dr}\Big(2\langle\sigma_n(s,u_n(s),z), \sigma_n(s,v(s),z) \rangle_{\mathbb{H}^0}-\|\sigma_n(s,v(s),z)\|^2_{\mathbb{H}^0}\Big)\vartheta(dz)ds\Big]\\
&=& 2\mathbb{E}\int^t_0\int_{\mathbb{Z}}e^{-\int^s_0\rho(v(r),r)dr}\langle\sigma_n(s,u_n(s),z), \sigma(s,v(s),z) \rangle_{\mathbb{H}^0}\vartheta(dz)ds\\
&&\ +2\mathbb{E}\int^t_0\int_{\mathbb{Z}}e^{-\int^s_0\rho(v(r),r)dr}\langle\sigma_n(s,u_n(s),z), (\sigma(s,v(s),z)-\sigma_n(s,v(s),z))\rangle_{\mathbb{H}^0}\vartheta(dz)ds\\
&&\ -\mathbb{E}\int^t_0\int_{\mathbb{Z}}e^{-\int^s_0\rho(v(r),r)dr}\|\sigma_n(s,v(s),z)\|^2_{\mathbb{H}^0}\vartheta(dz)ds\\
&\leq& 2\mathbb{E}\int^t_0\int_{\mathbb{Z}}e^{-\int^s_0\rho(v(r),r)dr}\langle\sigma_n(s,u_n(s),z), \sigma(s,v(s),z) \rangle_{\mathbb{H}^0}\vartheta(dz)ds\\
&&\ +2C\Big(\mathbb{E}\int^t_0\int_{\mathbb{Z}}\| \sigma(s,v(s),z)-\sigma_n(s,v(s),z)\|_{\mathbb{H}^0}\vartheta(dz)ds\Big)^{\frac{1}{2}}\\
&&\ -\mathbb{E}\int^t_0\int_{\mathbb{Z}}e^{-\int^s_0\rho(v(r),r)dr}\|\sigma_n(s,v(s),z)\|^2_{\mathbb{H}^0}\vartheta(dz)ds,
\end{eqnarray*}
where
\begin{eqnarray*}
C&:=&\sup_n (\mathbb{E}\int^t_0\int_{\mathbb{Z}}e^{-2\int^s_0\rho(v(r),r)dr}\|\sigma_n(s,u_n(s),z)\|^2_{\mathbb{H}^0}\vartheta(dz)ds)^{\frac{1}{2}}\\
&\leq&\sqrt{K_1}\sup_{n}\mathbb{E}\Big(\int^t_0\|u_n(s)\|^2_{\mathbb{H}^0}ds\Big)^{\frac{1}{2}}\\
&\leq&\sqrt{K_1}T\sup_{n}\mathbb{E}\sup_{t\in [0,T]}\|u_n(t)\|^2_{\mathbb{H}^0}<\infty.
\end{eqnarray*}
Then applying the weak convergence (iv) and the Lebesgue dominated convergence theorem, we deduce that as $n\rightarrow \infty$,
\begin{eqnarray}\label{equ-804}
J_3(t)\rightarrow 2\mathbb{E}\int^t_0\int_{\mathbb{Z}}e^{-\int^s_0\rho(v(r),r)dr}\Big(\langle S(s,z), \sigma(s,v(s),z) \rangle_{\mathbb{H}^0}-\|\sigma(s,v(s),z)\|^2_{\mathbb{H}^0}\Big)\vartheta(dz)ds.
\end{eqnarray}
Combining (\ref{equ-19}), (\ref{equ-803}) and (\ref{equ-804}), we conclude that
\begin{eqnarray*}
&&\liminf_{n\rightarrow \infty}\Big[\mathbb{E}[e^{-\int^t_0\rho(v(s),s)ds}\|u_n(t)\|^2_{\mathbb{H}^0}]-\mathbb{E}\|\Pi_nu_0\|^2_{\mathbb{H}^0}\Big]\\
&\leq& \mathbb{E}\Big[-\int^t_0e^{-\int^s_0\rho(v(r),r)dr}\rho(v(s),s)(2\langle \bar{u}(s), v(s)\rangle_{\mathbb{H}^0}-\|v(s)\|^2_{\mathbb{H}^0})ds\Big]\\
&&\ +2\mathbb{E}\int^t_0e^{-\int^s_0\rho(v(r),r)dr}\Big(\langle G(s)-F(v(s)), v(s)\rangle_{\mathbb{H}^0}+\langle F(v(s)), \bar{u}(s)\rangle_{\mathbb{H}^0}\Big) ds\\
&&\ +2\mathbb{E}\int^t_0\int_{\mathbb{Z}}e^{-\int^s_0\rho(v(r),r)dr}\Big(\langle S(s,z), \sigma(s,v(s),z) \rangle_{\mathbb{H}^0}-\|\sigma(s,v(s),z)\|^2_{\mathbb{H}^0}\Big)\vartheta(dz)ds.
\end{eqnarray*}
Recall that $u_n(t)\rightarrow u(t)$ weakly in $L^2(\Omega; \mathbb{H}^0)$ and by $\mathbb{E}\|\Pi_nu_0\|^2_{\mathbb{H}^0}\leq \mathbb{E}\|u_0\|^2_{\mathbb{H}^0}$, we have
\begin{eqnarray*}
\mathbb{E}[e^{-\int^t_0\rho(v(s),s)ds}\|u(t)\|^2_{\mathbb{H}^0}]-\mathbb{E}\|u_0\|^2_{\mathbb{H}^0}\leq\liminf_{n\rightarrow \infty}\Big[\mathbb{E}e^{-\int^t_0\rho(v(s),s)ds}\|u_n(t)\|^2_{\mathbb{H}^0}-\mathbb{E}\|\Pi_nu_0\|^2_{\mathbb{H}^0}\Big].
\end{eqnarray*}
Hence, it follows that
\begin{eqnarray}\notag
&&\mathbb{E}[e^{-\int^t_0\rho(v(s),s)ds}\|u(t)\|^2_{\mathbb{H}^0}]-\mathbb{E}\|u_0\|^2_{\mathbb{H}^0}\\ \notag
&\leq& \mathbb{E}\Big[-\int^t_0e^{-\int^s_0\rho(v(r),r)dr}\rho(v(s),s)(2\langle \bar{u}(s), v(s)\rangle_{\mathbb{H}^0}-\|v(s)\|^2_{\mathbb{H}^0})ds\Big]\\ \notag
&&\ +2\mathbb{E}\int^t_0e^{-\int^s_0\rho(v(r),r)dr}\Big(\langle G(s)-F(v(s)), v(s)\rangle_{\mathbb{H}^0}+\langle F(v(s)), \bar{u}(s)\rangle_{\mathbb{H}^0}\Big) ds\\
\label{equ-808}
&&\ +2\mathbb{E}\int^t_0\int_{\mathbb{Z}}e^{-\int^s_0\rho(v(r),r)dr}(\langle S(s,z), \sigma(s,v(s),z) \rangle_{\mathbb{H}^0}-\|\sigma(s,v(s),z)\|^2_{\mathbb{H}^0})\vartheta(dz)ds.
\end{eqnarray}
Recall $\|u(t)\|^2_{\mathbb{H}^0}$ defined by (\ref{equ-805}), by applying It\^{o} formula to the process $\|u(t)\|^2_{\mathbb{H}^0}e^{-\int^t_0\rho(v(s),s)ds}$, we get
\begin{eqnarray}\notag
e^{-\int^t_0\rho(v(s),s)ds}\|u(t)\|^2_{\mathbb{H}^0}&=&\|u_0\|^2_{\mathbb{H}^0}-\int^t_0e^{-\int^s_0\rho(v(r),r)dr}\rho(v(s),s)\|u(s)\|^2_{\mathbb{H}^0}ds\\ \notag
&&\ +2\int^t_0e^{-\int^s_0\rho(v(r),r)dr}\langle u(s),  G(s)\rangle_{\mathbb{H}^0} ds\\ \notag
&&\ +2\int^t_0\int_{\mathbb{Z}}e^{-\int^s_0\rho(v(r),r)dr}\langle u(s), S(s,z) \rangle_{\mathbb{H}^0}\tilde{\eta}(ds,dz)\\
\label{equ-806}
&&\ +\int^t_0\int_{\mathbb{Z}}e^{-\int^s_0\rho(v(r),r)dr}\|S(s,z)\|^2_{\mathbb{H}^0}\eta(ds,dz).
\end{eqnarray}
Taking expectation of both sides of the above equation (\ref{equ-806}), we obtain
\begin{eqnarray}\notag
&&\mathbb{E}e^{-\int^t_0\rho(v(s),s)ds}\|u(t)\|^2_{\mathbb{H}^0}-\mathbb{E}\|u_0\|^2_{\mathbb{H}^0}\\ \notag
&=&\ -\mathbb{E}\int^t_0e^{-\int^s_0\rho(v(r),r)dr}\rho(v(s),s)\|u(s)\|^2_{\mathbb{H}^0}ds\\ \notag
&&\ +2\mathbb{E}\int^t_0e^{-\int^s_0\rho(v(r),r)dr}\langle u(s),  G(s)\rangle_{\mathbb{H}^0} ds\\
\label{equ-807}
&&\ +\mathbb{E}\int^t_0\int_{\mathbb{Z}}e^{-\int^s_0\rho(v(r),r)dr}\|S(s,z)\|^2_{\mathbb{H}^0}\eta(ds,dz).
\end{eqnarray}
Plugging (\ref{equ-807}) into (\ref{equ-808}), we get for any $v\in L^2(\Omega\times [0,T];\mathbb{H}^2)\cap L^6(\Omega; L^{\infty}([0,T];\mathbb{H}^1))$,
\begin{eqnarray}\notag
&&-\mathbb{E}\int^t_0e^{-\int^s_0\rho(v(r),r)dr}\rho(v(s),s)\|\bar{u}(s)-v(s)\|^2_{\mathbb{H}^0}ds\\ \notag
&&\
+2\mathbb{E}\int^t_0e^{-\int^s_0\rho(v(r),r)dr}\langle \bar{u}(s)-v(s), G(s)-F(v(s))\rangle_{\mathbb{H}^0}ds\\
\label{equ-809}
&&\ + \mathbb{E}\int^t_0\int_{\mathbb{Z}}e^{-\int^s_0\rho(v(r),r)dr}\|S(s,z)-\sigma(s,v(s),z)\|^2_{\mathbb{H}^0}\vartheta(dz)ds\leq 0.
\end{eqnarray}
Choosing $v=\bar{u}$, we obtain
\begin{eqnarray*}
\mathbb{E}\int^t_0\int_{\mathbb{Z}}e^{-\int^s_0\rho(v(r),r)dr}\|S(s,z)-\sigma(s,\bar{u}(s),z)\|^2_{\mathbb{H}^0}\vartheta(dz)ds\leq 0,
\end{eqnarray*}
which implies $S(\cdot,\cdot)=\sigma(\cdot, \bar{u}(\cdot), \cdot)$ in $L^2(\Omega\times [0,T]; L^2(\mathbb{Z},\vartheta;\mathbb{H}^0))$. Hence, (\ref{equu-2}) holds.

Now, replacing $v$ in (\ref{equ-809}) with $v_{\varepsilon}:=\bar{u}-\varepsilon \psi$, $\psi\in L^{\infty}(\Omega\times [0,T]; \mathbb{H}^1)$ and $\varepsilon\in[-1,1]$, then
 \begin{eqnarray*}
&&-\varepsilon^2\mathbb{E}\int^t_0e^{-\int^s_0\rho(\bar{u}-\varepsilon \psi,r)dr}\rho(v_{\varepsilon}(s),s)\|\psi(s)\|^2_{\mathbb{H}^0}ds\\
&&\ +2 \varepsilon \mathbb{E}\int^t_0e^{-\int^s_0\rho(\bar{u}-\varepsilon \psi,r)dr}\langle\psi(s), G(s)-F(u(s)-\varepsilon\psi(s))\rangle_{\mathbb{H}^0}ds\leq 0.
\end{eqnarray*}
Dividing the above inequality by $\varepsilon$, we get
\begin{eqnarray}\notag
&&-\varepsilon\mathbb{E}\int^t_0e^{-\int^s_0\rho(\bar{u}-\varepsilon \psi,r)dr}\rho(v_{\varepsilon}(s),s)\|\psi(s)\|^2_{\mathbb{H}^0}ds\\
\label{equ-30}
&&\ +2 \mathbb{E}\int^t_0e^{-\int^s_0\rho(\bar{u}-\varepsilon \psi,r)dr}\langle\psi(s), G(s)-F(u(s)-\varepsilon\psi(s))\rangle_{\mathbb{H}^0}ds\leq 0
\end{eqnarray}
for $\varepsilon>0$, and
\begin{eqnarray}\notag
&&-\varepsilon\mathbb{E}\int^t_0e^{-\int^s_0\rho(\bar{u}-\varepsilon \psi,r)dr}\rho(v_{\varepsilon}(s),s)\|\psi(s)\|^2_{\mathbb{H}^0}ds\\
\label{equ-30-1}
&&\ +2 \mathbb{E}\int^t_0e^{-\int^s_0\rho(\bar{u}-\varepsilon \psi,r)dr}\langle\psi(s), G(s)-F(u(s)-\varepsilon\psi(s))\rangle_{\mathbb{H}^0}ds\geq 0
\end{eqnarray}
for $\varepsilon<0$.

Due to (\ref{equ-18}), we get
\[
|\langle \varepsilon\psi(s), F(u(s)-\varepsilon\psi(s))-F(u(s))\rangle_{\mathbb{H}^0}|\leq \frac{\varepsilon^2}{2}\|\psi(s)\|^2_{\mathbb{H}^1}+\varepsilon^2 C_0 \|\psi(s)\|^2_{\mathbb{H}^0} (\|u(s)\|_{\mathbb{H}^1}\|u(s)\|_{\mathbb{H}^2}+1),
\]
by using the Lebesgue dominated convergence theorem, it follows that
\begin{eqnarray*}
&&2 \mathbb{E}\int^t_0e^{-\int^s_0\rho(\bar{u}(r)-\varepsilon \psi(r),r)dr}\langle\psi(s), G(s)-F(u(s)-\varepsilon \psi(s))\rangle_{\mathbb{H}^0}ds\\
&\rightarrow& 2 \mathbb{E}\int^t_0e^{-\int^s_0\rho(\bar{u}(r),r)dr}\langle\psi(s), G(s)-F(u(s))\rangle_{\mathbb{H}^0}ds, \quad {\rm{as}} \quad \varepsilon\rightarrow 0.
\end{eqnarray*}
Letting $\varepsilon\rightarrow 0^+$ in (\ref{equ-30}) and  $\varepsilon\rightarrow 0^-$ in (\ref{equ-30-1}) to obtain
\begin{eqnarray*}
\mathbb{E}\int^t_0e^{-\int^s_0\rho(\bar{u}(r),r)dr}\langle\psi(s), G(s)-F(u(s))\rangle_{\mathbb{H}^0}ds=0.
\end{eqnarray*}
Since $\psi$ is arbitrary, we conclude (\ref{equu-1}).

Based on all the above results, we establish that there exists a solution to (\ref{equ-4}) if the initial value $u_0 \in L^6(\Omega, \mathcal{F}_0; \mathbb{H}^1)$.

\textbf{Step 2: } General case $u_0\in L^2(\Omega, \mathcal{F}_0; \mathbb{H}^1)$.

Taking any sequence $Y_{n,0}\in L^6(\Omega, \mathcal{F}_0; \mathbb{H}^1)$  satisfing $\mathbb{E}[\|Y_{n,0}-u_0\|^2_{\mathbb{H}^1}]\rightarrow 0$. Let $Y_n(t), t\geq 0$ be the solution of the following equation
\begin{eqnarray*}
dY_n(t)&=&F(Y_n(t))dt+\int_{\mathbb{Z}}\sigma(t-,Y_n(t-), z)\tilde{\eta}(dt,dz),\\
Y_n(0)&=&Y_{n,0}\in \mathbb{H}^1.
\end{eqnarray*}
The existence of $Y_n$ is guaranteed by Step 1. Moreover, as in the proof of (\ref{equ-12})-(\ref{equ-13}), we can prove
\begin{eqnarray}\label{equ-70}
\sup_n \mathbb{E}\Big(\sup_{t\in [0,T]}\|Y_n(t)\|^{2}_{\mathbb{H}^1}+\int^T_0\|Y_n(t)\|^2_{\mathbb{H}^2}dt\Big)&\leq& C(1+\sup_{n}\mathbb{E}\|Y_{n,0}\|^2_{\mathbb{H}^1})<\infty,\\
\label{equ-70-1}
\sup_n \mathbb{E}\Big(\sup_{t\in [0,T]}\|Y_n(t)\|^{6}_{\mathbb{H}^1}+\int^T_0\|Y_n(t)\|^{4}_{\mathbb{H}^1}\|Y_n(t)\|^2_{\mathbb{H}^2}dt\Big)&\leq&  C(1+\sup_{n}\mathbb{E}\|Y_{n,0}\|^6_{\mathbb{H}^1})<\infty,
\end{eqnarray}
which yields $F(Y_n)$ is bounded in $L^2(\Omega\times[0,T];\mathbb{H}^0)$.
Then there exists a subsequence still denoted by $\{Y_n, n\geq 1\}$ and a process $Y\in L^2(\Omega\times [0,T]; \mathbb{H}^2)\cap L^2(\Omega; L^{\infty}([0,T];\mathbb{H}^1))$ such that
\begin{description}
  \item[(I)] $Y_n\rightarrow Y$ weakly \ in \ $L^2(\Omega\times [0,T]; \mathbb{H}^2)$,
  \item[(II)] $Y_n\rightarrow Y$ in $ L^2(\Omega; L^{\infty}([0,T];\mathbb{H}^1))$ with respect to the weak star topology,
   \item[(III)] $F(Y_n)\rightarrow F(Y)$ weakly \ in \ $L^2(\Omega\times [0,T]; \mathbb{H}^0)$.
\end{description}
In the following, we adopt the same method as the proof of (3.42) in \cite{R-Z} to obtain $Y_n$ converges to $Y$ in probability in $L^{\infty}([0,T];\mathbb{H}^0)$.

For $R>0$, define the stopping time
\[
\tau^n_R:=\inf\{t\in [0,\infty): \|Y_n(t)\|_{\mathbb{H}^1}>R\}.
\]
$\tau^n_R$ is really a stopping time since $Y_n$ is continuous in $\mathbb{H}^1$. Then it follows from (\ref{equ-70}) that there exists a constant $M$ independent of $n, R$, so that
\begin{eqnarray}\label{equ-73}
\mathbb{P}(\tau^n_R\leq T)\leq \mathbb{P}\Big(\sup_{0\leq t\leq T}\|Y_n(t)\|_{\mathbb{H}^1}> R\Big)\leq \frac{M}{R^2}.
\end{eqnarray}
When $R$ is fixed, as in the proof of Theorem 3.7 in \cite{R-Z-0}, we find that
\begin{eqnarray}\label{equ-71}
\mathbb{E}\Big[\sup_{t\in [0,T]}\|Y_n(t\wedge \tau^n_R\wedge \tau^m_R)-Y_m(t\wedge \tau^n_R\wedge \tau^m_R)\|^2_{\mathbb{H}^0}\Big]\leq C_{R,T}\mathbb{E}[\|Y_{n,0}-Y_{m,0}\|^2_{\mathbb{H}^0}].
\end{eqnarray}
For $\eta>0$ and any $R>0$, we have
\begin{eqnarray}\notag
&&\mathbb{P}\Big(\sup_{0\leq t\leq T}\|Y_n(t)-Y_m(t)\|_{\mathbb{H}^0}>\eta\Big)\\ \notag
&\leq& \mathbb{P}(\tau^n_R\leq T)+\mathbb{P}(\tau^m_R\leq T)\\
\label{equ-72}
&&\ +\mathbb{P}\Big(\sup_{t\in [0,T]}\|Y_n(t\wedge \tau^n_R\wedge \tau^m_R)-Y_m(t\wedge \tau^n_R\wedge \tau^m_R)\|_{\mathbb{H}^0}> \eta\Big).
\end{eqnarray}
Given an arbitrarily small constant $\delta>0$, in view of (\ref{equ-73}), one can choose $R$ such that $\mathbb{P}(\tau^n_R\leq T)\leq \frac{\delta}{4}$ and $\mathbb{P}(\tau^m_R\leq T)\leq \frac{\delta}{4}$. For such $R$, by
(\ref{equ-71}), there exists $N_0$ such that for $m,n\geq N_0$,
\[
\mathbb{P}\Big(\sup_{t\in [0,T]}\|Y_n(t\wedge \tau^n_R\wedge \tau^m_R)-Y_m(t\wedge \tau^n_R\wedge \tau^m_R)\|_{\mathbb{H}^0}> \eta\Big)\leq \frac{\delta}{4}.
\]
Hence, by (\ref{equ-72}), we get
\[
\mathbb{P}\Big(\sup_{0\leq t\leq T}\|Y_n(t)-Y_m(t)\|_{\mathbb{H}^0}>\eta\Big)\leq \delta.
\]
That is
\begin{eqnarray}\label{equ-74}
\lim_{n,m\rightarrow \infty}\mathbb{P}\Big(\sup_{0\leq t\leq T}\|Y_n(t)-Y_m(t)\|_{\mathbb{H}^0}>\eta\Big)=0,
\end{eqnarray}
which implies that $Y_n$ converges to $Y$ in probability in $L^{\infty}([0,T];\mathbb{H}^0)$.

Finally, we need to show that $Y$ solves (\ref{equ-4}). It suffices to prove that for $v\in \mathbb{H}^0$,
\begin{eqnarray}\label{equ-31}
\langle Y(t), v\rangle_{\mathbb{H}^0}=\langle u_0, v\rangle_{\mathbb{H}^0}+\int^t_0\langle F(Y(s)), v\rangle_{\mathbb{H}^0}ds
 +\int^t_0\int_{\mathbb{Z}}\langle \sigma(s-,Y(s-),z)\tilde{\eta}(ds,dz),v\rangle_{\mathbb{H}^0} ds.
\end{eqnarray}
We know that for every $n\geq 1$,
\begin{eqnarray}\label{equ-32}
\langle Y_n(t), v\rangle_{\mathbb{H}^0}=\langle Y_{n,0}, v\rangle_{\mathbb{H}^0}+\int^t_0\langle F(Y_n(s)), v\rangle_{\mathbb{H}^0}ds +\int^t_0\int_{\mathbb{Z}}\langle \sigma(s-,Y_n(s-),z)\tilde{\eta}(ds,dz),v\rangle_{\mathbb{H}^0} ds.
\end{eqnarray}
Letting $n\rightarrow \infty$, using the convergence in probability, (I)-(III) and the Lebesgue dominated convergence theorem, we see that each term in (\ref{equ-31}) tends to the corresponding term in (\ref{equ-32}).
Hence, there exists a strong solution to (\ref{equ-4}) when the initial value $u_0\in  L^2(\Omega, \mathcal{F}_0; \mathbb{H}^1)$.

%

\textbf{Step 3: } Uniqueness.
Suppose that $u$ and $v$ are two solutions of (\ref{equ-4}) with initial values $u_0$ and $v_0$, respectively.
For some constant $R>0$, define the stopping time
\begin{eqnarray*}
\tau_R:=\inf\{t\in [0,T]: \|u(t)\|^2_{\mathbb{H}^1}\geq R\}\wedge \inf\{t\in [0,T]: \|v(t)\|^2_{\mathbb{H}^1}\geq R\}\wedge T.
\end{eqnarray*}
Applying It\^{o} formula, we get
\begin{eqnarray}\notag
&&e^{-\int^{t\wedge \tau_R}_0\rho(v(s),s)ds}\|u(t\wedge \tau_R)-v(t\wedge \tau_R)\|^2_{\mathbb{H}^0}\\ \notag
&=&\|u_0-v_0\|^2_{\mathbb{H}^0}-\int^{t\wedge \tau_R}_0e^{-\int^s_0\rho(v(r),r)dr}\rho(v(s),s)\|u(s)-v(s)\|^2_{\mathbb{H}^0}ds\\ \notag
&&\ +2\int^{t\wedge \tau_R}_0e^{-\int^s_0\rho(v(r),r)dr}\langle u(s)-v(s), F(u(s))-F(v(s))\rangle ds\\ \notag
&&\ +2\int^{t\wedge \tau_R}_0\int_{\mathbb{Z}}e^{-\int^s_0\rho(v(r),r)dr}\langle u(s)-v(s), \sigma(s-,u(s-),z)-\sigma(s-,v(s-),z) \rangle_{\mathbb{H}^0}\tilde{\eta}(ds,dz)\\
\label{equ-33}
&&\ +\int^{t\wedge \tau_R}_0\int_{\mathbb{Z}}e^{-\int^s_0\rho(v(r),r)dr}\|\sigma(s-,u(s-),z)-\sigma(s-,v(s-),z)\|^2_{\mathbb{H}^0}\eta(ds,dz).
\end{eqnarray}
We also choose $\rho(v(s),s)= 2C_0(\|v(s)\|_{\mathbb{H}^1}\|v(s)\|_{\mathbb{H}^2}+1)+K_2$, where $C_0$ is in (\ref{equ-18}) and $K_2$ is appeared in (\ref{equ-6-1}). Then, we deduce from (\ref{equ-6-1}) and (\ref{equ-18})that
\begin{eqnarray}\notag
&&\mathbb{E}\Big[e^{-\int^{t\wedge \tau_R}_0\rho(v(s),s)ds}\|u(t\wedge \tau_R)-v(t\wedge \tau_R)\|^2_{\mathbb{H}^0}\Big]-\mathbb{E}\|u_0-v_0\|^2_{\mathbb{H}^0}\\ \notag
&=&-\mathbb{E}\int^{t\wedge \tau_R}_0e^{-\int^s_0\rho(v(r),r)dr}\rho(v(s),s)\|u(s)-v(s)\|^2_{\mathbb{H}^0}ds\\ \notag
&&\ +2\mathbb{E}\int^{t\wedge \tau_R}_0e^{-\int^s_0\rho(v(r),r)dr}\langle u(s)-v(s), F(u(s))-F(v(s))\rangle_{\mathbb{H}^0} ds\\
\label{equ-34}
&&\ +\mathbb{E}\int^{t\wedge \tau_R}_0\int_{\mathbb{Z}}e^{-\int^s_0\rho(v(r),r)dr}\|\sigma(s-,u(s-),z)-\sigma(s-,v(s-),z)\|^2_{\mathbb{H}^0}\eta(ds,dz)\leq 0.
\end{eqnarray}
Hence, if $u_0=v_0$, $\mathbb{P}-$a.s., then
\begin{eqnarray}\notag
\mathbb{E}\Big[e^{-\int^{t\wedge \tau_R}_0\rho(v(s),s)ds}\|u(t\wedge \tau_R)-v(t\wedge \tau_R)\|^2_{\mathbb{H}^0}\Big]=0,\quad t\in [0,T].
\end{eqnarray}
Clearly,
\[
\int^T_0\rho(v(s),s)ds<\infty.
\]
Therefore, by setting $R\rightarrow \infty$ (hence $\tau_R\uparrow T$), we obtain $u(t)=v(t)$, $\mathbb{P}-$a.s. $t\in [0,T]$. Then the pathwise uniqueness follows from the c\`{a}dl\`{a}g property of $u$ and $v$ in $\mathbb{H}^0$. We complete the proof.

\end{proof}

\section{The weak convergence approach}
In this part, we aim to prove the large deviations for (\ref{equ-4}).
\subsection{Controlled Poisson random measure}\label{ss-1}
Recall that a Poisson random measure $\mathbf{n}$ on $\mathbb{Z}_T$ with intensity measure $\vartheta_T$ and the definition of  $\mathcal{M}_{FC}(\mathbb{Z})$ have been introduced in
subsection \ref{s-2}.
Denote by $\mathbb{P}$ the measure induced by $\mathbf{n}$ on $(\mathcal{M}_{FC}(\mathbb{Z}_T),\mathcal{B}(\mathcal{M}_{FC}(\mathbb{Z}_T)))$. Let $\mathbb{M}=\mathcal{M}_{FC}(\mathbb{Z}_T)$, then $\mathbb{P}$ is the unique probability measure on $(\mathbb{M}, \mathcal{B}(\mathbb{M}))$, under which the canonical map $\eta:\mathbb{M}\rightarrow \mathbb{M}$, $\eta(m):=m$ is a Poisson random measure with intensity measure $\vartheta_T$. In this paper, we also consider probability $\mathbb{P}_{\theta}$, for $\theta>0$, under which $\eta$ is a Poisson random measure with intensity $\theta \vartheta_T$. The corresponding expectation operators will be denoted by $\mathbb{E}$ and  $\mathbb{E}_{\theta}$, respectively.

Set
\[
\mathbb{Y}=\mathbb{Z}\times[0,\infty), \quad \mathbb{Y}_T=[0,T]\times\mathbb{Y}.
\]
 Similarly, let $\bar{\mathbb{M}}=\mathcal{M}_{FC}(\mathbb{Y}_T)$ and let $\bar{\mathbb{P}}$ be the unique probability measure on $(\bar{\mathbb{M}},\mathcal{B}(\bar{\mathbb{M}}))$ under which the canonical mapping $\bar{\eta}:\bar{\mathbb{M}}\rightarrow \bar{\mathbb{M}}, \bar{\eta}(\bar{m}):= \bar{m}$ is a Poisson random measure with intensity measure $\bar{\vartheta}_T=\lambda_T\otimes\vartheta\otimes\lambda_{\infty}$, with $\lambda_{\infty}$ being Lebesgue measure on $[0,\infty)$. The expectation operator will be denoted by $\bar{\mathbb{E}}$. Let $\mathcal{F}_t:=\sigma\{\bar{\eta}((0,s]\times {O}):0\leq s\leq t, {O}\in \mathcal{B}(\mathbb{Y})\}$,
and denote by $\bar{\mathcal{F}}_t$ the completion under $\bar{\mathbb{P}}$. Let
 \[
 \bar{\mathcal{P}}\  {\rm{ be\ the\ predictable \ }} \ {\rm{\sigma-field\ on}}\ [0,T]\times \bar{\mathbb{M}}\ {\rm{with\ the\ filtration}}\  \{\bar{\mathcal{F}}_t:0\leq t\leq T\}\ {\rm{on }}\ (\bar{\mathbb{M}},\mathcal{B}(\bar{\mathbb{M}})
 \]
 and
 \[
  \bar{\mathcal{A}}\ {\rm{ be\ the\ class\ of\ all \ (\mathcal{B}(\mathbb{Z})\otimes\bar{\mathcal{P}} )/(\mathcal{B}[0,\infty))-measurable\ maps}}\  \varphi:\mathbb{Z}_T\times \bar{\mathbb{M}}\rightarrow [0,\infty).
  \]
For $\varphi\in \bar{\mathcal{A}}$, define a counting process $\eta^\varphi$ on $\mathbb{Z}_T$ by
\begin{eqnarray}
\eta^\varphi((0,t]\times U)=\int_{(0,t]\times U}\int_{(0,\infty)}I_{[0,\varphi(s,z)]}(r)\bar{\eta}(ds dzdr), \quad t\in [0,T], \quad U\in \mathcal{B}(\mathbb{Z}).
\end{eqnarray}

$\eta^\varphi$ is the controlled random measure with $\varphi$ selecting the intensity for the points at location $x$ and time $s$, in a possibly random but non-anticipating way. If $\varphi(s,z, \bar{m})\equiv \theta\in(0,\infty)$ for some $\bar{m}\in \bar{\mathbb{M}}$, we write $\eta^\varphi=\eta^{\theta}$. Note that $\eta^{\theta}$ has the same distribution with respect to $\bar{\mathbb{P}}$ as $\eta$ has with respect to $\mathbb{P}_{\theta}$.
 Define $l:[0,\infty)\rightarrow[0,\infty)$ by
\[
l(r)=r \log r-r+1,\quad r\in [0, \infty).
\]
For any $\varphi\in \bar{\mathcal{A}}$, the quantity
\begin{eqnarray}
L_T(\varphi)=\int_{\mathbb{Z}_T}l(\varphi(t,z,\omega))\vartheta_T(dtdz)
\end{eqnarray}
is well-defined as a $[0,\infty]-$valued random variable.
\subsection{A general large deviation principle}
In order to state a general criterion for large deviation principle (LDP) introduced by Budhiraja et al. in \cite{B-D-M}, we need the following notations.
Define
\[
S^{M}=\Big\{g:\mathbb{Z}_T\rightarrow [0,\infty): L_T(g)\leq M\Big\}, \quad S=\cup_{M\geq 1}S^M.
\]
A function $g\in S^{M}$ can be identified with a measure $\vartheta^g_T\in \mathbb{M}$ given by
\[
\vartheta^g_T(O)=\int_{\mathbb{O}}g(s,x)\vartheta_T(dsdx),\quad O\in \mathcal{B}(\mathbb{Z}_T).
\]
This identification induces a topology on $ S^{M}$, under which $ S^{M}$ is a compact space (see the Appendix of \cite{B-C-D}). Throughout this paper, we always  use this topology on $S^M$. Let
\[
\mathcal{U}^M=\Big\{\varphi\in \bar{\mathcal{A}}: \varphi(\omega)\in S^M, \bar{\mathbb{P}}-a.e. \Big\},
\]
where $\bar{\mathcal{A}}$ is defined in subsection \ref{ss-1}.

Let $\{\mathcal{G}^\varepsilon\}_{\varepsilon>0}$ be a family of measurable maps from $\bar{\mathbb{M}}$ to $\mathbb{U}$, where $\bar{\mathbb{M}}$ is introduced in subsection \ref{ss-1} and $\mathbb{U}$ is a Polish space. Let $u^\varepsilon=\mathcal{G}^\varepsilon(\varepsilon \eta^{\varepsilon^{-1}})$. Now, we list the following sufficient conditions for establishing LDP for the family $\{u^\varepsilon\}_{\varepsilon>0}$.
\begin{description}
  \item[\textbf{Condition A} ] There exists a measurable map $\mathcal{G}^0: \bar{\mathbb{M}}\rightarrow \mathbb{U}$ such that the following hold.
\end{description}
\begin{description}
\item[(i)] For every $M<\infty$, let $g_n, g\in S^M$  be such that $g_n\rightarrow g$ as $n\rightarrow \infty$. Then,
      $\mathcal{G}^0(\vartheta^{g_n}_T)\rightarrow \mathcal{G}^0(\vartheta^{g}_T)$ in $\mathbb{U}$.
  \item[(ii)] For every $M<\infty$, let $\{\varphi_{\varepsilon}: \varepsilon>0\}\subset \mathcal{U}^M$ be such that $\varphi_{\varepsilon}$ converges in distribution to $\varphi$ as $\varepsilon\rightarrow 0$. Then,
     $\mathcal{G}^{\varepsilon}(\varepsilon \eta^{\varepsilon^{-1}\varphi_{\varepsilon}})$ converges to $\mathcal{G}^0(\vartheta^{\varphi}_T)$ in distribution.

\end{description}
The following result is due to Budhiraja et al. in \cite{B-D-M}.
\begin{thm}
Suppose the above \textbf{Condition A} holds. Then $u^{\varepsilon}$ satisfies a large deviation principle on $\mathbb{U}$ with the good rate function $I$ given by
\begin{eqnarray}\label{equ-57}
I(f)=\inf_{\{g\in {S}: f=\mathcal{G}^0(\vartheta^g_T)\}}   \Big\{L_T(g)\Big\},\ \ \forall f\in\mathbb{U}.
\end{eqnarray}
By convention, $I(\emptyset)=\infty.$
\end{thm}
\subsection{Hypotheses and the statement of main results}
In order to obtain LDP for stochastic 3D tamed equations (\ref{equ-4}), we need additional conditions on the coefficients. Here, we adopt similar conditions as \cite{Y-Z-Z} and state some preliminary results from Budhiraja et al. \cite{B-C-D}.

Let ${\sigma}:[0,T]\times \mathbb{H}^0\times \mathbb{Z}\rightarrow \mathbb{H}^0$ ($[0,T]\times \mathbb{H}^1\times \mathbb{Z}\rightarrow \mathbb{H}^1$) be a measurable mapping. Set
\begin{eqnarray*}
 \|{\sigma}(t,z)\|_{0, \mathbb{H}^0}&:=&\sup_{u\in \mathbb{H}^0}\frac{\|{\sigma}(t,u,z)\|_{\mathbb{H}^0}}{1+\|u\|_{\mathbb{H}^0}},\quad (t,z)\in [0,T]\times \mathbb{Z},\\
  \|{\sigma}(t,z)\|_{1, \mathbb{H}^0}&:=&\sup_{u_1,u_2\in \mathbb{H}^0, u_1\neq u_2}\frac{\|{\sigma}(t,u_1,v)-{\sigma}(t,u_2,z)\|_{\mathbb{H}^0}}{\|u_1-u_2\|_{\mathbb{H}^0}},\quad (t,z)\in [0,T]\times \mathbb{Z},\\
  \|{\sigma}(t,z)\|_{0, \mathbb{H}^1}&:=&\sup_{u\in \mathbb{H}^1}\frac{\|{\sigma}(t,u,z)\|_{\mathbb{H}^1}}{1+\|u\|_{\mathbb{H}^1}}, \quad (t,z)\in [0,T]\times \mathbb{Z},\\
  \|{\sigma}(t,z)\|_{1, \mathbb{H}^1}&:=&\sup_{u_1,u_2\in \mathbb{H}^1, u_1\neq u_2}\frac{\|{\sigma}(t,u_1,z)-{\sigma}(t,u_2,z)\|_{\mathbb{H}^1}}{\|u_1-u_2\|_{\mathbb{H}^1}}, \quad (t,z)\in [0,T]\times \mathbb{Z},
  \end{eqnarray*}
\begin{description}
  \item[\textbf{Hypothesis H1} ] For $i=0,1$, $j=0,1$, there exists $\delta^{i,j}>0$ such that for all $E\in \mathcal{B}([0,T]\times \mathbb{Z})$ satisfying $\vartheta_T(E)<\infty$, the following holds
      \[
      \int_E e^{\delta^{i,j} \|\sigma(s,z)\|^2_{i,\mathbb{H}^j}}\vartheta(dz)ds<\infty.
      \]
\end{description}
Now, we state the following lemmas established by \cite{B-C-D} and \cite{Y-Z-Z}.
\begin{lemma}\label{lem-2}
Under \textbf{Hypothesis H0} and  \textbf{Hypothesis H1},
\begin{description}
  \item[(i)] For $i=0,1$, $j=0,1$ and every $M\in \mathbb{N}$,
  \begin{eqnarray}\label{equ-36}
  C^{M,1}_{i,j}&:=&\sup_{g\in S^M}\int_{\mathbb{Z}_T}\|\sigma(s,z)\|_{i,\mathbb{H}^j
  }|g(s,z)-1|\vartheta(dz)ds<\infty,\\
  \label{equ-37}
  C^{M,2}_{i,j}&:=&\sup_{g\in S^M}\int_{\mathbb{Z}_T}\|\sigma(s,z)\|^2_{i,\mathbb{H}^j
  }|g(s,z)+1|\vartheta(dz)ds<\infty.
  \end{eqnarray}
  \item[(ii)] For  $i=0,1$, $j=0,1$ and every $\eta>0$, there exists $\delta>0$ such that for any $A\subset [0,T]$ satisfying $\lambda_T(A)<\delta$
    \begin{eqnarray}\label{equ-38}
  \sup_{g\in S^M}\int_A\int_{\mathbb{Z}}\|\sigma(s,z)\|_{i,\mathbb{H}^j
  }|g(s,z)-1|\vartheta(dz)ds\leq \eta.
  \end{eqnarray}

\end{description}

\end{lemma}
\begin{lemma}\label{lem-3}
\begin{description}
  \item[(1)] For any $g\in S$, if $\sup_{t\in[0,T]}\|Y(t)\|_{\mathbb{H}^1}<\infty$, then
  \[
  \int_{\mathbb{Z}}\sigma(\cdot,Y(\cdot),z)(g(\cdot,z)-1)\vartheta(dz)\in L^1([0,T];\mathbb{H}^1).
  \]
  \item[(2)] If the family of mappings $\{Y_n:[0,T]\rightarrow \mathbb{H}^1, n\geq 1\}$ satisfying $\sup_n\sup_{t\in[0,T]}\|Y_n(t)\|_{\mathbb{H}^1}<\infty$, then
      \[
      \tilde{C}_{M}:=\sup_{g\in S^M}\sup_n \int^T_0\|\int_{\mathbb{Z}}\sigma(t,Y_n(t),z)(g(t,z)-1)\vartheta(dz)\|_{\mathbb{H}^1}ds<\infty.
      \]
\end{description}
\end{lemma}

\begin{lemma}\label{lem-4}
Let $h:[0,T]\times \mathbb{Z}\rightarrow \mathbb{R}$ be a measurable function such that
\[
\int_{\mathbb{Z}_T}|h(s,z)|^2\vartheta(dz)ds<\infty,
\]
and for all $\delta\in (0,\infty)$ and $E\in \mathcal{B}([0,T]\times \mathbb{Z})$ satisfying $\vartheta_T(E)<\infty$,
\[
\int_E \exp(\delta |h(s,z)|)\vartheta(dz)ds<\infty.
\]
Then, we have
\begin{description}
  \item[(1)] Fix $M\in \mathbb{N}$. Let $g_n, g\in S^M$ be such that $g_n\rightarrow g$ as $n\rightarrow \infty$. Then
      \[
      \lim_{n\rightarrow \infty}\int_{\mathbb{Z}_T}h(s,z)(g_n(s,z)-1)\vartheta(dz)ds=\int_{\mathbb{Z}_T}h(s,z)(g(s,z)-1)\vartheta(dz)ds.
      \]
  \item[(2)] Fix $M\in \mathbb{N}$. Given $\varepsilon>0$, there exists a compact set $K_{\varepsilon}\subset \mathbb{Z}$, such that
      \[
      \sup_{g\in S^M}\int^T_0\int_{K^c_{\varepsilon}}|h(s,z)||g(s,z)-1|\vartheta(dz)ds\leq \varepsilon.
      \]
  \item[(3)] For every compact set $K\subset \mathbb{Z}$,
  \[
  \lim_{M\rightarrow \infty}\sup_{g\in S^M}\int^T_0\int_{K}|h(s,z)|I_{\{h\geq M\}}g(s,z)\vartheta(dz)ds=0.
  \]
\end{description}
\end{lemma}

In this paper, we consider the following stochastic 3D tamed equations driven by small multiplicative L\'{e}vy noise:
\begin{eqnarray}\label{equ-39}
   du^\varepsilon(t)=-[A u^\varepsilon(t)+B(u^\varepsilon(t))+\mathcal{P}g_N(|u^\varepsilon(t)|^2)u^\varepsilon(t)]dt
+\varepsilon\int_{\mathbb{Z}}\sigma(t-,u^\varepsilon(t-),z)\tilde{\eta}^{\varepsilon^{-1}}(dt,dz).
\end{eqnarray}
By Theorem \ref{thm-1}, under \textbf{Hypothesis H0}, there exists a unique strong solution of (\ref{equ-39}) in $\mathcal{D}([0,T];\mathbb{H}^1)\cap L^2([0,T];\mathbb{H}^2)$. Therefore, there exists a Borel-measurable mapping:
\[
\mathcal{G}^{\varepsilon}: \bar{\mathbb{M}}\rightarrow \mathcal{D}([0,T];\mathbb{H}^1) \cap L^2([0,T];\mathbb{H}^2)
\]
such that $u^{\varepsilon}(\cdot)=\mathcal{G}^{\varepsilon}(\varepsilon {\eta}^{\varepsilon^{-1}})$.

For $g\in S$, consider the following skeleton equation
\begin{eqnarray}\label{equ-40}
    du^g(t)=-[A u^g(t)+B(u^g(t))+\mathcal{P}g_N(|u^g(t)|^2)u^g(t)]dt+\int_{\mathbb{Z}}\sigma(t,u^g(t),z)(g(t,z)-1)\vartheta(dz)dt.
\end{eqnarray}
The solution $u^g$ defines a mapping $\mathcal{G}^0: \bar{\mathbb{M}}\rightarrow \mathcal{D}([0,T];\mathbb{H}^1) \cap L^2([0,T];\mathbb{H}^2)$ such that  $u^g(\cdot)=\mathcal{G}^0(\vartheta^g_T)$.

Our main result reads as
\begin{thm}\label{thm-3}
Let $u_0 \in \mathbb{H}^1$. Under \textbf{Hypothesis H0} and \textbf{Hypothesis H1}, $u^{\varepsilon}$ satisfies a large deviation principle on $\mathcal{D}([0,T]; \mathbb{H}^1)$ with the good rate function $I$ defined by (\ref{equ-57}) with respect to the uniform convergence.
\end{thm}
\begin{proof}
According to Theorem \ref{thm-1}, we need to prove (i) and (ii) in \textbf{Condition A}. The verification of (i) will be established by Proposition \ref{prp-1}, (ii) will be proved by Theorem \ref{thm-4}.
\end{proof}

\section{The skeleton equation}
In this section, we will show that the skeleton equation (\ref{equ-40}) admits a unique solution for every $g\in S$.

Let $K$ be a Banach space with norm $\|\cdot\|_K$.
Given $p>1, \alpha\in (0,1)$, as in \cite{FG95}, let $W^{\alpha,p}([0,T]; K)$ be the Sobolev space of all $u\in L^p([0,T];K)$ such that
\[
\int^T_0\int^T_0\frac{\|u(t)-u(s)\|_K^{ p}}{|t-s|^{1+\alpha p}}dtds< \infty,
\]
endowed with the norm
\[
\|u\|^p_{W^{\alpha,p}([0,T]; K)}=\int^T_0\|u(t)\|_K^pdt+\int^T_0\int^T_0\frac{\|u(t)-u(s)\|_K^{ p}}{|t-s|^{1+\alpha p}}dtds.
\]
The following results can be found in \cite{FG95}.
\begin{lemma}\label{lem-5}
Let $B_0\subset B\subset B_1$ be Banach spaces, $B_0$ and $B_1$ reflexive, with compact embedding $B_0\subset B$. Let $p\in (1, \infty)$ and $\alpha \in (0, 1)$ be given. Let $X$ be the space
\[
X= L^p([0, T]; B_0)\cap W^{\alpha, p}([0,T]; B_1),
\]
endowed with the natural norm. Then the embedding of $X$ in $L^p([0,T];B)$ is compact.
\end{lemma}
\begin{lemma}\label{lem-6}
For $V$ and $H$ are two Hilbert spaces ($V'$ is the dual space of $V$) with $V\subset\subset H=H'\subset V'$, where $V\subset\subset H$ denotes $V$ is compactly embedded in $H$. If $u\in L^2([0,T];V)$, $\frac{du}{dt}\in L^2([0,T];V')$, then $u\in C([0,T];H)$.
\end{lemma}
For the skeleton equation (\ref{equ-40}), we have
\begin{thm}\label{thm-2}
Given $u_0\in \mathbb{H}^1 $ and $g\in S$. Assume \textbf{Hypothesis H0} and \textbf{Hypothesis H1} hold, then there exists a unique solution $u^g$ such that
\[
u^g\in C([0,T];\mathbb{H}^1)\cap L^2([0,T];\mathbb{H}^2),
\]
and
\begin{eqnarray}\notag
u^g(t)&=&u_0-\int^t_0 A u^g(s)ds-\int^t_0 B( u^g(s))ds-\int^t_0 \mathcal{P}g_N( |u^g(s)|^2)u^g(s)ds\\
\label{equ-41}
&&\quad +\int^t_0\int_{\mathbb{Z}}\sigma(s,u^g(s),z)(g(s,z)-1)\vartheta(dz)ds\quad {\rm{holds\ \ in}} \quad L^2([0,T];\mathbb{H}^{0}).
\end{eqnarray}
Moreover, for any $M\in \mathbb{N}$, there exists $C(p,M)>0$ such that
\begin{eqnarray}\label{equ-42}
\sup_{g\in S^M} \left(\sup_{s\in [0,T]}\|u^g(s)\|^2_{\mathbb{H}^1}+\int^T_0\|u^g(s)\|^{2}_{\mathbb{H}^2}ds\right)\leq C(p,M).
\end{eqnarray}

\end{thm}
\begin{proof}
\textbf{(Existence)} \ Let $\Phi_n: \mathbb{R}\rightarrow [0,1]$ be a smooth function such that $\Phi_n(t)=1$, if $|t|\leq n$, $\Phi_n(t)=0$ if $|t|> n+1$. Set $\chi_n(u)=\Phi_n(\|u\|_{\mathbb{H}^1})$, $u\in \mathbb{H}^1$.
Recall  $\Pi_n$ is the orthogonal projection from $\mathbb{H}^0$ onto the finite dimensional space $\mathbb{H}_n:=span\{e_1,e_2, \cdot\cdot\cdot, e_n\}$ defined as
\[
\Pi_n u:=\sum^n_{i=1}\langle u, e_i\rangle_{\mathbb{H}^0}e_i,
\]
where $\{e_i\}^{\infty}_{i=1}\subset \mathbb{H}^2$ is an orthonormal basis in $\mathbb{H}^0$ composed of eigenvectors of  $A$ such that $span\{e_i, i\geq 1\}$ is dense in $\mathbb{H}^1$. Moreover, it is easy to see that $\{e_i\}^{\infty}_{i=1}$ is also orthogonal in $\mathbb{H}^1$.
Define
\[
B_n(u,u)=\chi_n(u)B(u,u), \quad u\in \Pi_n\mathbb{H}^1.
\]
Consider the following Faedo-Galerkin's approximations: $u_n(t)\in \Pi_n \mathbb{H}^1$ satisfying that
\begin{eqnarray}\notag
du_n(t)&=&-Au_n(t)dt-\Pi_n B_n(u_n,u_n)dt-\Pi_n\mathcal{P}g_N(|u_n(t)|^2)u_n(t)dt\\
\label{equ-43}
&&\ +\int_{\mathbb{Z}} {\sigma}_n(t,u_n(t),z)(g(t,z)-1)\vartheta(dz)dt,
\end{eqnarray}
with initial value $u_n(0)=\Pi_nu_0$, where $ {\sigma}_n=\Pi_n \sigma$.

Since $B_n$ is a Lipschitz operator from $\Pi_n \mathbb{H}^1$ onto $\Pi_n \mathbb{H}^1$, the solution of equation (\ref{equ-43}) can be obtained through an iteration argument as follows.
Let $Y_0(t)=\Pi_nu_0$, $t\in [0,T]$. Suppose that $Y_{m-1}$ has been defined. Define $Y_m\in C([0,T];\Pi_n\mathbb{H}^1)\cap L^2([0,T];\Pi_n\mathbb{H}^2)$ as the unique solution of the equation
\begin{eqnarray*}
dY_m(t)&=&-AY_m(t)dt-\Pi_n B_n(Y_m,Y_m)dt-\Pi_n\mathcal{P}g_N(|Y_m(t)|^2)Y_m(t)dt\\
\label{equ-44}
&&\ +\int_{\mathbb{Z}} {\sigma}_n(t,Y_{m-1}(t),z)(g(t,z)-1)\vartheta(dz)dt,
\end{eqnarray*}
and $Y_m(0)=\Pi_nu_0$. Using similar methods as the proof of Theorem 3.1 in \cite{Y-Z-Z}, we can show that the limit $u_n$ of $Y_m$ is the unique strong solution to (\ref{equ-43}) and $u_n\in C([0,T];\Pi_n\mathbb{H}^1)\cap L^2([0,T];\Pi_n\mathbb{H}^2)$.

Now, for the solution $u_n(t)$ of (\ref{equ-43}), we aim to prove the following estimates:
\begin{eqnarray}\label{equ-45}
\sup_{n\geq 1}\Big(\sup_{t\in [0,T]}\|u_n(t)\|^2_{\mathbb{H}^1}+\int^T_0\|u_n(t)\|^2_{\mathbb{H}^2}dt\Big)&\leq& C_1,\\
\label{equ-45-1}
\sup_{n\geq 1}\Big(\sup_{t\in [0,T]}\|u_n(t)\|^6_{\mathbb{H}^1}+\int^T_0\|u_n(t)\|^4_{\mathbb{H}^1}\|u_n(t)\|^2_{\mathbb{H}^2}dt\Big)&\leq& C_2,
\end{eqnarray}
and for $\alpha\in (0,\frac{1}{2})$, there exists a constant $C_{\alpha}>0$ such that
\begin{eqnarray}\label{equ-46}
\sup_{n\geq 1}\|u_n\|^2_{W^{\alpha,2}([0,T];\mathbb{H}^0)}\leq C_{\alpha}.
\end{eqnarray}
Firstly, we make estimates of $\|u_n(t)\|^2_{\mathbb{H}^1}$. By the chain rule, we  obtain
\begin{eqnarray*}
\|u_n(t)\|^2_{\mathbb{H}^1}&=&\|u_n(0)\|^2_{\mathbb{H}^1}-2\int^t_0\langle u_n(s),Au_n(s)\rangle_{\mathbb{H}^1}ds-2\int^t_0\langle u_n(s),\Pi_nB_n(u_n,u_n)\rangle_{\mathbb{H}^1}ds\\
&&\ -2\int^t_0\langle u_n(s),\Pi_n\mathcal{P}g_N(|u_n(s)|^2)u_n(s)\rangle_{\mathbb{H}^1}ds\\
&&\ +2\int^t_0\langle u_n(s),\int_{\mathbb{Z}}{\sigma}_n(s,u_n(s),z)(g(s,z)-1)\vartheta(dz)\rangle_{\mathbb{H}^1}ds.
\end{eqnarray*}
Using (\ref{equ-7}), it follows that
\begin{eqnarray*}
&&\|u_n(t)\|^2_{\mathbb{H}^1}+\int^t_0\|u_n(s)\|^2_{\mathbb{H}^2}ds+\int^t_0\||u_n(s)|\cdot |\nabla u_n(s)|\|^2_{L^2}ds\\
&\leq&\|u_n(0)\|^2_{\mathbb{H}^1}+C_N\int^t_0(1+\|u_n(s)\|^2_{\mathbb{H}^1})ds\\
&&\ +2\int^t_0\langle u_n(s),\int_{\mathbb{Z}}{\sigma}_n(s,u_n(s),z)(g(s,z)-1)\vartheta(dz)\rangle_{\mathbb{H}^1}ds.
\end{eqnarray*}
According to Hypothesis H1, we obtain
\begin{eqnarray*}
&&2\int^t_0\langle u_n(s),\int_{\mathbb{Z}}{\sigma}_n(s,u_n(s),z)(g(s,z)-1)\vartheta(dz)\rangle_{\mathbb{H}^1}ds\\
&\leq& 2\int^t_0\int_{\mathbb{Z}}\|{\sigma}_n(s,u_n(s),z)\|_{\mathbb{H}^1}|g(s,z)-1|\|u_n(s)\|_{\mathbb{H}^1}\vartheta(dz)ds\\
&\leq& 2\int^t_0\int_{\mathbb{Z}}\|{\sigma}(s,z)\|_{0,\mathbb{H}^1}|g(s,z)-1|(1+2\|u_n(s)\|^2_{\mathbb{H}^1})\vartheta(dz)ds\\
&\leq& 2\int^t_0\int_{\mathbb{Z}}\|{\sigma}(s,z)\|_{0,\mathbb{H}^1}|g(s,z)-1|\vartheta(dz)ds\\
&&\ +4\int^t_0\|u_n(s)\|^2_{\mathbb{H}^1}\int_{\mathbb{Z}}\|{\sigma}(s,z)\|_{0,\mathbb{H}^1}|g(s,z)-1|\vartheta(dz)ds.
\end{eqnarray*}
Hence,
\begin{eqnarray*}
&&\sup_{s\in [0,t]}\|u_n(s)\|^2_{\mathbb{H}^1}+\int^t_0\|u_n(s)\|^2_{\mathbb{H}^2}ds+\int^t_0\||u_n(s)|\cdot |\nabla u_n(s)|\|^2_{L^2}ds\\
&\leq&\|u_n(0)\|^2_{\mathbb{H}^1}+C_N\int^t_0(1+\|u_n(s)\|^2_{\mathbb{H}^1})ds+2\int^t_0\int_{\mathbb{Z}}\|{\sigma}(s,z)\|_{0,\mathbb{H}^1}|g(s,z)-1|\vartheta(dz)ds\\
&&\ +4\int^t_0\|u_n(s)\|^2_{\mathbb{H}^1}\int_{\mathbb{Z}}\|{\sigma}(s,z)\|_{0,\mathbb{H}^1}|g(s,z)-1|\vartheta(dz)ds.
\end{eqnarray*}
Applying Gronwall inequality, we get
\begin{eqnarray*}
&&\sup_{s\in [0,t]}\|u_n(s)\|^2_{\mathbb{H}^1}+\int^t_0\|u_n(s)\|^2_{\mathbb{H}^2}ds\\
&\leq&\Big(\|u_n(0)\|^2_{\mathbb{H}^1}+C_Nt+2\int^t_0\int_{\mathbb{Z}}\|\sigma(s,z)\|_{0,\mathbb{H}^1}|g(s,z)-1|\vartheta(dz)ds\Big)\\
&&\ \times\exp\Big\{\int^t_0\Big(C_N+\int_{\mathbb{Z}}\|\sigma(s,z)\|_{0,\mathbb{H}^1}|g(s,z)-1|\vartheta(dz)\Big)ds\Big\}.
\end{eqnarray*}
With the help of Lemma \ref{lem-2}, we conclude the result (\ref{equ-45}).

Similarly to the above, by the chain rule, we obtain
\begin{eqnarray*}
\|u_n(t)\|^6_{\mathbb{H}^1}&=&\|u_n(0)\|^6_{\mathbb{H}^1}-6\int^t_0\|u_n(s)\|^4_{\mathbb{H}^1}\langle u_n(s), Au_n(s)\rangle_{\mathbb{H}^1}\\
&&\ -6\int^t_0\|u_n(s)\|^4_{\mathbb{H}^1}\langle u_n(s), \Pi_nB_n(u_n(s),u_n(s))\rangle_{\mathbb{H}^1}ds\\
&&\ -6\int^t_0\|u_n(s)\|^4_{\mathbb{H}^1}\langle u_n(s), \Pi_n \mathcal{P}g_N(|u_n(s)|^2)u_n(s)\rangle_{\mathbb{H}^1}ds\\
&&\ +6\int^t_0\|u_n(s)\|^4_{\mathbb{H}^1}\langle u_n(s),\int_{\mathbb{Z}}\sigma_n(s,u_n(s),z)(g(s,z)-1)\vartheta(dz)\rangle_{\mathbb{H}^1} ds.
\end{eqnarray*}
We deduce from  (\ref{equ-7}) that
\begin{eqnarray*}
&&\|u_n(t)\|^6_{\mathbb{H}^1}+3\int^t_0\|u_n(s)\|^4_{\mathbb{H}^1}\|u_n(s)\|^2_{\mathbb{H}^2}ds+3\int^t_0\|u_n(s)\|^4_{\mathbb{H}^1}\||u_n(s)|\cdot|\nabla u_n(s)|\|^2_{L^2}ds\\
&\leq&\|u_n(0)\|^6_{\mathbb{H}^1}+6C_N\int^t_0\|u_n(s)\|^4_{\mathbb{H}^1}(1+\| u_n(s)\|^2_{\mathbb{H}^1})ds\\
&&\ +6\int^t_0\|u_n(s)\|^4_{\mathbb{H}^1}\langle u_n(s),\int_{\mathbb{Z}}\sigma_n(s,u_n(s),z)(g(s,z)-1)\vartheta(dz)\rangle_{\mathbb{H}^1} ds.
\end{eqnarray*}
Using Hypothesis H1, we get
\begin{eqnarray*}
&&6\int^t_0\|u_n(s)\|^4_{\mathbb{H}^1}\langle u_n(s),\int_{\mathbb{Z}}\sigma_n(s,u_n(s),z)(g(s,z)-1)\vartheta(dz)\rangle_{\mathbb{H}^1} ds\\
&\leq& 6\int^t_0\|u_n(s)\|^4_{\mathbb{H}^1}\int_{\mathbb{Z}}\|\sigma(s,u_n(s),z)\|_{\mathbb{H}^1}|g(s,z)-1|\|u_n(s)\|_{\mathbb{H}^1}\vartheta(dz) ds\\
&\leq& 6\int^t_0\|u_n(s)\|^4_{\mathbb{H}^1}\int_{\mathbb{Z}}\|\sigma(s,z)\|_{0,\mathbb{H}^1}|g(s,z)-1|(1+2\|u_n(s)\|^2_{\mathbb{H}^1})\vartheta(dz) ds\\
&\leq& 6\int^t_0\int_{\mathbb{Z}}\|\sigma(s,z)\|_{0,\mathbb{H}^1}|g(s,z)-1|\vartheta(dz)(1+\|u_n(s)\|^6_{\mathbb{H}^1})ds\\
&&\ +12\int^t_0\int_{\mathbb{Z}}\|\sigma(s,z)\|_{0,\mathbb{H}^1}|g(s,z)-1|\vartheta(dz)\|u_n(s)\|^6_{\mathbb{H}^1}ds\\
&=& 6\int^t_0\int_{\mathbb{Z}}\|\sigma(s,z)\|_{0,\mathbb{H}^1}|g(s,z)-1|\vartheta(dz)ds+18\int^t_0\|u_n(s)\|^6_{\mathbb{H}^1}\int_{\mathbb{Z}}\|\sigma(s,z)\|_{0,\mathbb{H}^1}|g(s,z)-1|\vartheta(dz)ds.
\end{eqnarray*}
Hence, we conclude that
\begin{eqnarray*}
&&\sup_{s\in [0,t]}\|u_n(s)\|^6_{\mathbb{H}^1}+3\int^t_0\|u_n(s)\|^4_{\mathbb{H}^1}\|u_n(s)\|^2_{\mathbb{H}^2}ds+3\int^t_0\|u_n(s)\|^4_{\mathbb{H}^1}\||u_n(s)|\cdot|\nabla u_n(s)|\|^2_{L^2}ds\\
&\leq&\|u_n(0)\|^6_{\mathbb{H}^1}+6C_N\int^t_0\|u_n(s)\|^4_{\mathbb{H}^1}(1+\| u_n(s)\|^2_{\mathbb{H}^1})ds+6\int^t_0\int_{\mathbb{Z}}\|\sigma(s,z)\|_{0,\mathbb{H}^1}|g(s,z)-1|\vartheta(dz)ds\\
&&\ +18\int^t_0\|u_n(s)\|^6_{\mathbb{H}^1}\int_{\mathbb{Z}}\|\sigma(s,z)\|_{0,\mathbb{H}^1}|g(s,z)-1|\vartheta(dz)ds.
\end{eqnarray*}
Applying Gronwall inequality, we deduce the result (\ref{equ-45-1}).

For $u_n(t)$, it can be written as
\begin{eqnarray}\notag
u_n(t)&=&\Pi_nu_0-\int^t_0Au_n(s)ds-\int^t_0\Pi_n B_n(u_n(s))ds-\int^t_0\Pi_n \mathcal{P}g_N(|u_n(s)|^2)u_n(s)ds\\ \notag
&&\ +\int^t_0 \int_{\mathbb{Z}} {\sigma}_n(s,u_n(s),z)(g(s,z)-1)\vartheta(dz)ds\\ \notag
&=& \Pi_nu_0+\int^t_0\tilde{F}_n(u_n(s))ds+\int^t_0\int_{\mathbb{Z}} {\sigma}_n(s,u_n(s),z)(g(s,z)-1)\vartheta(dz)ds\\
\label{equ-47}
&=:& J^1_n(t)+ J^2_n(t)+ J^3_n(t).
\end{eqnarray}
Clearly, $\sup_{n\geq 1}\|J^1_n(t)\|_{\mathbb{H}^1}=\sup_{n\geq 1}\|\Pi_nu_0\|_{\mathbb{H}^1}\leq C_1$. Using (3.14) in \cite{R-Z}, we have
\begin{eqnarray}\label{equ-48}
\|\tilde{F}_n(u_n(t))\|_{L^2}\leq C(\|u_n(t)\|^3_{\mathbb{H}^1}+\|u_n(t)\|_{\mathbb{H}^2}).
\end{eqnarray}
For $s<t$, it follows from (\ref{equ-48}) that
\begin{eqnarray}\notag
\|J^2_n(t)-J^2_n(s)\|^2_{\mathbb{H}^0}&=&\|\int^t_s\tilde{F}_n(u_n(l))dl\|^2_{\mathbb{H}^0}\\ \notag
&\leq& \Big(\int^t_s\|\tilde{F}_n(u_n(l))\|_{\mathbb{H}^0}\Big)^2\\ \notag
&\leq& C\Big[\int^t_s(\|u_n(l)\|^3_{\mathbb{H}^1}+\|u_n(l)\|_{\mathbb{H}^2})dl\Big]^2\\
\label{equ-49}
&\leq& C(t-s)^2\sup_{l\in[0,T]}{\|u_n(l)\|^6_{\mathbb{H}^1}}+C(t-s)\int^t_s\|u_n(l)\|^2_{\mathbb{H}^2}dl,
\end{eqnarray}
which implies that
\begin{eqnarray}\label{equ-50}
\int^T_0\|J^2_n(t)\|^2_{\mathbb{H}^0}dt
\leq CT^3\sup_{l\in[0,T]}{\|u_n(l)\|^6_{\mathbb{H}^1}}+CT^2\int^T_0\|u_n(l)\|^2_{\mathbb{H}^2}dl.
\end{eqnarray}
Hence, using (\ref{equ-45})-(\ref{equ-45-1}) and (\ref{equ-49})-(\ref{equ-50}), for $\alpha\in (0,\frac{1}{2})$, we have
\begin{eqnarray*}
\sup_{n\geq 1}\|J^2_n\|^2_{W^{\alpha,2}([0,T];\mathbb{H}^0)}\leq C_2(\alpha).
\end{eqnarray*}
Using the same method as the proof of (4.20) in \cite{Z-Z}, we get
\begin{eqnarray*}
\sup_{n\geq 1}\|J^3_n\|^2_{W^{\alpha,2}([0,T];\mathbb{H}^0)}\leq C_3(\alpha).
\end{eqnarray*}
Collecting all the above estimates, we deduce the result (\ref{equ-46}).

Based on (\ref{equ-46}) and by $u_n\in L^2([0,T];\mathbb{H}^2)$, it follows from Lemma \ref{lem-5} that
$u_n(t)$ is compact in $L^2([0,T];\mathbb{H}^1)$. Thus, there exists an element $u\in L^2([0,T];\mathbb{H}^2)\cap L^{\infty}([0,T];\mathbb{H}^1)$ and a subsequence $u_{m'}$, $m'\rightarrow \infty$, such that
\begin{description}
  \item[1.] $u_{m'}\rightarrow u$ \ weakly\ in\ $L^2([0,T];\mathbb{H}^2)$,
  \item[2.] $u_{m'}\rightarrow u $\ in\ the\ weak-star\ topology\ of\ $L^{\infty}([0,T];\mathbb{H}^1)$,
  \item[3.] $u_{m'}\rightarrow u $ strongly\ in\ $L^2([0,T];\mathbb{H}^1)$.
\end{description}
Finally, we show that $u$ is the unique solution of (\ref{equ-41}).
We will use the similar arguments as in the proof of Theorem 4.1 in \cite{Z-Z}.

Let $\psi$ be a continuously differential function defined on $[0,T]$ with $\psi(T)=0$. Recall $\{e_j\}_{j\geq 1}$ is an orthonormal eigenfunction of $\mathbb{H}^0$. Multiplying (\ref{equ-43})
by $\psi(t)e_j$ and using integration by parts, we obtain
\begin{eqnarray*}
&&-\int^T_0\langle u_n(t), \psi'(t)e_j\rangle_{\mathbb{H}^0} dt+\int^T_0\langle Au_n(t),\psi(t)e_j\rangle_{\mathbb{H}^0} dt\\
&=&\langle  u_n(0), \psi(0)e_j\rangle_{\mathbb{H}^0} -\int^T_0\langle \Pi_n B_n(u_n(t),u_n(t)),\psi(t)e_j\rangle_{\mathbb{H}^{0}} dt\\
&&\ -\int^T_0\langle \Pi_n \mathcal{P}g_N(|u_n(t)|^2)u_n(t),\psi(t)e_j\rangle_{\mathbb{H}^{0}} dt\\
&& \
+\int^T_0\langle \int_{\mathbb{Z}}\sigma_n(t, u_n(t),z)(g(t,z)-1)\vartheta(dz), \psi(t)e_j\rangle_{\mathbb{H}^0} dt.
\end{eqnarray*}
Recall the definition of $B_n$ and (\ref{equ-45}),
for every $n>\sup_{m\in \mathbb{N}^+}\sup_{t\in[0,T]}\|u_m(t)\|^2_{\mathbb{H}^1}\vee j$, we arrive at
\begin{eqnarray*}
&&-\int^T_0\langle u_n(t), \psi'(t)e_j\rangle_{\mathbb{H}^0} dt+\int^T_0\langle Au_n(t),\psi(t)e_j\rangle_{\mathbb{H}^0} dt\\
&=&\langle  u_n(0), \psi(0)e_j\rangle_{\mathbb{H}^0} -\int^T_0\langle \Pi_nB_n(u_n(t),u_n(t)),\psi(t)e_j\rangle_{\mathbb{H}^{0}} dt\\
&&\ -\int^T_0\langle  \Pi_n\mathcal{P}g_N(|u_n(t)|^2)u_n(t),\psi(t)e_j\rangle_{\mathbb{H}^{0}} dt\\
&& \
+\int^T_0\langle \int_{\mathbb{Z}}\sigma_n(t, u_n(t),z)(g(t,z)-1)\vartheta(dz), \psi(t)e_j\rangle_{\mathbb{H}^0} dt.
\end{eqnarray*}
In the following, we devote to proving that as $n\rightarrow \infty$, it holds that
\begin{eqnarray*}
&&-\int^T_0\langle u(t), \psi'(t)e_j\rangle_{\mathbb{H}^0} dt+\int^T_0\langle A u(t),\psi(t)e_j\rangle_{\mathbb{H}^0} dt\\
&=&\langle  u(0), \psi(0)e_j\rangle_{\mathbb{H}^0} -\int^T_0\langle B(u(t),u(t)),\psi(t)e_j\rangle_{\mathbb{H}^{0}} dt\\
&&\ -\int^T_0\langle \mathcal{P} g_N(|u(t)|^2)u(t),\psi(t)e_j\rangle_{\mathbb{H}^{0}} dt\\
&& \
+\int^T_0\langle \int_{\mathbb{Z}}\sigma(t, u(t),z)(g(t,z)-1)\vartheta(dz), \psi(t)e_j\rangle_{\mathbb{H}^0} dt.
\end{eqnarray*}
Since $u_{m'}\rightarrow u$ strongly in $L^2([0,T];\mathbb{H}^1)$ as $m'\rightarrow \infty$, we deduce that
\begin{eqnarray*}
\Big|\int^T_0\langle u_{m'}(t)-u(t), \psi'(t)e_j\rangle_{\mathbb{H}^0} dt\Big|
&\leq& C\int^T_0\|u_{m'}(t)-u(t)\|^2_{\mathbb{H}^0}dt\\
&\leq& C\int^T_0\|u_{m'}(t)-u(t)\|^2_{\mathbb{H}^1}dt\rightarrow 0,
\end{eqnarray*}
and
\begin{eqnarray*}
\langle  u_{m'}(0)-u(0), \psi(0)e_j\rangle_{\mathbb{H}^0}\leq C\|u_{m'}(0)-u(0)\|^2_{\mathbb{H}^0}\rightarrow 0,\quad {\rm{as}}\quad  m'\rightarrow \infty.
\end{eqnarray*}
Moreover, by $u_{m'}\rightarrow u$ strongly in $L^2([0,T];\mathbb{H}^1)$, we have
\begin{eqnarray*}
\int^T_0\langle u_{m'}(t)-u(t),\psi(t)e_j\rangle_{\mathbb{H}^0} dt\leq C\int^T_0\|u_{m'}(t)-u(t)\|_{\mathbb{H}^1}dt\rightarrow 0,\quad {\rm{as}}\quad  m'\rightarrow \infty.
\end{eqnarray*}
Since $u_{m'}\rightarrow u$ strongly in $L^2([0,T];\mathbb{H}^1)$ and by using (\ref{equ-45}), it follows that
\begin{eqnarray*}
&&\Big|\int^T_0\langle  B(u_{m'}(t),u_{m'}(t))-B(u(t),u(t)),\psi(t)e_j\rangle_{\mathbb{H}^{0}} dt\Big|\\
&=&\Big|\int^T_0\langle  B(u_{m'},u_{m'}-u)+B(u_{m'}-u,u),\psi(t)e_j\rangle_{\mathbb{H}^{0}} dt\Big|\\
&\leq& \int^T_0[\|\psi(t)e_j\|_{L^2}\|u_{m'}-u\|_{\mathbb{H}^1}\|u_{m'}\|_{L^{\infty}}+\|\psi(t)e_j\|_{L^2}\|u\|^{\frac{1}{2}}_{\mathbb{H}^1}\|u\|^{\frac{1}{2}}_{\mathbb{H}^2}\|u_{m'}-u\|_{\mathbb{H}^1}]dt\\
&\leq& C\int^T_0[\|u_{m'}\|^{\frac{1}{2}}_{\mathbb{H}^1}\|u_{m'}\|^{\frac{1}{2}}_{\mathbb{H}^2}\|u_{m'}-u\|_{\mathbb{H}^1}+\|u\|^{\frac{1}{2}}_{\mathbb{H}^1}\|u\|^{\frac{1}{2}}_{\mathbb{H}^2}\|u_{m'}-u\|_{\mathbb{H}^1}]dt\\
&\leq& C \sup_{t\in[0,T]}\|u_{m'}(t)\|^{\frac{1}{2}}_{\mathbb{H}^1}\int^T_0\|u_{m'}\|^{\frac{1}{2}}_{\mathbb{H}^2}\|u_{m'}-u\|_{\mathbb{H}^1}dt\\
&&\ +C\sup_{t\in[0,T]}\|u(t)\|^{\frac{1}{2}}_{\mathbb{H}^1}\int^T_0\|u\|^{\frac{1}{2}}_{\mathbb{H}^2}\|u_{m'}-u\|_{\mathbb{H}^1}dt\\
&\leq& C \sup_{t\in[0,T]}\|u_{m'}(t)\|^{\frac{1}{2}}_{\mathbb{H}^1}\Big(\int^T_0\|u_{m'}(t)\|_{\mathbb{H}^2}dt\Big)^{\frac{1}{2}}\Big(\int^T_0\|u_{m'}(t)-u(t)\|^2_{\mathbb{H}^1}dt\Big)^{\frac{1}{2}}\\
&&\ + C \sup_{t\in[0,T]}\|u(t)\|^{\frac{1}{2}}_{\mathbb{H}^1}\Big(\int^T_0\|u(t)\|_{\mathbb{H}^2}dt\Big)^{\frac{1}{2}}\Big(\int^T_0\|u_{m'}(t)-u(t)\|^2_{\mathbb{H}^1}dt\Big)^{\frac{1}{2}}\\
&\rightarrow& 0, \quad {\rm{as}}\quad  m'\rightarrow \infty.
\end{eqnarray*}
Clearly, we have
\begin{eqnarray}\notag
&&\int^T_0\langle \mathcal{P}g_N(|u_{m'}|^2)u_{m'}-\mathcal{P}g_N(|u|^2)u,\psi(t)e_j\rangle_{\mathbb{H}^{0}} dt\\ \notag
&\leq& C\int^T_0\|g_N(|u_{m'}|^2)u_{m'}-g_N(|u|^2)u\|_{\mathbb{H}^0}ds\\
\label{equ-75}
&\leq& C\int^T_0\|g_N(|u_{m'}|^2)(u_{m'}-u)\|_{\mathbb{H}^0}ds+C\int^T_0\|(g_N(|u_{m'}|^2)-g_N(|u|^2))u\|_{\mathbb{H}^0}ds.
\end{eqnarray}
By using $\sup_{x}|u(x)|^2\leq C\|u\|_{\mathbb{H}^1}\|u\|_{\mathbb{H}^2}$, we get
\begin{eqnarray}\notag
&&\|g_N(|u_{m'}|^2)(u_{m'}-u)\|_{\mathbb{H}^0}\\ \notag
&\leq&\Big(\int_{\mathbb{T}^3}|u^2_{m'}(x)+N|^2|u_{m'}-u|^2dx\Big)^{\frac{1}{2}}\\ \notag
&\leq& \Big(\sup_{x\in \mathbb{T}^3}|u_{m'}(x)|^2-N\Big)\Big(\int_{\mathbb{T}^3}|u_{m'}-u|^2dx\Big)^{\frac{1}{2}}\\
\label{equ-76}
&\leq&C (\|u_{m'}\|_{\mathbb{H}^1}\|u_{m'}\|_{\mathbb{H}^2}+N)\|u_{m'}-u\|_{\mathbb{H}^0}.
\end{eqnarray}
Moreover, we deduce that
\begin{eqnarray}\notag
&&\|(g_N(|u_{m'}|^2)-g_N(|u|^2))u\|_{\mathbb{H}^0}\\ \notag
&\leq& C\Big(\int_{\mathbb{T}^3}(|u_{m'}|+|u|)^2|u_{m'}-u|^2|u|^2dx\Big)^{\frac{1}{2}}\\
\label{equ-78}
&\leq& C\sup_{x}(|u_{m'}(x)|+|u(x)|)\Big(\int_{\mathbb{T}^3}|u_{m'}-u|^2|u|^2dx\Big)^{\frac{1}{2}}.
\end{eqnarray}
By the H\"{o}lder inequality and the Sobolev embedding inequality, we get
\begin{eqnarray}\notag
&&\Big(\int_{\mathbb{T}^3}|u_{m'}-u|^2|u|^2dx\Big)^{\frac{1}{2}}\\ \notag
&\leq& \||u|^2\|^{\frac{1}{2}}_{L^2}\|u_{m'}-u\|^{\frac{1}{2}}_{L^3}\|u_{m'}-u\|^{\frac{1}{2}}_{L^6}\\ \notag
&\leq& \|u\|_{L^4}\|u_{m'}-u\|^{\frac{1}{2}}_{L^3}\|u_{m'}-u\|^{\frac{1}{2}}_{L^6}\\
\label{equ-79}
&\leq& C \|u\|^{\frac{1}{4}}_{\mathbb{H}^0}\|u\|^{\frac{3}{4}}_{\mathbb{H}^1}\|u_{m'}-u\|^{\frac{1}{4}}_{\mathbb{H}^0}\|u_{m'}-u\|^{\frac{3}{4}}_{\mathbb{H}^1}.
\end{eqnarray}
Hence, combining (\ref{equ-78}) and (\ref{equ-79}), it follows that
\begin{eqnarray}\notag
&&\|(g_N(|u_{m'}|^2)-g_N(|u|^2))u\|_{\mathbb{H}^0}\\ \notag
&\leq& C\sup_{x}(|u_{m'}(x)|+|u(x)|)\|u\|^{\frac{1}{4}}_{\mathbb{H}^0}\|u\|^{\frac{3}{4}}_{\mathbb{H}^1}\|u_{m'}-u\|^{\frac{1}{4}}_{\mathbb{H}^0}\|u_{m'}-u\|^{\frac{3}{4}}_{\mathbb{H}^1}\\
\notag
&\leq& C(\|u_{m'}\|^{\frac{1}{2}}_{\mathbb{H}^1}\|u_{m'}\|^{\frac{1}{2}}_{\mathbb{H}^2}+\|u\|^{\frac{1}{2}}_{\mathbb{H}^1}\|u\|^{\frac{1}{2}}_{\mathbb{H}^2})\|u\|^{\frac{1}{4}}_{\mathbb{H}^0}\|u\|^{\frac{3}{4}}_{\mathbb{H}^1}\|u_{m'}-u\|^{\frac{1}{4}}_{\mathbb{H}^0}\|u_{m'}-u\|^{\frac{3}{4}}_{\mathbb{H}^1}\\
\notag
&\leq& C\|u_{m'}\|^{\frac{1}{2}}_{\mathbb{H}^1}\|u_{m'}\|^{\frac{1}{2}}_{\mathbb{H}^2}\|u\|^{\frac{1}{4}}_{\mathbb{H}^0}\|u\|^{\frac{3}{4}}_{\mathbb{H}^1}\|u_{m'}-u\|^{\frac{1}{4}}_{\mathbb{H}^0}\|u_{m'}-u\|^{\frac{3}{4}}_{\mathbb{H}^1}\\
\label{equ-77}
&&\ +\|u\|^{\frac{1}{4}}_{\mathbb{H}^0}\|u\|^{\frac{5}{4}}_{\mathbb{H}^1}\|u\|^{\frac{1}{2}}_{\mathbb{H}^2}\|u_{m'}-u\|^{\frac{1}{4}}_{\mathbb{H}^0}\|u_{m'}-u\|^{\frac{3}{4}}_{\mathbb{H}^1}.
\end{eqnarray}
Collecting (\ref{equ-75})-(\ref{equ-77}) and using the H\"{o}lder inequality, we deduce that
\begin{eqnarray*}\notag
&&\int^T_0\langle \mathcal{P}g_N(|u_{m'}|^2)u_{m'}-\mathcal{P}g_N(|u|^2)u,\psi(t)e_j\rangle_{\mathbb{H}^{0}} dt\\ \notag
&\leq& C\sup_{t\in [0,T]}\|u_{m'}(t)\|_{\mathbb{H}^1}\Big(\int^T_0\|u_{m'}(t)\|^2_{\mathbb{H}^2}dt\Big)^{\frac{1}{2}}\Big(\int^T_0\|u_{m'}(t)-u(t)\|^2_{\mathbb{H}^0}dt\Big)^{\frac{1}{2}}+CN\int^T_0\|u_{m'}(t)-u(t)\|_{\mathbb{H}^0}dt\\
&&\ +C\int^T_0\|u\|^{\frac{1}{4}}_{\mathbb{H}^0}\|u\|^{\frac{3}{4}}_{\mathbb{H}^1}\|u_{m'}(t)\|^{\frac{1}{2}}_{\mathbb{H}^1}\|u_{m'}\|^{\frac{1}{2}}_{\mathbb{H}^2}\|u_{m'}-u\|^{\frac{1}{4}}_{\mathbb{H}^0}\|u_{m'}-u\|^{\frac{3}{4}}_{\mathbb{H}^1}dt\\
&&\ +C\int^T_0\|u\|^{\frac{1}{4}}_{\mathbb{H}^0}\|u\|^{\frac{5}{4}}_{\mathbb{H}^1}\|u\|^{\frac{1}{2}}_{\mathbb{H}^2}\|u_{m'}-u\|^{\frac{1}{4}}_{\mathbb{H}^0}\|u_{m'}-u\|^{\frac{3}{4}}_{\mathbb{H}^1}dt\\
&\leq& C\sup_{t\in [0,T]}\|u_{m'}(t)\|_{\mathbb{H}^1}\Big(\int^T_0\|u_{m'}(t)\|^2_{\mathbb{H}^2}dt\Big)^{\frac{1}{2}}\Big(\int^T_0\|u_{m'}(t)-u(t)\|^2_{\mathbb{H}^0}dt\Big)^{\frac{1}{2}}+CN\int^T_0\|u_{m'}(t)-u(t)\|_{\mathbb{H}^0}dt\\
&&\ +C\sup_{t\in [0,T]}\|u(t)\|^{\frac{1}{4}}_{\mathbb{H}^0}\|u(t)\|^{\frac{3}{4}}_{\mathbb{H}^1}\|u_{m'}(t)\|^{\frac{1}{2}}_{\mathbb{H}^1}\Big(\int^T_0\|u_{m'}\|^{2}_{\mathbb{H}^2}dt\Big)^{\frac{1}{4}}\Big(\int^T_0\|u_{m'}-u\|^{\frac{2}{3}}_{\mathbb{H}^0}dt\Big)^{\frac{3}{8}}\Big(\int^T_0\|u_{m'}-u\|^{2}_{\mathbb{H}^1}dt\Big)^{\frac{3}{8}}\\
&&\ +C\sup_{t\in [0,T]}\|u(t)\|^{\frac{1}{4}}_{\mathbb{H}^0}\|u(t)\|^{\frac{5}{4}}_{\mathbb{H}^1}\Big(\int^T_0\|u(t)\|^{2}_{\mathbb{H}^2}dt\Big)^{\frac{1}{4}}\Big(\int^T_0\|u_{m'}-u\|^{\frac{2}{3}}_{\mathbb{H}^0}dt\Big)^{\frac{3}{8}}\Big(\int^T_0\|u_{m'}(t)-u(t)\|^{2}_{\mathbb{H}^1}dt\Big)^{\frac{3}{8}}\\
&&\ \rightarrow 0,\quad {\rm{as}} \quad m'\rightarrow \infty,
\end{eqnarray*}
where $(\ref{equ-45})$, $u_{m'}\rightarrow u$ strongly in $L^2([0,T];\mathbb{H}^1)$ and  $u\in L^2([0,T];\mathbb{H}^2)\cap L^{\infty}([0,T];\mathbb{H}^1) $ are used.

Using similar method as (4.29) in \cite{Z-Z}, we get
\begin{eqnarray*}
\lim_{m'\rightarrow \infty}\int^T_0 \int_{\mathbb{Z}}\|\sigma(t, u_{m'}(t),z)(g(t,z)-1)-\sigma(t, u(t),z)(g(t,z)-1)\|_{\mathbb{H}^0} \vartheta(dz)dt=0,
\end{eqnarray*}
which implies
\begin{eqnarray*}
&&\int^T_0\langle \int_{\mathbb{Z}}\sigma(t, u_{m'}(t),z)(g(t,z)-1)\vartheta(dz), \psi(t)e_j\rangle_{\mathbb{H}^0} dt\\
&\rightarrow& \int^T_0\langle \int_{\mathbb{Z}}\sigma(t, u(t),z)(g(t,z)-1)\vartheta(dz), \psi(t)e_j\rangle_{\mathbb{H}^0} dt,
\end{eqnarray*}
as $m'\rightarrow \infty$.
Based on the above steps, we conclude that for any $j\geq 1$,
\begin{eqnarray}\notag
&&-\int^T_0\langle u(t), \psi'(t)e_j\rangle_{\mathbb{H}^0} dt+\int^T_0\langle Au(t),\psi(t)e_j\rangle_{\mathbb{H}^0} dt\\ \notag
&=&\langle  u(0), \psi(0)e_j\rangle_{\mathbb{H}^0} -\int^T_0\langle B(u(t),u(t)),\psi(t)e_j\rangle_{\mathbb{H}^{0}} dt\\ \notag
&&\ -\int^T_0\langle  \mathcal{P}g_N(|u(t)|^2)u(t),\psi(t)e_j\rangle_{\mathbb{H}^{0}} dt\\
\label{equ-53}
&& \
+\int^T_0\langle \int_{\mathbb{Z}}\sigma(t, u(t),z)(g(t,z)-1)\vartheta(dz), \psi(t)e_j\rangle_{\mathbb{H}^0} dt.
\end{eqnarray}
Actually, (\ref{equ-53}) holds for any $\zeta\in \mathbb{H}^0$, which is a finite linear combination of $e_j$. That is
\begin{eqnarray}\notag
&&-\int^T_0\langle u(t), \psi'(t)\zeta\rangle_{\mathbb{H}^0} dt+\int^T_0\langle Au(t),\psi(t)\zeta\rangle_{\mathbb{H}^0} dt\\ \notag
&=&\langle  u(0), \psi(0)\zeta\rangle_{\mathbb{H}^0} -\int^T_0\langle B(u(t),u(t)),\psi(t)\zeta\rangle_{\mathbb{H}^{0}} dt\\ \notag
&&\ -\int^T_0\langle  g_N(|u(t)|^2)u(t),\psi(t)\zeta\rangle_{\mathbb{H}^{0}} dt\\
\label{equ-54}
&& \
+\int^T_0\langle \int_{\mathbb{Z}}\sigma(t, u(t),z)(g(t,z)-1)\vartheta(dz), \psi(t)\zeta\rangle_{\mathbb{H}^0} dt.
\end{eqnarray}
As a result, we obtain
\begin{eqnarray}\label{eq-55}
du(t)+A u(t)dt+B(u(t))dt+ \mathcal{P}g_N(|u(t)|^2)u(t)dt=\int_{\mathbb{Z}}\sigma(t, u(t),z)(g(t,z)-1)\vartheta(dz)dt
\end{eqnarray}
holds as an equality in distribution in $L^2([0,T];\mathbb{H}^{0})$.

From here, using similar arguments as in the proof of Theorem 3.1 in Temam \cite{Temam-1}, we can conclude that $u$ is the desired solution.

By Lemma \ref{lem-3} and using the same arguments as in the proof of Theorem 3.2 in Temam \cite{Temam-1}, we can obtain
\[
\frac{du}{dt}\in L^2([0,T];\mathbb{H}^0).
\]
Since $u\in L^2([0,T];\mathbb{H}^2)$, we deduce from Lemma \ref{lem-6} that $u\in C([0,T];\mathbb{H}^1)$.

\textbf{(Uniqueness)}\ Suppose $u_1, u_2$ are two solutions of  (\ref{equ-41}). Let $u=u_1-u_2$ and $u(0)=u_1(0)-u_2(0)$, we have
\begin{eqnarray*}
du(t)=(F(u_1)-F(u_2))dt+\int_{\mathbb{Z}}(\sigma(t,u_1,z)-\sigma(t,u_2,z))(g(t,z)-1)\vartheta(dz)dt.
\end{eqnarray*}
Now, we make $L^2$  estimates of $u(t)$ as follows.
By the chain rule, it gives that
\begin{eqnarray*}
\|u(t)\|^2_{\mathbb{H}^0}&=&\|u(0)\|^2_{\mathbb{H}^0}+2\int^t_0\langle u(s), F(u_1)-F(u_2)\rangle_{\mathbb{H}^0}ds\\
&&\ +2\int^t_0\langle u(s), \int_{\mathbb{Z}}(\sigma(s,u_1,z)-\sigma(s,u_2,z))(g(s,z)-1)\vartheta(dz)\rangle_{\mathbb{H}^0}ds\\
&=:& \|u(0)\|^2_{\mathbb{H}^0}+I_1(t)+I_2(t).
\end{eqnarray*}
By (\ref{equ-18}), we obtain
\begin{eqnarray*}
I_1(t)&\leq& -\int^t_0\|u(s)\|^2_{\mathbb{H}^1}ds+2C_0\int^t_0(\|u_2\|_{\mathbb{H}^1}\|u_2\|_{\mathbb{H}^2}+1)\|u(s)\|^2_{\mathbb{H}^0}ds,\\
I_2(t)&\leq& 2\int^t_0\int_{\mathbb{Z}}\|\sigma(s,u_1,z)-\sigma(s,u_2,z)\|_{\mathbb{H}^0}|g(s,z)-1|\|u(s)\|_{\mathbb{H}^0}\vartheta(dz)ds\\
&\leq& 2\int^t_0\|u(s)\|^2_{\mathbb{H}^0}\Big(\int_{\mathbb{Z}}\|\sigma(s,z)\|_{1, \mathbb{H}^0}|g(s,z)-1|\vartheta(dz)\Big)ds.
\end{eqnarray*}
Collecting all the above estimates, we get
\begin{eqnarray*}
\|u(t)\|^2_{\mathbb{H}^0}+\int^t_0\|u(s)\|^2_{\mathbb{H}^1}ds&\leq&\|u(0)\|^2_{\mathbb{H}^0} +2C_0\int^t_0(\|u_2\|_{\mathbb{H}^1}\|u_2\|_{\mathbb{H}^2}+1)\|u(s)\|^2_{\mathbb{H}^0}ds\\
&&\ +2\int^t_0\|u(s)\|^2_{\mathbb{H}^0}\Big(\int_{\mathbb{Z}}\|\sigma(s,z)\|_{1, \mathbb{H}^0}|g(s,z)-1|\vartheta(dz)\Big)ds.
\end{eqnarray*}
By Gronwall inequality, we deduce that
\begin{eqnarray*}
&&\|u(t)\|^2_{\mathbb{H}^0}\\
&\leq&\|u(0)\|^2_{\mathbb{H}^0} \exp\Big\{2C_0\int^t_0(\|u_2\|_{\mathbb{H}^1}\|u_2\|_{\mathbb{H}^2}+1)ds+2\int^t_0\int_{\mathbb{Z}}\|\sigma(s,z)\|_{1, \mathbb{H}^0}|g(s,z)-1|\vartheta(dz)ds\Big\}\\
&\leq& \|u(0)\|^2_{\mathbb{H}^0}\exp\Big\{2C_0t+2C_0\sup_{t\in [0,T]}\|u_2(t)\|_{\mathbb{H}^1}\int^t_0\|u_2(s)\|_{\mathbb{H}^2}ds+2\int^t_0\int_{\mathbb{Z}}\|\sigma(s,z)\|_{1, \mathbb{H}^0}|g(s,z)-1|\vartheta(dz)ds\Big\}.
\end{eqnarray*}
Since the two solutions $u_i$ of (\ref{equ-41}) are in the state space $ C([0,T];\mathbb{H}^1)\cap L^2([0,T];\mathbb{H}^2)$ for $i=1,2$, we deduce that $u_1=u_2$ if $u_1(0)=u_2(0)$.

\end{proof}

\section{Large deviations}
This section is devoted to the proof of the main result (Theorem \ref{thm-3}).
According to Theorem \ref{thm-1}, we need to prove (i) and (ii) in \textbf{Condition A}.

Firstly, we prove (i) in \textbf{Condition A}.
For $g\in S$, from  Theorem \ref{thm-2}, we can define
\begin{eqnarray*}
\mathcal{G}^0(\vartheta^g_T):=u^g.
\end{eqnarray*}
\begin{prp}\label{prp-1}
For any $M\in \mathbb{N}^+$, and $\{g_n\}_{n\geq 1}\subset S^M, g\in S^M$ satisfying $g_n\rightarrow g$ as $n\rightarrow \infty$. Then
\begin{eqnarray*}
\mathcal{G}^0(\vartheta^{g_n}_T)\rightarrow \mathcal{G}^0(\vartheta^{g}_T) \quad {\rm{in}}\quad C([0,T];\mathbb{H}^1).
\end{eqnarray*}
\end{prp}
\begin{proof}
Recall that $\mathcal{G}^0(\vartheta^{g_n}_T)=u^{g_n}$. For simplicity, denote $u^n=u^{g_n}$.

Using similar method as Theorem \ref{thm-1}, we obtain that there exist constants $C_1(M), C_2(M)$ and $C_{\alpha, M}$ such that
\begin{description}
  \item[1.] $
\sup_{s\in [0,T]}\|u^n(s)\|^2_{\mathbb{H}^1}+\int^T_0\|u^n(s)\|^{2}_{\mathbb{H}^2}ds\leq C_1(M),$
  \item[2.] $
\sup_{s\in [0,T]}\|u^n(s)\|^6_{\mathbb{H}^1}+\int^T_0\|u^n(s)\|^4_{\mathbb{H}^1}\|u^n(s)\|^{2}_{\mathbb{H}^2}ds\leq C_2(M),$
\item[3.] $
\|u^n\|^2_{W^{\alpha,2}([0,T];\mathbb{H}^0)}\leq C_{\alpha, M}, \ \alpha\in (0,\frac{1}{2}).
$
\end{description}
Hence, by Lemma \ref{lem-5}, we can assert the existence of an element $u\in L^2([0,T];\mathbb{H}^2)\cap L^{\infty}([0,T];\mathbb{H}^1)$ and a subsequence $u^{m'}$ such that
\begin{description}
  \item[(a)] $u^{m'}\rightarrow u$ \ weakly\ in\ $L^2([0,T];\mathbb{H}^2)$,
  \item[(b)] $u^{m'}\rightarrow u$ \ weak-star\ in\ $L^{\infty}([0,T];\mathbb{H}^1)$,
  \item[(c)] $u^{m'}\rightarrow u$ \ strongly\ in \ $L^2([0,T];\mathbb{H}^1)$.
\end{description}
We will prove that $u=u^g=\mathcal{G}^0(\vartheta^{g}_T)$.

Let $\psi$ be a continuously differential function defined on $[0,T]$ with $\psi(T)=0$. Multiplying the equation (\ref{equ-41}) satisfied by $u^{m'}(t)$ by $\psi(t)e_j$ and using integration by parts, we obtain
\begin{eqnarray*}
&&-\int^T_0\langle u^{m'}(t), \psi'(t)e_j\rangle_{\mathbb{H}^0} dt+\int^T_0\langle Au^{m'}(t),\psi(t)e_j\rangle_{\mathbb{H}^0} dt\\
&=&\langle  u_0, \psi(0)e_j\rangle_{\mathbb{H}^0} -\int^T_0\langle B(u^{m'}(t),u^{m'}(t)),\psi(t)e_j\rangle_{\mathbb{H}^{0}} dt\\
&&\ -\int^T_0\langle  \mathcal{P}g_N(|u^{m'}(t)|^2)u^{m'}(t),\psi(t)e_j\rangle_{\mathbb{H}^{0}} dt\\
&& \
+\int^T_0\langle \int_{\mathbb{Z}}\sigma(t, u^{m'}(t),z)(g_{m'}(t,z)-1)\vartheta(dz), \psi(t)e_j\rangle_{\mathbb{H}^0} dt.
\end{eqnarray*}
Utilizing the same method as Theorem \ref{thm-2}, we deduce that
\begin{eqnarray*}
&&-\int^T_0\langle u^{m'}(t), \psi'(t)e_j\rangle_{\mathbb{H}^0} dt+\int^T_0\langle Au^{m'}(t),\psi(t)e_j\rangle_{\mathbb{H}^0} dt\\
&&\ -\langle  u_0, \psi(0)e_j\rangle_{\mathbb{H}^0} +\int^T_0\langle B(u^{m'}(t),u^{m'}(t)),\psi(t)e_j\rangle_{\mathbb{H}^{0}} dt\\
&&\ +\int^T_0\langle  \mathcal{P}g_N(|u^{m'}(t)|^2)u^{m'}(t),\psi(t)e_j\rangle_{\mathbb{H}^{0}} dt\\
& \rightarrow& -\int^T_0\langle u(t), \psi'(t)e_j\rangle_{\mathbb{H}^0} dt+\int^T_0\langle Au(t),\psi(t)e_j\rangle_{\mathbb{H}^0} dt\\
&&\ -\langle  u_0, \psi(0)e_j\rangle_{\mathbb{H}^0} +\int^T_0\langle B(u(t),u(t)),\psi(t)e_j\rangle_{\mathbb{H}^{0}} dt\\
&&\ +\int^T_0\langle \mathcal{P}g_N(|u(t)|^2)u(t),\psi(t)e_j\rangle_{\mathbb{H}^{0}} dt.
\end{eqnarray*}
For the remain term $\int^T_0\langle \int_{\mathbb{Z}}\sigma(t, u^n(t),z)(g_n(t,z)-1)\vartheta(dz), \psi(t)e_j\rangle_{\mathbb{H}^0} dt$, applying the same method as Proposition 4.1 in \cite{Z-Z}, and by using Lemma \ref{lem-2} and Lemma \ref{lem-4}, it gives that
\begin{eqnarray*}
&&\int^T_0\langle \int_{\mathbb{Z}}\sigma(t, u^{m'}(t),z)(g_{m'}(t,z)-1)\vartheta(dz), \psi(t)e_j\rangle_{\mathbb{H}^0} dt\\
& \rightarrow &  \int^T_0\langle \int_{\mathbb{Z}}\sigma(t, u(t),z)(g(t,z)-1)\vartheta(dz), \psi(t)e_j\rangle_{\mathbb{H}^0} dt.
\end{eqnarray*}
Therefore, we get
\begin{eqnarray*}
&&-\int^T_0\langle u(t), \psi'(t)e_j\rangle_{\mathbb{H}^0} dt+\int^T_0\langle Au(t),\psi(t)e_j\rangle_{\mathbb{H}^0} dt\\
&=&\langle  u_0, \psi(0)e_j\rangle_{\mathbb{H}^0} -\int^T_0\langle B(u(t),u(t)),\psi(t)e_j\rangle_{\mathbb{H}^{0}} dt\\
&&\ -\int^T_0\langle  \mathcal{P}g_N(|u(t)|^2)u(t),\psi(t)e_j\rangle_{\mathbb{H}^{0}} dt\\
&& \
+\int^T_0\langle \int_{\mathbb{Z}}\sigma(t, u(t),z)(g(t,z)-1)\vartheta(dz), \psi(t)e_j\rangle_{\mathbb{H}^0} dt.
\end{eqnarray*}
Based on the above, applying the same method as the proof of Theorem 3.1 in \cite{Temam-1}, we obtain $u=u^g$.

In the following, we aim to prove $u^{m'}\rightarrow u$ in $C([0,T];\mathbb{H}^1)$.
Let $Z^{m'}=u^{m'}-u$, then we have
\begin{eqnarray*}
dZ^{m'}(t)=(F(u^{m'})-F(u))dt+\int_{\mathbb{Z}}[\sigma(t,u^{m'}(t),z)(g_{m'}(t,z)-1)-\sigma(t,u,z)(g(t,z)-1)]\vartheta(dz)dt,
\end{eqnarray*}
with $Z^{m'}(0)=0$.

By the chain rule, we get
\begin{eqnarray*}
\|Z^{m'}(t)\|^2_{\mathbb{H}^1}&=&2\int^t_0\langle Z^{m'}(s),F(u^{m'})-F(u) \rangle_{\mathbb{H}^1}ds\\
&&\ +2\int^t_0\langle Z^{m'}(s),\int_{\mathbb{Z}}[\sigma(s,u^{m'}(s),z)(g_{m'}(s,z)-1)-\sigma(s,u,z)(g(s,z)-1)]\vartheta(dz)\rangle_{\mathbb{H}^1}ds\\
&=:& J^{m'}_1(t)+J^{m'}_2(t).
\end{eqnarray*}
By (\ref{equ-55}), it follows that
\begin{eqnarray*}
J^{m'}_1(t)\leq -\frac{1}{2}\int^t_0\|Z^{m'}(s)\|^2_{\mathbb{H}^2}ds+C\int^t_0(1+\|u^{m'}\|^4_{\mathbb{H}^1}+\|u\|^4_{\mathbb{H}^1}+\|u\|^2_{\mathbb{H}^2})\|Z^{m'}(s)\|^2_{\mathbb{H}^1}ds,
\end{eqnarray*}
and
\begin{eqnarray*}
J^{m'}_2(t)&=& 2\int^t_0\langle Z^{m'}(s),\int_{\mathbb{Z}}[\sigma(s,u^{m'}(s),z)(g_{m'}(s,z)-1)-G(s,u,z)(g(s,z)-1)]\vartheta(dz)\rangle_{\mathbb{H}^1}ds\\
&=& 2\int^t_0\langle Z^{m'}(s),\int_{\mathbb{Z}}[(\sigma(s,u^{m'}(s),z)-\sigma(s,u(s),z))(g(s,z)-1)]\vartheta(dz)\rangle_{\mathbb{H}^1}ds\\
&&\ +2\int^t_0\langle Z^{m'}(s),\int_{\mathbb{Z}}[\sigma(s,u^{m'}(s),z)(g_{m'}(s,z)-1)-\sigma(s,u^{m'}(s),z)(g(s,z)-1)]\vartheta(dz)\rangle_{\mathbb{H}^1}ds\\
&=:& K^{m'}_1(t)+K^{m'}_2(t).
\end{eqnarray*}
It's easy to deduce that
\begin{eqnarray*}
K^{m'}_1(t)&\leq & 2\int^t_0\|Z^{m'}\|_{\mathbb{H}^1}\int_{\mathbb{Z}}\|\sigma(s,u^{m'}(s),z)-\sigma(s,u(s),z))\|_{\mathbb{H}^1}|g(s,z)-1|\vartheta(dz)ds\\
&\leq & 2\int^t_0\|Z^{m'}\|^2_{\mathbb{H}^1}\Big(\int_{\mathbb{Z}}\|\sigma(s,z)\|_{1,\mathbb{H}^1}|g(s,z)-1|\vartheta(dz)\Big)ds.
\end{eqnarray*}
Collecting all the above estimates, we get
\begin{eqnarray*}
\|Z^{m'}(t)\|^2_{\mathbb{H}^1}+\frac{1}{2}\int^t_0\|Z^{m'}(s)\|^2_{\mathbb{H}^2}ds&\leq& C\int^t_0(1+\|u^{m'}\|^4_{\mathbb{H}^1}+\|u\|^4_{\mathbb{H}^1}+\|u\|^2_{\mathbb{H}^2})\|Z^{m'}(s)\|^2_{\mathbb{H}^1}ds\\
&&\ +2\int^t_0\|Z^{m'}\|^2_{\mathbb{H}^1}\Big(\int_{\mathbb{Z}}\|\sigma(s,z)\|_{1,\mathbb{H}^1}|g(s,z)-1|\vartheta(dz)\Big)ds+K^{m'}_2(t).
\end{eqnarray*}
Setting
\[
\psi(s)=C(1+\|u^{m'}\|^4_{\mathbb{H}^1}+\|u\|^4_{\mathbb{H}^1}+\|u\|^2_{\mathbb{H}^2})+2\int_{\mathbb{Z}}\|\sigma(s,z)\|_{1,\mathbb{H}^1}|g(s,z)-1|\vartheta(dz),
\]
we have
\begin{eqnarray*}
\|Z^{m'}(t)\|^2_{\mathbb{H}^1}+\frac{1}{2}\int^t_0\|Z^{m'}(s)\|^2_{\mathbb{H}^2}ds\leq \int^t_0\psi(s)\|Z^{m'}(s)\|^2_{\mathbb{H}^1}ds+K^{m'}_2(t).
\end{eqnarray*}
Hence, we obtain
\begin{eqnarray*}
\exp\Big(-\int^t_0\psi(s)ds\Big)\|Z^{m'}(t)\|^2_{\mathbb{H}^1}\leq \int^T_0K^{m'}_2(s)ds.
\end{eqnarray*}
By Lemma \ref{lem-2}, it follows that
\[
\int^T_0\psi(s)ds<\infty,
\]
which implies that
\begin{eqnarray}\label{equ-56}
\sup_{t\in [0,T]}\|Z^{m'}(t)\|^2_{\mathbb{H}^1}\leq \exp\Big(\int^t_0\psi(s)ds\Big)\int^T_0K^{m'}_2(s)ds.
\end{eqnarray}
Using the same method as (4.61) in \cite{Z-Z}, we have
\[
\int^T_0K^{m'}_2(s)ds\rightarrow 0, \quad {\rm{as}}\quad m'\rightarrow \infty.
\]
Therefore, by (\ref{equ-56}), we get
\[
\lim_{m'\rightarrow \infty}\sup_{t\in [0,T]}\|Z^{m'}(t)\|^2_{\mathbb{H}^1}=0,
\]
which implies the desired result.

\end{proof}

Recall that (\ref{equ-39}) has a unique strong solution $u^{\varepsilon}$ for every $\varepsilon>0$, which defines a measurable mapping $\mathcal{G}^{\varepsilon}: \bar{\mathbb{M}}\rightarrow \mathcal{D}([0,T];\mathbb{H}^1) $ such that, for any Poisson random measure $n^{\varepsilon^{-1}}$ on $[0,T]\times \mathbb{Z}$ with intensity measure $\varepsilon^{-1}\lambda_T\otimes \vartheta$ given on some probability space,
$\mathcal{G}^{\varepsilon}(\varepsilon {n}^{\varepsilon^{-1}})$ is the unique solution of (\ref{equ-39}) with $\tilde{\eta}^{\varepsilon^{-1}}$ replaced by $\tilde{n}^{\varepsilon^{-1}}$.

Let $\varphi_{\varepsilon}\in \mathcal{U}^M$ and $\vartheta_{\varepsilon}=\frac{1}{\varphi_{\varepsilon}}$. The following lemma was proved by Budhiraja et al. \cite{B-D-M}.
\begin{lemma}\label{lem-7} Recall $(\bar{\eta},\bar{\vartheta}_T) $ are introduced in Section \ref{ss-1}.
\begin{eqnarray*}
\mathcal{E}^{\varepsilon}_t(\vartheta_{\varepsilon})&:=&\exp\Big\{\int_{(0,t)\times \mathbb{Z}\times [0,\varepsilon^{-1}]}\log(\vartheta_{\varepsilon}(s,z)) \bar{\eta}(dsdzdr)\\
&&\ +\int_{(0,t)\times \mathbb{Z}\times [0,\varepsilon^{-1}]}(-\vartheta_{\varepsilon}(s,z)+1) \bar{\vartheta}_T(dsdzdr)\Big\},
\end{eqnarray*}
is an $\{\bar{\mathcal{F}}_t\}-$martingale. Then
\[
\mathbb{Q}^{\varepsilon}_t(G)=\int_{G}\mathcal{E}^{\varepsilon}_t(\vartheta_{\varepsilon})d \bar{\mathbb{P}},\quad {\rm{for }}\quad G\in \mathcal{B}(\bar{\mathbb{M}})
\]
defines a probability measure on $\bar{\mathbb{M}}$.
\end{lemma}
Since $\varepsilon {\eta}^{\varepsilon^{-1}\varphi_{\varepsilon}}$ under $\mathbb{Q}^{\varepsilon}_T$ has the same law as that of $\varepsilon {\eta}^{\varepsilon^{-1}}$ under $\bar{\mathbb{P}}$, it follows that
there exists a unique solution to the following controlled stochastic evolution equations $\tilde{u}^{\varepsilon}$:
\begin{eqnarray}\notag
\tilde{u}^{\varepsilon}(t)&=&u_0-\int^t_0 A \tilde{u}^{\varepsilon}(s)ds-\int^t_0 B(\tilde{u}^{\varepsilon}(s),\tilde{u}^{\varepsilon}(s) )ds\\ \notag
&& \ -\int^t_0 \mathcal{P}g_N(|\tilde{u}^{\varepsilon}(s)|^2)\tilde{u}^{\varepsilon}(s)ds+\int^t_0\int_{\mathbb{Z}}\sigma(s,\tilde{u}^{\varepsilon}(s),z)(\varepsilon {\eta}^{\varepsilon^{-1}\varphi_{\varepsilon}}(ds,dz)-\vartheta(dz)ds)\\ \notag
&=& u_0-\int^t_0 A\tilde{u}^{\varepsilon}(s)ds-\int^t_0 B(\tilde{u}^{\varepsilon}(s),\tilde{u}^{\varepsilon}(s))ds\\ \notag
&& \ -\int^t_0 \mathcal{P}g_N(|\tilde{u}^{\varepsilon}(s)|^2)\tilde{u}^{\varepsilon}(s)ds+\int^t_0\int_{\mathbb{Z}}\sigma(s,\tilde{u}^{\varepsilon}(s),z)(\varphi_{\varepsilon}(s,z)-1)\vartheta(dz)ds)\\
\label{equ-58}
&&\ +\int^t_0\int_{\mathbb{Z}}\varepsilon \sigma(s,\tilde{u}^{\varepsilon}(s),z)( {\eta}^{\varepsilon^{-1}\varphi_{\varepsilon}}(ds,dz)-\varepsilon^{-1}\varphi_{\varepsilon}(s,z)\vartheta(dz)ds),
\end{eqnarray}
and we have
\begin{eqnarray}\label{equ-59}
\mathcal{G}^{\varepsilon}(\varepsilon {\eta}^{\varepsilon^{-1}\varphi_{\varepsilon}})=\tilde{u}^{\varepsilon}.
\end{eqnarray}
Before proving (ii) in \textbf{Condition A}, we make a priori estimates of $\tilde{u}^{\varepsilon}$.
\begin{lemma}\label{lem-8}
There exists $\varepsilon_0>0$ such that
\begin{eqnarray}\label{equ-60}
\sup_{0<\varepsilon<\varepsilon_0}\Big[\mathbb{E}\sup_{0\leq t\leq T}\|\tilde{u}^{\varepsilon}\|^2_{\mathbb{H}^1}+\mathbb{E}\int^T_0\|\tilde{u}^{\varepsilon}(t)\|^2_{\mathbb{H}^2}dt\Big]<\infty,\\
\label{equ-60-1}
\sup_{0<\varepsilon<\varepsilon_0}\Big[\mathbb{E}\sup_{0\leq t\leq T}\|\tilde{u}^{\varepsilon}\|^6_{\mathbb{H}^1}+\mathbb{E}\int^T_0\|\tilde{u}^{\varepsilon}(t)\|^2_{\mathbb{H}^2}\|\tilde{u}^{\varepsilon}\|^4_{\mathbb{H}^1}dt\Big]<\infty,
\end{eqnarray}
and for $\alpha\in (0, \frac{1}{2})$, there exists a positive constant $C_{\alpha}$ such that
\begin{eqnarray}\label{equ-61}
\sup_{0<\varepsilon<\varepsilon_0}\mathbb{E}[\|\tilde{u}^{\varepsilon}\|_{W^{\alpha,2}([0,T]; \mathbb{H}^0)}]\leq C_{\alpha}<\infty.
\end{eqnarray}
Thus, the family $\{\tilde{u}^{\varepsilon}; 0<\varepsilon<\varepsilon_0\}$ is tight in $L^2([0,T];\mathbb{H}^1)$.

\end{lemma}

The proof of Lemma \ref{lem-8} is similar to the proof of (\ref{equ-45})-(\ref{equ-46}) and Lemma 4.2 in \cite{Z-Z}, hence, we omit it here.

To get our main results, we need to prove that $\{\tilde{u}^{\varepsilon}\}_{0<\varepsilon<\varepsilon_0}$ is tight in $\mathcal{D}([0,T];D(A^{-\alpha}))$ for some $\alpha\geq0$. Firstly,
we recall the following two lemmas related to the tightness of $\{\tilde{u}^{\varepsilon};{0<\varepsilon<\varepsilon_0}\}$. The proof can be found in \cite{J} and \cite{A-1}.
\begin{lemma}
Let $E$ be a separable Hilbert space with the inner product $(\cdot,\cdot)$. For an orthonormal basis $\{\xi_k\}_{k\in \mathbb{N}}$ in $E$, define the function $r^2_{L}:E\rightarrow \mathbb{R}^{+}$ by
\[
r^2_{L}(x)=\sum_{k\geq L+1}(x,\xi_k)^2,\quad L\in \mathbb{N}.
\]
Let $E_0$ be a total and closed under addition subset of $E$. Then a sequence $\{X_{\varepsilon}\}_{\varepsilon\in(0,1)}$ of stochastic process with trajectories in $\mathcal{D}([0,T];E)$ iff the following \textbf{Condition B } holds:
 \begin{description}
  \item[1.] $\{X_{\varepsilon}\}_{\varepsilon\in(0,1)}$ is $E_0-$weakly tight, that is, for every $h\in E_0$, $\{(X_{\varepsilon},h)\}_{\varepsilon\in(0,1)}$ is tight in $\mathcal{D}([0,T];\mathbb{R})$,
  \item[2.] For every $\eta>0$,
  \begin{eqnarray}\label{eq-65}
\lim_{L\rightarrow \infty}\lim_{\varepsilon\rightarrow 0}\mathbb{P}\Big(r^2_{L}(X_{\varepsilon}(s))>\eta\ {\rm{for\ some}}\ s\in[0,T]\Big)=0.
\end{eqnarray}
\end{description}

\end{lemma}

Consider a sequence $\{\tau_{\varepsilon},\delta_{\varepsilon}\}$ satisfying the following \textbf{Condition C}:
\begin{description}
  \item[(1)] For each $\varepsilon$, $\tau_{\varepsilon}$ ia a stopping time with respect to the natural $\sigma-$fields, and takes only finitely many values.
  \item[(2)] The constant $\delta_{\varepsilon}\in[0,T]$ satisfying $\delta_{\varepsilon}\rightarrow 0$ as $\varepsilon\rightarrow 0$.
\end{description}
Let $\{Y_{\varepsilon}\}_{\varepsilon\in(0,1)}$ be a sequence of random elements of $\mathcal{D}([0,T];\mathbb{R})$. For $f\in \mathcal{D}([0,T];\mathbb{R})$,  let $J(f)$ denote the maximum of the jump $|f(t)-f(t-)|$. We introduce the following \textbf{Condition D } on $\{Y_{\varepsilon}\}$:
\begin{description}
  \item[{(I)}] For each sequence $\{\tau_{\varepsilon},\delta_{\varepsilon}\}$ satisfying \textbf{Condition C}, $Y_{\varepsilon}(\tau_\varepsilon+\delta_\varepsilon)-Y_{\varepsilon}(\tau_\varepsilon)\rightarrow 0$ in probability, as $\varepsilon\rightarrow 0$.
\end{description}
\begin{lemma}\label{lem-11}
Assume $\{Y_{\varepsilon}\}_{\varepsilon\in(0,1)}$ satisfies \textbf{Condition D}, and either $\{Y_{\varepsilon}(0)\}$ and $J(Y_{\varepsilon})$ are tight on the line or $\{Y_{\varepsilon}(t)\}$ is tight on the line for each $t\in[0,T]$, then $\{Y_{\varepsilon}\}$ is tight in $\mathcal{D}([0,T];\mathbb{R})$.
\end{lemma}

Let $\tilde{u}^{\varepsilon}$ be defined by (\ref{equ-59}). We have
\begin{lemma}\label{lem-13}
$\{\tilde{u}^{\varepsilon}\}_{0<\varepsilon<\varepsilon_0}$ is tight in  $\mathcal{D}([0,T];D(A^{-\alpha}))$, for some $\alpha\geq0$.
\end{lemma}
\begin{proof}
Note that $\{\lambda^{\alpha}_i e_i\}_{i\in \mathbb{N}}$ is a complete orthonormal system of $D(A^{-\alpha})$. Since
\begin{eqnarray*}
\lim_{L\rightarrow \infty}\lim_{\varepsilon\rightarrow 0}\mathbb{E}\sup_{t\in[0,T]}r^2_{L}(\tilde{u}^{\varepsilon}(t))&=&\lim_{L\rightarrow \infty}\lim_{\varepsilon\rightarrow 0}\mathbb{E}\sup_{t\in[0,T]}\sum^{\infty}_{i=L+1}(\tilde{u}^{\varepsilon}(t),\lambda^{\alpha}_i e_i)^2_{D(A^{-\alpha})}\\
&=&\lim_{L\rightarrow \infty}\lim_{\varepsilon\rightarrow 0}\mathbb{E}\sup_{t\in[0,T]}\sum^{\infty}_{i=L+1}\langle A^{-\alpha}\tilde{u}^{\varepsilon}(t), e_i\rangle^2_{\mathbb{H}^0}\\
&=&\lim_{L\rightarrow \infty}\lim_{\varepsilon\rightarrow 0}\mathbb{E}\sup_{t\in[0,T]}\sum^{\infty}_{i=L+1}\frac{\langle\tilde{u}^{\varepsilon}(t), e_i\rangle^2_{\mathbb{H}^0}}{\lambda^{2\alpha}_i}\\
&\leq&\lim_{L\rightarrow \infty}\frac{\lim_{\varepsilon\rightarrow 0}\mathbb{E}\sup_{t\in[0,T]}\|\tilde{u}^{\varepsilon}(t)\|^2_{\mathbb{H}^0}}{\lambda^{2\alpha}_{L+1}}\\
&=&0,
\end{eqnarray*}
which implies (\ref{eq-65}) holds with $E=D(A^{-\alpha})$.

Choosing $E_0=D(A^{\alpha})$. We now prove $\{\tilde{u}^{\varepsilon}; 0<\varepsilon <\varepsilon_0\}\subset \mathbb{H}^0$ is $E_0-$weakly tight. Let $h\in D(A^{\alpha})$, and $\{\tau_{\varepsilon},\delta_{\varepsilon}\}$ satisfies
\textbf{Condition C}. It's easy to see $\{(\tilde{u}^{\varepsilon}(t), h)_{E}; 0<\varepsilon<\varepsilon_0\}$ is tight on the real line for each $t\in[0,T]$.

We now prove that $\{(\tilde{u}^{\varepsilon}(t), h)_{E}, 0<\varepsilon<\varepsilon_0\}$ satisfies \textbf{(I)}.
It follows from (\ref{equ-58}) that
\begin{eqnarray}\notag
\tilde{u}^{\varepsilon}(\tau_{\varepsilon}+\delta_{\varepsilon})-\tilde{u}^{\varepsilon}(\tau_{\varepsilon})
&=& \int^{\tau_{\varepsilon}+\delta_{\varepsilon}}_{\tau_{\varepsilon}} F(\tilde{u}^{\varepsilon}(s))ds+\int^{\tau_{\varepsilon}+\delta_{\varepsilon}}_{\tau_{\varepsilon}}\int_{\mathbb{Z}}\sigma(s,\tilde{u}^{\varepsilon}(s),z)(\varphi_{\varepsilon}(s,z)-1)\vartheta(dz)ds\\
\notag
&&\ +\int^{\tau_{\varepsilon}+\delta_{\varepsilon}}_{\tau_{\varepsilon}}\int_{\mathbb{Z}}\varepsilon \sigma(s,\tilde{u}^{\varepsilon}(s),z)\Big( {\eta}^{\varepsilon^{-1}\varphi_{\varepsilon}}(ds,dz)-\varepsilon^{-1}\varphi_{\varepsilon}(s,z)\vartheta(dz)ds\Big)\\
\label{equ-62}
&=:&I^{\varepsilon}_1+I^{\varepsilon}_2+I^{\varepsilon}_3,
\end{eqnarray}
It's easy to show
\begin{eqnarray}\label{equ-66}
\lim_{\varepsilon\rightarrow 0}\mathbb{E}|\langle I^{\varepsilon}_3,h \rangle_E|^2=0.
\end{eqnarray}
Referring to (3.14) in \cite{Z-Z}, and by using (\ref{equ-60})-(\ref{equ-61}), we deduce that
\begin{eqnarray}\notag
&&\lim_{\varepsilon\rightarrow 0}\mathbb{E}|\langle I^{\varepsilon}_1,h \rangle_E|\\ \notag
&\leq& \lim_{\varepsilon\rightarrow 0}\|h\|_{\mathbb{H}^0}\mathbb{E}\Big[\int^{\tau_{\varepsilon}+\delta_{\varepsilon}}_{\tau_{\varepsilon}}\|F(\tilde{u}^{\varepsilon})\|_{\mathbb{H}^0}ds\Big]\\
\notag
&\leq&\lim_{\varepsilon\rightarrow 0}\|h\|_{\mathbb{H}^0}\mathbb{E}\Big[\int^{\tau_{\varepsilon}+\delta_{\varepsilon}}_{\tau_{\varepsilon}}(\|\tilde{u}^{\varepsilon}\|^3_{\mathbb{H}^1}+\|\tilde{u}^{\varepsilon}\|_{\mathbb{H}^2})ds\Big]\\
\notag
&\leq&\lim_{\varepsilon\rightarrow 0}\|h\|_{\mathbb{H}^0}\mathbb{E}\Big[\delta_{\varepsilon}\sup_{t\in [0,T]}\|\tilde{u}^{\varepsilon}\|^3_{\mathbb{H}^1}+(\delta_{\varepsilon})^{\frac{1}{2}}\Big(\int^{\tau_{\varepsilon}+\delta_{\varepsilon}}_{\tau_{\varepsilon}}\|\tilde{u}^{\varepsilon}\|^2_{\mathbb{H}^2}ds\Big)^{\frac{1}{2}}\Big]\\
\label{equ-63}
&=&0.
\end{eqnarray}
For $I^{\varepsilon}_2$, we have
\begin{eqnarray*}\notag
&&\lim_{\varepsilon\rightarrow 0}\mathbb{E}|\langle I^{\varepsilon}_2,h \rangle_E|\\ \notag
&\leq& \|h\|_{\mathbb{H}^0}\lim_{\varepsilon\rightarrow 0}\mathbb{E}\Big[\int^{\tau_{\varepsilon}+\delta_{\varepsilon}}_{\tau_{\varepsilon}}\int_{\mathbb{Z}}\|\sigma(s,\tilde{u}^{\varepsilon},z)\|_{\mathbb{H}^0}|\varphi_{\varepsilon}(s,z)-1|\vartheta(dz)ds\Big]\\
\notag
&\leq& \|h\|_{\mathbb{H}^0}\lim_{\varepsilon\rightarrow 0}\mathbb{E}\Big[\sup_{s\in [0,T]}(1+\|\tilde{u}^{\varepsilon}\|_{\mathbb{H}^0})\int^{\tau_{\varepsilon}+\delta_{\varepsilon}}_{\tau_{\varepsilon}}\int_{\mathbb{Z}}\|\sigma(s,z)\|_{0,\mathbb{H}^0}|\varphi_{\varepsilon}(s,z)-1|\vartheta(dz)ds\Big]\\
\notag
&\leq& \|h\|_{\mathbb{H}^0}\lim_{\varepsilon\rightarrow 0}\mathbb{E}\Big[\sup_{s\in [0,T]}(1+\|\tilde{u}^{\varepsilon}\|_{\mathbb{H}^0})\sup_{g\in S^{M}}\int^{\tau_{\varepsilon}+\delta_{\varepsilon}}_{\tau_{\varepsilon}}\int_{\mathbb{Z}}\|\sigma(s,z)\|_{0,\mathbb{H}^0}|\varphi_{\varepsilon}(s,z)-1|\vartheta(dz)ds\Big].
\end{eqnarray*}
By Lemma \ref{lem-2}, we have
\begin{eqnarray}\label{equ-65}
\lim_{\varepsilon\rightarrow 0}\mathbb{E}|\langle I^{\varepsilon}_2,h \rangle_E|=0.
\end{eqnarray}
Based on (\ref{equ-66})-(\ref{equ-65}), we conclude that $\{(\tilde{u}^{\varepsilon}(t), h)_{E}; 0<\varepsilon<\varepsilon_0\}$ satisfies \textbf{Condition D}.

\end{proof}
Fix the solution $\tilde{u}^{\varepsilon}$ of (\ref{equ-58}), consider the following equation:
\begin{eqnarray}\label{equ-67}
d\tilde{Y}^{\varepsilon}(t)=-A\tilde{Y}^{\varepsilon}(t)dt
+\varepsilon\int_{\mathbb{Z}}\sigma(t,\tilde{u}^{\varepsilon}(t-),z)
\Big({\eta}^{\varepsilon^{-1}\varphi_\varepsilon}(dt,dz)-\varepsilon^{-1}\varphi_\varepsilon(t,z)\vartheta(dz)dt\Big),
\end{eqnarray}
with $\tilde{Y}^{\varepsilon}(0)=0$. Referring to Proposition 3.1 in \cite{R-Z}, (\ref{equ-67}) admits a unique solution $\tilde{Y}^{\varepsilon}(t), t\geq 0$. Moreover,
\begin{eqnarray}\label{equ-68}
\tilde{Y}^{\varepsilon}\in \mathcal{D}([0,T];\mathbb{H}^1)\cap L^2([0,T];\mathbb{H}^2),
\end{eqnarray}
and
\begin{eqnarray}\label{equ-69}
\lim_{\varepsilon\rightarrow 0}\mathbb{E}\sup_{t\in[0,T]}\|\tilde{Y}^{\varepsilon}(t)\|^2_{\mathbb{H}^1}+\mathbb{E}\int^T_0\|\tilde{Y}^{\varepsilon}(t)\|^2_{\mathbb{H}^2}dt
=0.
\end{eqnarray}

Now, we are ready to prove (ii) in \textbf{Condition A}. Recall $\mathcal{G}^{\varepsilon}(\varepsilon {\eta}^{\varepsilon^{-1}\varphi_{\varepsilon}})=\tilde{u}^{\varepsilon}$ is defined by (\ref{equ-59}).

\begin{thm}\label{thm-4}
Fix $M\in \mathbb{N}$, and let $\{\varphi_\varepsilon;0<\varepsilon<{\varepsilon}_0\}\subset \mathcal{U}^M, \varphi\in\mathcal{U}^M$ be such that $\varphi_\varepsilon$ converges in distribution to $\varphi$ as $\varepsilon\rightarrow 0$. Then
\begin{eqnarray*}
\mathcal{G}^{\varepsilon}(\varepsilon {\eta}^{\varepsilon^{-1}\varphi_\varepsilon})\ \ {\rm{converges\ in\ distribution\ to}}\ \ \mathcal{G}^{0}(\vartheta^\varphi_T)
\end{eqnarray*}
in $\mathcal{D}([0,T];\mathbb{H}^1)$.
\end{thm}

\begin{proof}
Note that $\mathcal{G}^{\varepsilon}(\varepsilon {\eta}^{\varepsilon^{-1}\varphi_{\varepsilon}})=\tilde{u}^{\varepsilon}$. From Lemma \ref{lem-8}, Lemma \ref{lem-13} and (\ref{equ-69}), we know that
\begin{description}
  \item[1.] $\{\tilde{u}^{\varepsilon}; 0<\varepsilon<\varepsilon_0\}$ is tight in $\mathcal{D}([0,T];D(A^{-\alpha}))\cap L^2([0,T];\mathbb{H}^1)$, \ for\ $\alpha\geq0 $,
  \item[2.]$\lim_{\varepsilon\rightarrow 0}\mathbb{E}\Big[\sup_{t\in[0,T]}\|\tilde{Y}^{\varepsilon}(t)\|^2_{\mathbb{H}^1}+\int^T_0\|\tilde{Y}^{\varepsilon}(t)\|^2_{\mathbb{H}^2}dt\Big]
=0,$
\end{description}
where $\tilde{Y}^{\varepsilon}$ is defined in (\ref{equ-67}). Set
\[
\Pi=\Big(\mathcal{D}([0,T];D(A^{-\alpha}))\cap L^2([0,T];\mathbb{H}^1), \mathcal{U}^M, \mathcal{D}([0,T];\mathbb{H}^1)\cap L^2([0,T];\mathbb{H}^2)\Big).
\]
Let $(\tilde{u}, \varphi, 0)$ be any limit of the tight family $\{(\tilde{u}^{\varepsilon}, \varphi_\varepsilon, \tilde{Y}^{\varepsilon});\varepsilon\in (0,{\varepsilon}_0)\}$. We will show that $\tilde{u}$
has the same law as $\mathcal{G}^{0}(\vartheta^\varphi_T)$ and $\tilde{u}^\varepsilon$ converges in distribution to $\tilde{u}$ in $\mathcal{D}([0,T];\mathbb{H}^1)$.

By the Skorokhod representative theorem, there exists a stochastic basis $(\Omega^1, \mathcal{F}^1, \{\mathcal{F}^1_t\}_{t\in [0,T]}, \mathbb{P}^1)$ and, on this basis, $\Pi-$valued random variables $(\tilde{u}^\varepsilon_1, \varphi^1_\varepsilon,\tilde{Y}^{\varepsilon}_1)$
(resp. $(\tilde{u}_1,\varphi^1,0)$) such that $(\tilde{u}^\varepsilon_1, \varphi^1_\varepsilon,\tilde{Y}^{\varepsilon}_1)$ (resp. $(\tilde{u}_1,\varphi^1,0)$) has the same law as $(\tilde{u}^{\varepsilon}, \varphi_\varepsilon, \tilde{Y}^{\varepsilon})$ (resp. $(\tilde{u},\varphi, 0)$), and $(\tilde{u}^\varepsilon_1, \varphi^1_\varepsilon,\tilde{Y}^{\varepsilon}_1)\rightarrow (\tilde{u}_1,\varphi^1,0)$ in $\Pi$, $\mathbb{P}^1-$a.s.

From the equations satisfied by $(\tilde{u}^{\varepsilon}, \varphi_\varepsilon, \tilde{Y}^{\varepsilon})$, we see that $(\tilde{u}^\varepsilon_1, \varphi^1_\varepsilon,\tilde{Y}^{\varepsilon}_1)$ satisfies the following integral equations:
\begin{eqnarray*}
\tilde{u}^\varepsilon_1(t)-\tilde{Y}^{\varepsilon}_1(t)
&=&u_0-\int^t_0A(\tilde{u}^\varepsilon_1(s)-\tilde{Y}^{\varepsilon}_1(s))ds
-\int^t_0B(\tilde{u}^\varepsilon_1(s),\tilde{u}^\varepsilon_1(s))ds\\
&&\ -\int^t_0\mathcal{P}g_N(|\tilde{u}^\varepsilon_1(s)|^2)\tilde{u}^\varepsilon_1(s)ds
+\int^t_0\int_{\mathbb{Z}}\sigma(s,\tilde{u}^\varepsilon_1(s),z)(\varphi^1_\varepsilon(s,z)-1)\vartheta(dz)ds.
\end{eqnarray*}
and
\begin{eqnarray*}
&&\mathbb{P}^1\Big(\tilde{u}^\varepsilon_1-\tilde{Y}^{\varepsilon}_1\in C([0,T];\mathbb{H}^1)\cap L^2([0,T];\mathbb{H}^2)\Big)\\
&=&\bar{\mathbb{P}}\Big(\tilde{u}^\varepsilon-\tilde{Y}^{\varepsilon}\in C([0,T];\mathbb{H}^1)\cap L^2([0,T];\mathbb{H}^2)\Big)=1.
\end{eqnarray*}

Let $\Omega^1_0$ be the subset of $\Omega^1$ such that $(\tilde{u}^\varepsilon_1, \varphi^1_\varepsilon,\tilde{Y}^{\varepsilon}_1)\rightarrow (\tilde{u}_1,\varphi^1,0)$ in $\Pi$, then $\mathbb{P}^1(\Omega^1_0)=1$ and for any fixed $\omega^1\in\Omega^1_0$,
\begin{eqnarray}\label{eq-71}
\sup_{t\in[0,T]}\|\tilde{u}^\varepsilon_1(\omega^1,t)-\tilde{u}_1(\omega^1,t)\|^2_{\mathbb{H}^1}\rightarrow 0\quad {\rm{as}} \quad \varepsilon\rightarrow 0.
\end{eqnarray}

Set $Z^\varepsilon(t)=\tilde{u}^\varepsilon_1(t)-\tilde{Y}^{\varepsilon}_1(t)$. Then, $Z^\varepsilon(\omega^1,t)\in C([0,T];\mathbb{H}^1)\cap L^2([0,T];\mathbb{H}^2) $, and $Z^\varepsilon(\omega^1,t)$
satisfies
\begin{eqnarray*}
Z^\varepsilon(t)&=&u_0-\int^t_0AZ^\varepsilon(s)ds
-\int^t_0B(Z^\varepsilon(s)+\tilde{Y}^{\varepsilon}_1(s),Z^\varepsilon(s)+\tilde{Y}^{\varepsilon}_1(s))ds\\
&&\ -\int^t_0\mathcal{P}g_N(|Z^\varepsilon(s)+\tilde{Y}^{\varepsilon}_1(s)|^2)(Z^\varepsilon(s)+\tilde{Y}^{\varepsilon}_1(s))ds\\
&&\ +\int^t_0\int_{\mathbb{Z}}\sigma(s,Z^\varepsilon(s)+\tilde{Y}^{\varepsilon}_1(s),z)(\varphi^1_\varepsilon(s,z)-1)\vartheta(dz)ds,
\end{eqnarray*}
Define $\hat{u}(t)$ be the solution of
\begin{eqnarray*}
\hat{u}(t)&=&u_0-\int^t_0A\hat{u}(s)ds-\int^t_0B(\hat{u}(s),\hat{u}(s))ds-\int^t_0\mathcal{P}g_N(|\hat{u}(s)|^2)\hat{u}(s)ds\\
&&\ +\int^t_0\int_{\mathbb{Z}}\sigma(s,\hat{u}(s),z)(\varphi^1(s,z)-1)\vartheta(dz)ds.
\end{eqnarray*}

Since
\[
\lim_{\varepsilon\rightarrow 0}\Big[\sup_{t\in[0,T]}\|\tilde{Y}^{\varepsilon}(\omega^1,t)\|^2_{\mathbb{H}^1}+\int^T_0\|\tilde{Y}^{\varepsilon}(\omega^1,t)\|^2_{\mathbb{H}^2}dt\Big]
=0,
\]
we have
\begin{eqnarray}\notag
&&\lim_{\varepsilon\rightarrow 0}\sup_{t\in[0,T]}\|\tilde{u}^{\varepsilon}_1(\omega^1,t)-\hat{u}(\omega^1,t)\|^2_{\mathbb{H}^1}\\ \notag
&\leq& \lim_{\varepsilon\rightarrow 0}\sup_{t\in[0,T]}[\|Z^{\varepsilon}(\omega^1,t)-\hat{u}(\omega^1,t)\|^2_{\mathbb{H}^1}+\|\tilde{Y}^{\varepsilon}_1(\omega^1,t)\|^2_{\mathbb{H}^1}]
 \\ \label{eq-72}
&\leq&\lim_{\varepsilon\rightarrow 0}\sup_{t\in[0,T]}\|Z^{\varepsilon}(\omega^1,t)-\hat{u}(\omega^1,t)\|^2_{\mathbb{H}^1}.
\end{eqnarray}
Using the similar argument as Proposition \ref{prp-1}, we obtain
\begin{eqnarray}\label{eq-73}
\lim_{\varepsilon\rightarrow 0}\sup_{t\in[0,T]}\|Z^{\varepsilon}(\omega^1,t)-\hat{u}(\omega^1,t)\|^2_{\mathbb{H}^1}=0.
\end{eqnarray}
Hence, combining (\ref{eq-72}) and (\ref{eq-73}), we deduce that
 \begin{eqnarray}\label{eq-72-1}
\lim_{\varepsilon\rightarrow 0}\sup_{t\in[0,T]}\|\tilde{u}^{\varepsilon}_1(\omega^1,t)-\hat{u}(\omega^1,t)\|^2_{\mathbb{H}^1}=0,
\end{eqnarray}
which imply that
 $\tilde{u}_1=\hat{u}=\mathcal{G}^0(\vartheta^{\varphi^1})$, and $\tilde{u}$ has the same law as $\mathcal{G}^0(\vartheta^{\varphi})$.
Since $\tilde{u}^{\varepsilon}=\tilde{u}^{\varepsilon}_1$ in law, we deduce from (\ref{eq-72-1}) that $\tilde{u}^{\varepsilon}$ converges to $\mathcal{G}^0(\vartheta^{\varphi})$ . We complete the proof.
\end{proof}

\vskip 0.2cm
\noindent{\bf  Acknowledgements}\
This work was supported by National Natural Science Foundation of China (NSFC) (No. 11431014, 11801032),
Key Laboratory of Random Complex Structures and Data Science, Academy of Mathematics and Systems Science, Chinese Academy of Sciences (No. 2008DP173182).
\def\refname{ References}


\begin{thebibliography}{2}

\bibitem {A-B-W} S. Albeverio, Z. Brze\'{z}niak, J. Wu: \emph{Existence of global solutions and invariant measures for stochastic differential equations driven by Poisson type noise with non-Lipschitz coefficients.}  J. Math. Anal. Appl. 371, no. 1, 309-322 (2010).
\bibitem{A-1} D. Aldous: \emph{Stopping times and tightness.} Ann. Probab. 6 335-340 (1978).
\bibitem{BHZ} Z. Brze\'{z}niak, E. Hausenblas, J. Zhu: \emph{2D Navier-Stokes equations driven by jump noise.} Nonlinear Anal. 79, 122-139 (2013).
\bibitem{BLZ} Z. Brze\'{z}niak, W. Liu, J. Zhu: \emph{Strong solutions for SPDE with locally
monotone coefficients driven by L\'{e}vy noise}. Nonlinear Anal. Real World Appl. 17, 283-310 (2014).

\bibitem{B-C-D} A. Budhiraja, J. Chen, P. Dupuis: \emph{Large deviations for stochastic partial differential equations driven by a Poisson random measure.}
Stochastic Process. Appl. 123, no. 2, 523-560 (2013).

\bibitem{B-D-M} A. Budhiraja, P. Dupuis and V. Maroulas : \emph{Variational representations for continuous time processes}. Ann. Inst. Henri Poincar\'{e} Probab. Stat. 47, no. 3, 725-747 (2011).









\bibitem{FG95}  F. Flandoli, D. Gatarek: \emph{Martingale and stationary solutions for stochastic Navier-Stokes equations}. Probab. Theory Related Fields 102, no. 3, 367-391 (1995).
\bibitem{I} A. Ichikawa: \emph{Some inequalities for martingales and stochastic convolutions}. Stoch. Anal. Appl. 4 329-339 (1986).


\bibitem {I-W} N. Ikeda, Nobuyuki, S. Watanabe: \emph{Stochastic differential equations and diffusion processes.} Second edition. North-Holland Mathematical Library, 24.
    \bibitem {J} A. Jakubowski: \emph{On the Skorokhod topology.}
Ann. Inst. H. Poincar\'{e} Probab. Statist. 22, no. 3, 263-285 (1986).
\bibitem{Lions} P.L. Lions: \emph{Mathematical Topics in Fluid Mechanics.} Volume1, in: Incompressible Models. Oxford Lect. Series in Math and its App., vol.3, 1996.

\bibitem{M-S} J.L. Menaldi, S.S. Sritharan: \emph{Stochastic 2D Navier-Stokes equation.} Appl. Math. Optim. 46, 31-53 (2002).
















%

%
%
%
%
%

%
\bibitem{M-R} R. Mikulevicius, B.L. Rozovskii: \emph{ Global $L^2$ solution of stochastic Navier-Stokes equations.} Ann. Probab. 33 (1) 137-176 (2005).
\bibitem{R-TS} M. R\"{o}ckner, T. Zhang: \emph{Stochastic evolution equations of jump type: existence, uniqueness and large deviation principles.} Potential Anal. 26, no. 3, 255-279 (2007).
\bibitem{R-Z} M. R\"{o}ckner, T. Zhang: \emph{Stochastic 3D tamed Navier-Stokes equations: existence, uniqueness and small time large deviation principles.} J. Differential Equations 252, no. 1, 716-744 (2012).

\bibitem{R-Z-Z} M. R\"{o}ckner, T. Zhang, X. Zhang: \emph{Large Deviations for Stochastic Tamed 3D Navier-Stokes Equations.} Appl. Math. Optim. 61, no. 2, 267-285 (2010).
\bibitem{R-Z-0} M. R\"{o}ckner, X. Zhang: \emph{Tamed 3D Navier-Stokes equation: existence, uniqueness and regularity }. Infin. Dimens. Anal. Quantum Probab. Relat. Top. 12, no. 4, 525-549 (2009).
\bibitem{R-Z-0-0} M. R\"{o}ckner, X. Zhang: \emph{Stochastic tamed 3D Navier-Stokes equations: existence, uniqueness and ergodicity.} Probab. Theory Related Fields 145, no. 1-2, 211-267 (2009).




\bibitem{S-S} S.S. Sritharan, P. Sundar: \emph{Large deviations for the two-dimensional Navier-Stokes equations with multiplicative noise}. Stochastic Process. Appl. 116, 1636-1659 (2006).

\bibitem{S-Z} A. \'{S}wiech,  J. Zabczyk: \emph{
Large deviations for stochastic PDE with L\'{e}vy noise.} J. Funct. Anal. 260, no. 3, 674-723 (2011).
\bibitem {S-Z-Z} S. Shang, J. Zhai, T. Zhang: \emph{Strong solutions for a stochastic model of two-dimensional second grade fluids driven by L\'{e}vy noise.} J. Math. Anal. Appl. 471, no. 1-2, 126-146 (2019).


\bibitem{Temam-1} R. Temam: \emph{Navier-Stokes equations and nonlinear functional analysis}. Second edition. CBMS-NSF Regional Conference Series in Applied Mathematics, 66. Society for Industrial and Applied Mathematics (SIAM), Philadelphia, PA, 1995.


\bibitem{Y-Z-Z} X. Yang, J. Zhai, T. Zhang: \emph{Large deviations for SPDEs of jump type.}
Stoch. Dyn. 15, no. 4, 1550026, 30 pp (2015).
\bibitem{Z-Z} J. Zhai, T. Zhang: \emph{Large deviations for 2D stochastic Navier-Stokes equations driven by multiplicative L\'{e}vy noises. }
Bernoulli 21, no. 4, 2351-2392 (2015).
\end{thebibliography}
\end{document}